\documentclass{article}
\usepackage{amsmath,amssymb,amsfonts,fullpage,amsthm,color,mathrsfs,xcolor,tikz}

\newtheorem{theorem}{Theorem}[section]
\newtheorem{prop}[theorem]{Proposition}
\newtheorem{lemma}[theorem]{Lemma}
\newtheorem{corollary}[theorem]{Corollary}
\theoremstyle{remark}
\newtheorem{remark}{Remark}
\newtheorem{example}{Example}
\newtheorem{definition}[theorem]{Definition}

\newcommand\scalemath[2]{\scalebox{#1}{\mbox{\ensuremath{\displaystyle #2}}}}

\usepackage{authblk}

\begin{document}

\title{Orthogonal polynomial duality and unitary symmetries  of multi--species ASEP$(q,\boldsymbol{\theta})$  and  higher--spin vertex models via $^*$--bialgebra structure of higher rank quantum groups}

\author[1]{Chiara Franceschini}
\author[2]{Jeffrey Kuan}
\author[2]{Zhengye Zhou}
\affil[1]{University of
 Modena and Reggio Emilia, FIM, via G. Campi 213/b, 41125 Modena, Italy}
\affil[2]{Texas A\&M University, Department of Mathematics, College Station TX, 77843-3368.}

\date{}
\maketitle 

\abstract{We propose a  general method to produce orthogonal polynomial dualities from the $^*$--bialgebra structure of Drinfeld--Jimbo quantum groups. The $^*$--structure allows for the construction of certain \textit{unitary} symmetries, which imply the orthogonality of the duality functions. In the case of the quantum group $\mathcal{U}_q(\mathfrak{gl}_{n+1})$, the result is a nested multivariate $q$--Krawtchouk duality for the $n$--species ASEP$(q,\boldsymbol{\theta}) $. The method also applies to  other quantized simple Lie algebras  and to stochastic vertex models. 

As a probabilistic application of the duality relation found, we provide the explicit formula of the $q-$shifted factorial moments (namely the $q$-analogue of the Pochhammer symbol) for the two--species $q$--TAZRP (totally asymmetric zero range process).}

\section{Introduction}
Over the last several years, there has been a flurry of interest in so--called multi--species models, also known as colored or multi--class models \cite{BS,BS2,BS3,BGW,Gal,Kuan-IMRN,KuanCMP,KMMO,Take15,ZZ}. In the context of interacting particle systems, these were first introduced in 1976 by \cite{Ligg76}. More recently, algebraic methods have been developed to find Markov duality for these multi--species models \cite{BS,BS2,BS3,Kuan-IMRN,KuanCMP}, generalizing Sch\"{u}tz's method in \cite{Sch97}. In these papers, the duality functions are ``triangular'' duality functions, in the sense that when viewed as matrices they are lower triangular matrices. However, a priori there is not necessarily a reason to believe that triangular duality functions can be used for probabilistic applications, although at times they are \cite{CST,YLKPZ} in the context of KPZ universality class and for a single species only.

In some recent works by \cite{FGG, CFGGR, Gro19, CFG21}, algebraic methods have been developed to find \textbf{orthogonal polynomial} duality functions for interacting particle systems, both in the symmetric and asymmetric cases. Regarding applications we mention \cite{FGS} where duality relations are used to study the hydrodynamic limit of the inclusion process. In \cite{FRS} orthogonal duality functions are used to infer information and partially characterize the non-equilibrium stationary measure of an open system of interacting particles. In \cite{ACR18} and \cite{ACR21} the authors first present a quantitative Boltzmann–Gibbs principle used to study macroscopic fluctuations and then they study scaling limits of higher order fluctuation fields. In all the above it was crucial to have duality functions which are polynomials orthogonal with respect to the reversible measure of the Markov process considered. These orthogonal polynomial duality functions are advantageous, from a probabilistic application, because not only one can take advantage of the duality relation but also of the orthogonal relation and the fact that observables of interest can be written in terms of such orthogonal duality functions.  In a sense, the orthogonality guarantees that the duality function is ``useful'', as long as the dual process can be analyzed.

 In general, having a duality relation for a particle system yields a dual process with finitely many and, in particular, very few particles. Via the analysis of these few dual particles one can gather information on  some relevant quantities for the initial process. These quantities (e.g. truncated correlations for symmetric particle systems or $q^{-2}$ exponential moments for asymmetric ones) are suggested by the type of duality function one has.

In this paper, we introduce a general method to produce orthogonal polynomial dualities, using the $^*$--structure of bialgebras. The $^*$--structure allows for the construction of unitary symmetries, from which the orthogonality of the duality functions will follow. This method is used explicitly in the case of the Drinfeld--Jimbo quantum group $\mathcal{U}_q(\mathfrak{gl}_{n+1})$, which is an underlying algebraic structure of the  $n$--species ASEP$(q,\boldsymbol{\theta})$. As far as the authors know, this is the first probabilistic application of $^*$--bialgebras. This method is expected to work for other simple Lie algebras \cite{REU2022}. Note that this method is not the same as the ``inner product method'' of \cite{CFGGR} used to construct new duality functions starting from known ones, which turns out to not generalize to higher--rank Lie algebras (see Remark \ref{InnerProductRemark}). 
The idea, as in \cite{CFGGR}, relies on the identification of the unitary symmetry which allows to produce duality functions in terms of orthogonal polynomials. However, in our case, we identify such unitary symmetry without relying on the knowledge of the expression of the triangular self-duality function.

The $n$--species ASEP$(q,\boldsymbol{\theta})$ allows for $\theta^x$ particles to occupy  site $x$, and has an asymmetry parameter of $q^{\boldsymbol\theta}$. In the $\boldsymbol\theta\rightarrow \infty$ limit the process converges to a $n$--species $q$--TAZRP (totally asymmetric zero range process), in which arbitrarily many particles may occupy a site and all jumps occur in one direction. Applying a charge--parity symmetry, the duality function converges to a nontrivial $q$--Pochhammer. Furthermore, in this limit, the reversible measures converge to a deterministic initial condition, since the totally asymmetric process is not reversible. This suggests, a priori, that the duality can be used to find explicit formulas for the $n$--species $q$--TAZRP. Indeed,  for the $2$--species $q$--TAZRP, we use the  polynomial duality to give exact formulas for the $q$--shifted factorial moments that can be used to yield the rate of decorrelation at two points. This uses that the process is integrable, in the sense that its generator satisfies the Yang--Baxter equation, and has explicit transition probabilities expressed in terms of contour integral formulas. Indeed, an analysis of the dual (finite--particle) process involves explicit contour integral formulas for certain observables of the multi--species $q$--TAZRP, very recently found by the second author in \cite{Kua21}.

 The multi--species $q$--TAZRP is itself a special case of the multi--species $q$--Hahn TAZRP, which is itself a special case of the multi--species higher spin stochastic vertex model. Using the general algebraic setup of \cite{KuanCMP}, the orthogonality polynomial duality function for the ASEP$(q,\boldsymbol\theta)$ is also a duality function for the vertex model, after applying the charge--parity symmetry (CP symmetry). Figure \ref{Deg} shows all the processes discussed in the paper. The $\boldsymbol\theta\rightarrow \infty$ limit corresponds to the $\mu\rightarrow 0$ limit, so thus an explicit proof is additionally needed for the $q$--Hahn TAZRP.

\begin{figure}

\includegraphics[width=1\textwidth]{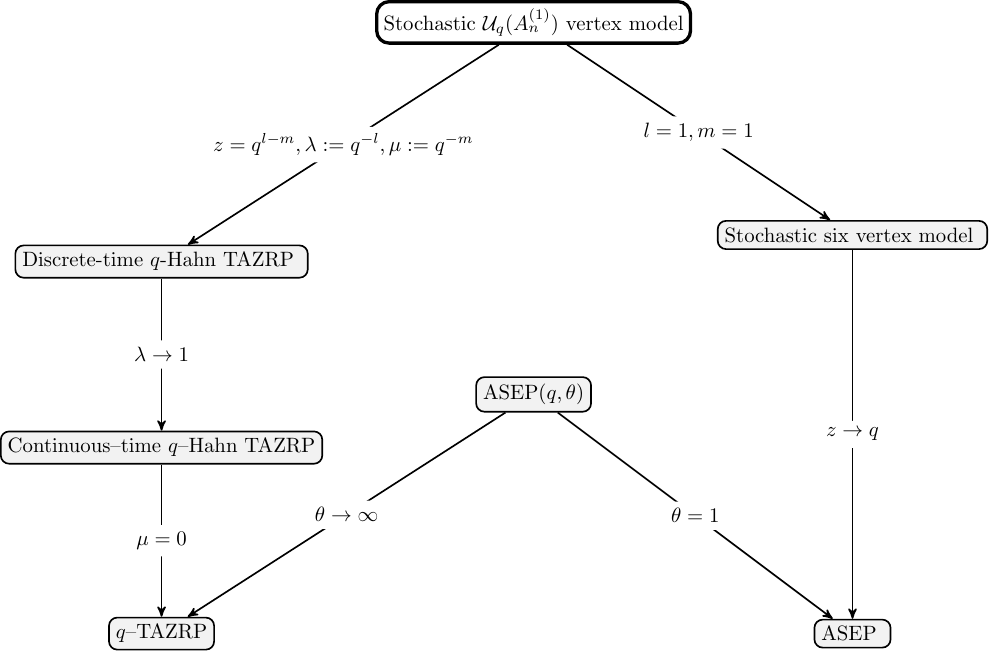}
\caption{The various degenerations and limits explicitly mentioned in this paper. See  \cite{KuanCMP} for some more degenerations not mentioned in this paper. }
\label{Deg}
\end{figure}

The outline of the paper is as follows. Section \ref{Definitions} defines notation, formulas, and we introduce all the models in a close boundary setting. Section \ref{MR} states the main results and it ends with an application of one of the duality relations shown. The proofs of all results are in section \ref{Pfs}.  Note that the method of the proof generalizes to other Lie algebras, as discussed in Remark \ref{REURem}.
\section{Definitions}\label{Definitions}

\subsection{$q$--notation}
For any $q$ which is not a root of unity, define the $q$--deformed integer, factorial, binomial, and multinomial by 
\begin{equation*}
[n]_q = \frac{q^n-q^{-n}}{q-q^{-1}}, \quad [n]_q^! = [1]_q \cdots [n]_q, 
\end{equation*}
\begin{equation*}
    \quad \binom{n}{k}_q = \frac{[n]_q^!}{[k]_q^! [n-k]_q^!}, \quad \binom{n}{k_1,\ldots,k_l}_q = \frac{[n]_q^!}{[k_1]_q^! \cdots [k_l]_q^!},    
\end{equation*}
and the $q$--Pochhammer for integer $n>0$:
\begin{equation*}
(a ; q)_{n}=(1-a)(1-a q) \cdots\left(1-a q^{n-1}\right), \quad(a ; q)_{0}:=1.
\end{equation*}
\begin{equation*}
    (a;q)_{-n} = \frac{1}{\prod_{k=1}^n (1-aq^{-k})} \;, \qquad \qquad a \neq q,q^2,q^3, \ldots, q^n.
\end{equation*}

A second $q$--deformed integer and factorial that will be used is
\begin{equation*}
\{n\}_{q^{2}}=\frac{1-q^{2 n}}{1-q^{2}}, \quad \{n\}_{q^2}^! = \{1\}_{q^2} \cdots \{n\}_{q^2}.
\end{equation*}
The two $q$--deformations are related by 
\begin{equation}\label{Relate}
q^{n-1}[n]_q = \{n\}_{q^2} , \quad q^{n(n-1)/2}[n]_q^! = \{n\}_{q^2}^!.
\end{equation}
The first $q$--deformation with the square brackets is more natural in an algebraic setting, because it reflects the $q\mapsto q^{-1}$ symmetry, while the second $q$--deformation with the curly brackets is more natural in a probabilistic setting, because $q^2$ is the asymmetry parameter in the asymmetric simple exclusion process.
Throughout the paper we will assume $0<q<1$. Also, define the $q$--exponentials
\begin{equation*}
e_{q}(z)=\sum_{n=0}^{\infty} \frac{z^{n}}{(q ; q)_{n}}=\frac{1}{(z ; q)_{\infty}}, \quad\mid z\mid<1,    
\end{equation*} 
and
\begin{equation*}
   \mathcal{E}_{q}(z)=\sum_{n=0}^{\infty} \frac{q^{\binom{n}{2}} z^{n}}{(q ; q)_{n}}=(-z ; q)_{\infty}. 
\end{equation*}
The more common notation for $\mathcal{E}_q$ is $E_q$, but in this paper the $\mathcal{E}$ will be used to avoid notation conflict with the quantum group generators $E_{ij}$. The $q$--exponentials satisfy two $q$--Baker--Campbell--Hausdorff equations, the first being
\begin{equation*}
\mathcal{E}_{q}(\lambda X) Y e_{q}\left(-\lambda q^{\alpha} X\right)=\sum_{n=0}^{\infty} \frac{\lambda^{n}}{(q ; q)_{n}}[X, Y]_{n},
\end{equation*}
where
\begin{equation*}
[X, Y]_{0}=Y, \quad[X, Y]_{n+1}=q^{n} X[X, Y]_{n}-q^{\alpha}[X, Y]_{n} X, \quad n=1,2, \ldots,
\end{equation*}
and the second being
\begin{equation}\label{SecondEq}
e_{q}(\lambda X) Y \mathcal{E}_{q}(-\lambda X)=\sum_{n=0}^{\infty} \frac{\lambda^{n}}{(q ; q)_{n}}[X, Y]_{n}^{\prime},
\end{equation}
where
$$
\qquad[X, Y]_{0}^{\prime}=Y, \quad[X, Y]_{n+1}^{\prime}=X[X, Y]_{n}^{\prime}-q^{n}[X, Y]_{n}^{\prime} X, \quad n=0,1,2, \ldots.
$$

If $xy=q^2yx$ then 
\begin{equation}\label{EFactor}
\mathcal{E}_{q^{2}}(x+y)=\mathcal{E}_{q^{2}}(x) \cdot \mathcal{E}_{q^{2}}(y) \quad \text { and } \quad e_{q^{2}}(x+y)=e_{q^{2}}(y) \cdot e_{q^{2}}(x).
\end{equation}

The $e_{q^2}$ and $\mathcal{E}_{q^2}$ are mutual inverses of each other, in the sense that
\begin{equation}\label{Inverses}
e_{q^2}(z)\mathcal{E}_{q^2}(-z)=1.
\end{equation}

\subsection{Orthogonal Polynomials}
Recall the $q$--hypergeometric series
\begin{equation}
{ }_{p} \varphi_{p-1}\left(\begin{array}{c}a_1, \ldots , a_p \\ b_1, \ldots , b_{p-1}\end{array} ; q, z\right):=\sum_{k=0}^{\infty} \frac{(a_1 ; q)_{k} \dots (a_p ; q)_{k}}{(b_1 ; q)_{k}  \dots (b_{p-1} ; q)_{k}}  \frac{z^{k}}{(q ; q)_{k}} \;,
\end{equation}
in particular, for $p=2$, we have
\begin{equation}
{ }_{2} \varphi_{1}\left(\begin{array}{c}a_1, a_2 \\ b_1\end{array} ; q, z\right)=\sum_{k=0}^{\infty} \frac{(a_1 ; q)_{k}(a_2 ; q)_{k}}{(b_1 ; q)_{k}} \frac{z^{k}}{(q ; q)_{k}} \;.
\end{equation}
Note that for
$a = q^{-n}$ and $n \in \mathbb{N}$, the series is a polynomial of degree $n$ with respect to the variable $b$.
Now define the $q$--Krawtchouk polynomial as in Chapter 14.14 of \cite{KLS}
\begin{equation*}
K_{n}\left(q^{-x} ; p, c ; q\right):={ }_{2} \varphi_{1}\left(\begin{array}{c}q^{-x}, q^{-n} \\ q^{-c}\end{array} ; q, p q^{n+1}\right).
\end{equation*}
These are orthogonal in the following sense: when $pq^c>1$ and $c \in \mathbb{N}$, 
\begin{multline*}
\sum_{x=0}^{c} \frac{(p q ; q)_{c-x}(-1)^{c-x}}{(q ; q)_{x}(q ; q)_{c-x}} q^{\left(\frac{x}{2}\right)} \cdot K_{m}\left(q^{-x} ; p, c ; q\right) \cdot K_{n}\left(q^{-x} ; p, c ; q\right)\\
=\frac{(-1)^{n} p^{c}(q ; q)_{c-n}(q ; q)_{n}(p q ; q)_{n}}{\left((q ; q)_{c}\right)^{2}} \cdot 
q^{\binom{c+1}{  2} -\binom{n+1}{  2} +cn}
\cdot
\delta_{m, n}.
\end{multline*}

\subsection{Markov Duality} \label{markovduality}

\begin{definition}
Two Markov processes $\xi(t)$ and $\eta(t)$ on states spaces $\mathfrak{X}$ and $\mathfrak{Y}$, respectively, are \textit{dual} with respect to the function $D(\xi,\eta)$ on $\mathfrak{X} \times \mathfrak{Y}$ if
$$
\mathbb{E}_\xi[D(\xi(t),\eta)] = \mathbb{E}_\eta[D(\xi,\eta(t))] \text{ for all } t\geq 0.
$$
If, in addition, $\mathfrak{X}=\mathfrak{Y}$ and $\xi(t)=\eta(t)$ then $\xi(t)$ is \textit{self--dual} with respect to the function $D(\xi,\eta)$.
\end{definition}

If $\mathfrak{X}$ and $\mathfrak{Y}$ are discrete, an equivalent definition of duality is the intertwining relation
$$
L_\xi^T D = DL_\eta,
$$
where $L_\xi$ is the generator of $\xi(t)$ viewed as a matrix with rows and columns indexed by $\mathfrak{X}$, $L_\eta$ is the generator of $\eta(t)$ with rows and columns indexed by $\mathfrak{Y}$, and $D$ is the duality function viewed as a matrix with rows indexed by $\mathfrak{X}$ and columns indexed by $\mathfrak{Y}$. The superscript $ ^T$ indicates matrix transpose. Note that this is using the convention in mathematical physics, rather than probability, that a stochastic matrix has columns, rather than rows, summing to $1$.

Any reversible Markov process with countable state space  is self--dual with respect to the so--called ``cheap'' duality function given by the inverses of the reversible measures. More precisely, if $\xi(t)$ is reversible, then detailed balance reads
$$
L_\xi^T = V^{-1}L_\xi V,
$$
where $V$ is the diagonal matrix with entries $V(\xi,\eta) = \delta_{\xi,\eta} \pi(\xi)$, and $\pi(\xi)$ is a reversible measure. Then defining $D^{\text{ch}}=V^{-1}$, it is immediate that
$$
L_\xi^TD^{\text{ch}} = D^{\text{ch}}L_\xi.
$$

\begin{definition} \label{symm}
If $\xi(t)$ is a Markov process with generator $L_\xi$, then a \textit{symmetry} of $\xi(t)$ is an operator $S$ such that $L_\xi S=SL_\xi$.
\end{definition}
In the context of self-duality there is a one-to-one correspondence between a self-duality function and the symmetry of the Markov process.  Indeed, if $\xi(t)$ is a Markov process which is self-dual via $D^{\text{ch}}$ and with non-trivial  symmetry $S$, then the new function $D^{\text{new}} = SD^{\text{ch}}$ is a different self-duality function. To see this, note that 
\begin{align*}
L_\xi^TD^{\text{new}} = L_\xi^T SD^{\text{ch}} = SL_\xi^T D^{\text{ch}} = SD^{\text{ch}}L_\xi =D^{\text{new}}L_\xi.
\end{align*}
In this way, applying symmetries to the ``cheap'' duality function $D^{\text{ch}}$ produces more interesting, non--trivial self--duality functions. 

\subsection{The $n$--species ASEP$(q,\boldsymbol{\theta})$ and $q$--Hahn Boson} \label{sec2.4}
Let $\Lambda_L := \left \{ 1, \ldots, L \right \} $ be a finite one dimensional chain where all our processes evolves and let $\boldsymbol{\theta}=(\theta^x)_{ x \in \Lambda_L}$ be a family of positive integers. The $\theta^x$ will indicate the maximum number of particles that may occupy the site $x$. There are $n$ different species of particles, labeled from $0$ to $n-1$. A particle configuration will be denoted in bold by
$$
\boldsymbol{\xi} = (\xi^x_{i})_{1 \leq x \leq L,0\leq i \leq n},
$$
where $\xi^x_{i}$ denotes the number of $i^{th}$ species particles at site $x$, with $i=n$ denoting a hole. Note that $\xi^x_{0} + \ldots + \xi^x_{n} = \theta^x$ for all $x\in \Lambda_L$.
For any lattice site $x$, the notation $\boldsymbol\xi^x$ will denote
$$
\boldsymbol\xi^x = (\xi^x_0, \ldots,\xi^x_n),
$$
which describes the particles located at site $x$ in particle configuration $\boldsymbol{\xi}$.
While $\boldsymbol\xi_i$ will denote
$$
\boldsymbol\xi_i = (\xi^1_i, \ldots,\xi^L_i),
$$
which describes the particles of species $i$  in particle configuration $\boldsymbol{\xi}$.
In terms of jump rates, species with smaller labels are considered to be heavier or to have priority. The generator is explicitly given by 
\begin{align*}
 & \mathcal{L}_{x, x+1} f(\boldsymbol\xi)  \\ & = \sum_{0 \leq k<l \leq n}\left( q^{-1+2\xi^x_{[0,k-1]}}\left\{\xi_{k}^{x}\right\}_{q^{2}} \cdot q^{2\xi_{[l+1,n]}^{x+1}}\left\{\xi_{l}^{x+1}\right\}_{q^{2}}\left(f\left(\boldsymbol\xi_{k \leftrightarrow l}^{x \leftrightarrow x+1}\right)-f(\boldsymbol\xi)\right)\right.\\ &\left.+ q^{1+2\xi_{[0,l-1]}^{x}}\left\{\xi_{l}^{x}\right\}_{q^{2}} \cdot q^{2\xi_{[k+1,n]}^{x+1}}\left\{\xi_{k}^{x+1}\right\}_{q^{2}}\left(f\left(\boldsymbol\xi_{l \leftrightarrow k}^{x \leftrightarrow x+1}\right)-f(\boldsymbol\xi)\right)\right), 
\end{align*}
where where $\xi_{[m, k]}^{x}=\delta_{k\ge m}\sum_{j=m}^k\xi_j^x$ and 
 $\boldsymbol\xi^{x\leftrightarrow x+1}_{i \leftrightarrow j}$ denotes the particle configuration
obtained by switching a species $i$ particle at lattice site $x$ with a $j^{th}$ species particle
at lattice site $x +$1 (assuming that such a particle configuration exists). If this particle configuration does not exist, then set $\boldsymbol\xi^{x\leftrightarrow x +1}_{i \leftrightarrow j} = \boldsymbol\xi$.
A verbal description of the model can be found in \cite{Kuan-IMRN}, which introduced the model. A family of reversible measures is also given by 
\begin{equation*}
\mu_{\boldsymbol k}^{n}(\boldsymbol{\xi})=1_{\left\{N\left(\boldsymbol\xi_{i}\right)=k_{i}, 0 \leq i \leq n\right\}} \prod_{x \in \Lambda_L} \prod_{i=0}^{n} \frac{1}{\left[\xi_{i}^{x}\right]_{q}^{!}} q^{\left(\xi_{i}^{x}\right)^{2} / 2} \prod_{y<x} \prod_{i=0}^{n-1} q^{-2 \xi_{[0, i]}^{x} \xi_{i+1}^{y}} \;,
\end{equation*}
where $N\left(\boldsymbol\xi_{i}\right)$ is the total number of particles of species $i$ in $\boldsymbol\xi$.
When taking the limit where all $\theta^x\rightarrow \infty$, the multi--species ASEP$(q,\boldsymbol\theta)$ converges to a multi--species $q$-TAZRP, which was introduced in \cite{Take15}, generalizing the single--species $q$--TAZRP in \cite{SW98}. The jump rates are much simpler to describe, and are pictorially represented in Figure \ref{qTAZRPFig}.

\begin{figure}
\centering
    \includegraphics[height=6.5cm]{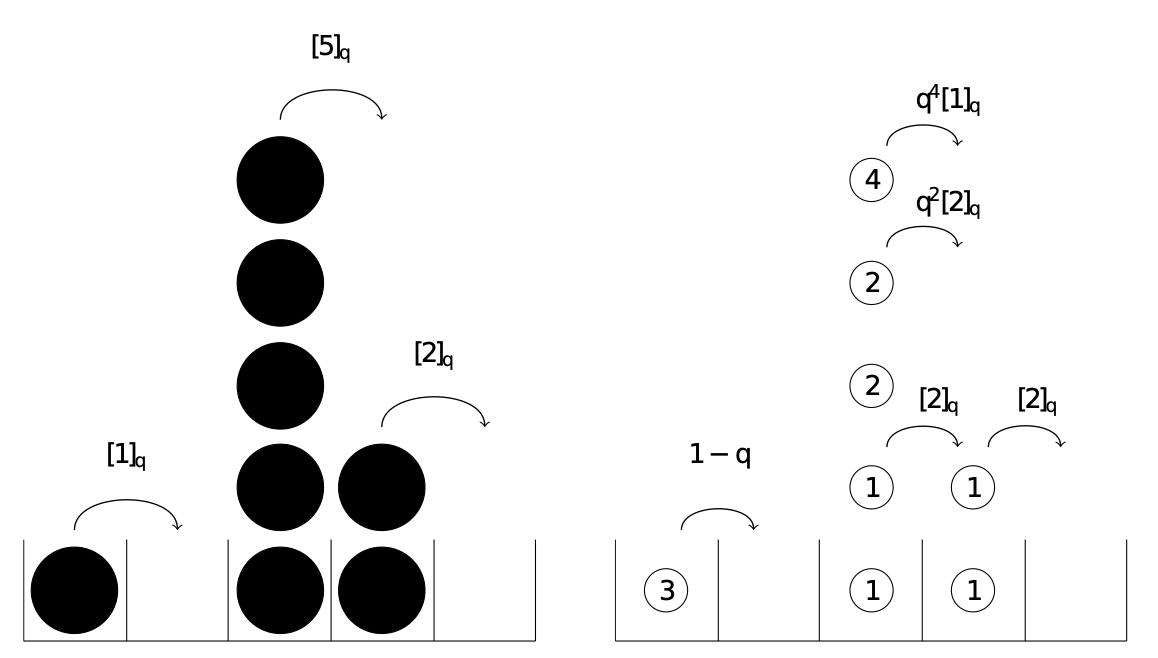}
    \caption{The jump rates for the multi--species $q$--TAZRP.}
    \label{qTAZRPFig}
\end{figure}

 Following \cite{KMMO}, Let $\boldsymbol\gamma , \boldsymbol\beta\in \mathbb{N}^n$ denote particle configurations at a single site,  define
$$
\Phi_{q}(\boldsymbol\gamma \mid \boldsymbol\beta ; \lambda, \mu)=q^{\chi_{\beta, \gamma}}\left(\frac{\mu}{\lambda}\right)^{\mid\gamma\mid} \frac{(\lambda ; q)_{\mid\gamma\mid}\left(\frac{\mu}{\lambda} ; q\right)_{\mid\beta\mid-\mid\gamma\mid}}{(\mu ; q)_{\mid\beta\mid}} \prod_{i=0}^{n} \binom{\beta_i}{\gamma_i}_q,
$$
where 
$$
\chi_{\beta, \gamma}=\sum_{0 \leq i<j \leq n}\left(\beta_{i}-\gamma_{i}\right) \gamma_{j},\qquad \mid\gamma\mid=\sum_{i=0}^n \gamma_i.
$$
A discrete--time multi--species $q$--Hahn TAZRP can be defined as follows:
given a particle configuration $\left(\boldsymbol\eta^x\right)$, the probability measure after the (discrete-time) update is described by
\begin{align*}
\mathbb{P}( \boldsymbol{k} & \text { particles at site } x) \\ & =\sum_{\substack{\boldsymbol\gamma^{x-1} \geq 0 \\ \boldsymbol\gamma^x \geq 0}} 1_{\left\{\boldsymbol\eta^x-\boldsymbol\gamma^x+\boldsymbol\gamma^{x-1}=\boldsymbol{k}\right\}} \Phi_q\left(\boldsymbol\gamma^x \mid \boldsymbol\eta^x\right) \Phi_q\left(\boldsymbol\gamma^{x-1} \mid \boldsymbol\eta^{x-1}\right) \text {. }
\end{align*}
 In this case, we will say that the process evolves with total asymmetry to the right, and the process evolves with total asymmetry to the left  if the probability measure after the update is given by
\begin{align*}
\mathbb{P}( \boldsymbol{k} & \text { particles at site } x) \\ & = \sum_{\substack{\boldsymbol\gamma^{x+1} \geq 0 \\ \boldsymbol\gamma^x \geq 0}} 1_{\left\{\boldsymbol\eta^x-\boldsymbol\gamma^x+\boldsymbol\gamma^{x-1}=\boldsymbol{k}\right\}} \Phi_q\left(\boldsymbol\gamma^x \mid \boldsymbol\eta^x\right) \Phi_q\left(\boldsymbol\gamma^{x+1} \mid \boldsymbol\eta^{x+1}\right) \text {. }
\end{align*}
 By taking the derivative at $\lambda=1$, these can  be used to define a continuous--time multi--species $q$--Hahn TAZRP, which then specializes to the multi--species $q$--TAZRP. The single--species $q$--Hahn TAZRP, in both discrete and continuous time, was introduced in \cite{Corwin2015TheQB}, using formulas from \cite{Po}. The precise formulas for the generators can be found in section 2.3 of \cite{KuanCMP}.

As a direct consequence of the previous results for multi-species ASEP$(q,\boldsymbol{\theta})$ and $q$--Hahn Boson, we have the following result for the so--called stochastic higher--rank, higher--spin vertex models. This is because the algebraic construction of the duality also holds for this model. Below, we introduce informally the model and point to the most relevant bibliography.

\subsubsection{Stochastic higher--rank, higher--spin vertex models}
The stochastic multi--species, higher--rank vertex model is defined from a stochastic matrix $S(z)$ depending on a spectral parameter $z$; see  \cite{KMMO} for definitions and \cite{BoMa16} for an explicit formula. The matrix $S(z)$ is obtained from the (non--stochastic) $R$--matrix of $\mathcal{U}_q( A_n^{(1)})$ by a gauge transform. For the purposes of the present paper, the most relevant information is that the matrix depends on an asymmetry parameter $q$, a spectral parameter $z$, and two additional parameters $\lambda$ and $\mu$.

\begin{figure}
\centering
\includegraphics[]{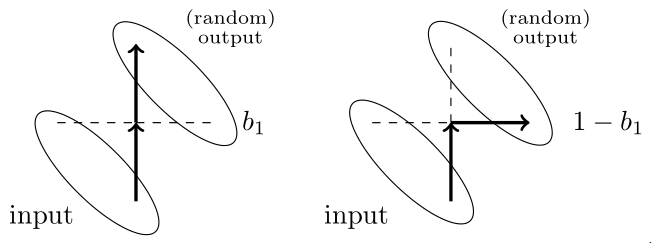}
\caption{The stochastic vertex weights.}
\label{VertexFig}
\end{figure}

The stochastic $S$--matrix defines a discrete--time Markov chain. The state space for this Markov chain is the same as the state space for the ASEP$(q,\boldsymbol{\theta})$. As in Figure \ref{FigureVertexModel}, one fixes boundaries at the left and bottom, and then updates from left to right. Thus, the arrows on the bottom are stochastically ``updated'' to arrows on the top. By viewing this update as one discrete--time update of a Markov chain, one can construct a discrete--time Markov chain. In this paper, we impose the boundary conditions in which no arrows enter from the left or exit from the right. See \cite{KuanCMP} for the mathematically precise definitions. 

\begin{figure}\label{ShowsUpdate}
\begin{center}
\includegraphics[height=5in]{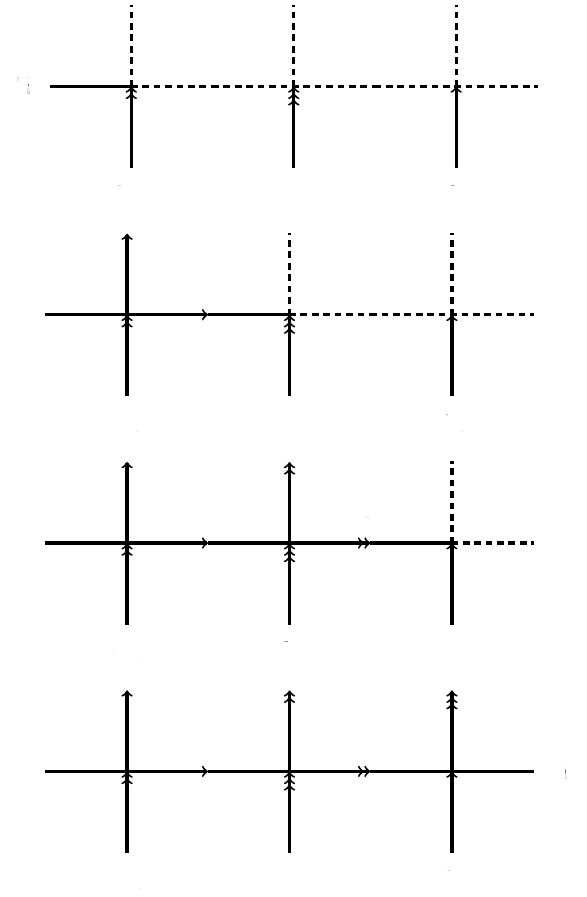}
\end{center}
\caption{This figure shows one update of the stochastic vertex model.}
\label{FigureVertexModel}
\end{figure}

\section{Main Results}\label{MR}
With all the notations and processes defined, the main results of the paper can be stated. The statements of the results are all purely probabilistic, even though some of the proofs are algebraic. 

\subsection{Orthogonal Polynomial Duality for the multi--species ASEP$(q,\boldsymbol{\theta})$}\label{OPD}
The first result regards self-duality for the multi-species ASEP$(q,\boldsymbol{\theta})$. In order to present it, we first need to recall and generalize the result for one species only.  
The function
\begin{align}\label{eq:4}
D_{\alpha_{i}}^{\boldsymbol{\theta}}\left(\boldsymbol\xi_{i},\boldsymbol\eta_{i}\right)=\prod_{x=1}^{L} K_{\eta_{i}^{x}}\left(q^{-2 \xi_{i}^{x}}, p_{i}^{x}\left(\boldsymbol\xi_{i}, \boldsymbol\eta_{i}\right), \theta^{x}, q^{2}\right) 
\\ \text{with} \ 
p_{i}^{x}\left(\boldsymbol\xi_{i}, \boldsymbol\eta_{i}\right)=\alpha_{i}^{-1} q^{-2\left(N_{x-1}^{-}\left(\boldsymbol\xi_{i}\right)-N_{x+1}^{+}\left(\boldsymbol\eta_{i}\right)\right)+2 N_{x-1}^{-}(\boldsymbol{\theta})-1} \nonumber
\end{align}
generalizes the single species homogeneous case to be inhomogeneous.
Above
$\displaystyle{N_{x-1}^{-}\left(\boldsymbol\xi_{i}\right)=\sum_{1\le y\le x-1} \xi_i^y} $ and $\displaystyle{N_{x+1}^{+}\left(\boldsymbol\eta_{i}\right)=\sum_{x+1\le y\le L}\eta_i^y }$ denotes the number of particles in the configuration considered at the left, respectively right, of  site $x$.
 Indeed, this $D_{\alpha}^{\boldsymbol{\theta}}$ with homogeneous $\boldsymbol\theta$ is up to a constant the orthogonal self--duality function for ASEP$(q,\theta)$, which was proven in Theorem 3.2 of \cite{ CFG21}.

 Below, in the section of proofs, we will see that this is the orthogonal self--duality function in the inhomogeneous case as well.
\begin{remark}

The choice of introducing a maximal occupancy $\theta^x \in \mathbb{N}$ to each site $x \in \Lambda_L$ can be interpreted as a random environment for the exclusion process. Here the maximal occupancy is not fixed and can vary among sites as the microscopic dynamics occur.  
Interacting particle systems in random environment have been considered extensively in the context of scaling limit where one is interested in the macroscopic effects of the environment. From a duality outlook, we mention the following result \cite{FRS}, where the authors extend the orthogonal polynomial dualities for symmetric particle systems in a general graph where each vertex allows an arbitrary number of particles with an open boundary.
\end{remark} 
\noindent
Now proceed to the $n$--species case. Given two particle configurations $\boldsymbol\xi$ and $\boldsymbol\eta$ of the $n$--species ASEP$(q,\boldsymbol\theta)$ on $L$ lattice sites, define intermediate particle configurations $\boldsymbol\zeta^{(i)}$ by 
$$
\boldsymbol\zeta^{(i)}_k =
\begin{cases}
\boldsymbol\xi_{k}, \quad \text{ if } k<i, \\
\boldsymbol\eta_{k}, \quad \text{ if } k>i.
\end{cases}
$$
The value of $\boldsymbol\zeta^{(i)}_i$ is uniquely determined by $\sum_{k=0}^n \boldsymbol\zeta^{(i)}_k  = \boldsymbol\theta$, i.e. $\boldsymbol\zeta^{(i)}_i=\boldsymbol\eta_{[0,i]}-\boldsymbol\xi_{[0,i-1]} $, where $\boldsymbol\eta_{[0,i]}$ is the vector $\sum_{j=0}^i\boldsymbol\eta_j=(\sum_{j=0}^i\boldsymbol\eta_j^x)_{\{x\in  \Lambda_L\}}$, to emphasis notations we are using, we recall that $\eta_{[0,i]}^x$ is the number $\sum_{j=0}^i\eta_j^x$.
\newline
As an example,
\begin{align*}
  \boldsymbol\eta &= (\boldsymbol\eta_{0}, \boldsymbol\eta_1, \ldots, \boldsymbol\eta_{n-1}, \boldsymbol\eta_{n}),\\
\boldsymbol\zeta^{(1)} &= (\boldsymbol\xi_{0}, \boldsymbol\eta_1+\boldsymbol\eta_0-\boldsymbol\xi_0, \ldots, \boldsymbol\eta_{n-1}, \boldsymbol\eta_{n}),\\
\ldots &= \ldots,
\\\boldsymbol\zeta^{(n-1)} &= (\boldsymbol\xi_{0}, \boldsymbol\xi_{1},\ldots,  \boldsymbol\xi_{n-1}+\boldsymbol\xi_{n}-\boldsymbol\eta_{n}, \boldsymbol\eta_{n}),
\\
  \boldsymbol\xi &= (\boldsymbol\xi_0, \boldsymbol\xi_1 \ldots, \boldsymbol\xi_{n-1},\boldsymbol\xi_{n}).
\end{align*}

 Additionally, define intermediate values of $\boldsymbol{\theta}^{(i)}$ by 
$$
\boldsymbol{\theta}^{(i)} = \boldsymbol\zeta^{(i)}_i +  \boldsymbol\zeta^{(i)}_{i+1}=\boldsymbol\eta_{[0,i+1]}-\boldsymbol\xi_{[0,i-1]}. 
$$

Equivalently,
\begin{multline*}
 \boldsymbol{\theta}^{(0)} = \boldsymbol\eta_0 + \boldsymbol\eta_1, \boldsymbol{\theta}^{(1)} = \boldsymbol\eta_1 + \boldsymbol\eta_2  - \boldsymbol\xi_0 +  \boldsymbol\eta_0 , \quad \ldots\\ \boldsymbol{\theta}^{(n-2)} =\boldsymbol\xi_{n-2}+\boldsymbol\xi_{n-1}-\boldsymbol\eta_{n-1} + \boldsymbol\xi_{n-3}  ,  \quad \boldsymbol{\theta}^{(n-1)} = \boldsymbol\xi_{n-1} + \boldsymbol\xi_{n}.   
\end{multline*}

 For the multi--species case, define
$$
\mathcal{D}^{\boldsymbol\theta}_{\boldsymbol{\alpha}}(\boldsymbol\xi, \boldsymbol\eta) =\mathcal{G}_{\boldsymbol{\alpha}}(\boldsymbol\xi,\boldsymbol\eta)\prod_{i=0}^{n-1}D^{  \boldsymbol{\theta}^{(i)} }_{\alpha_i}(\boldsymbol\xi_i, \boldsymbol\zeta^{(i)}_i)
 ,
$$
where 
\begin{equation} \label{curlyG}
\mathcal{G}_{\boldsymbol{\alpha}}(\boldsymbol\xi,\boldsymbol\eta) = \sqrt{ \frac{ \prod_{i=0}^{n-1}\mu_{\alpha_i}^{\boldsymbol{\theta}^{(i)}}(\boldsymbol\xi_i)\mu_{\alpha_i}^{\boldsymbol{\theta}^{(i)}}(\boldsymbol\zeta^{(i)}_i)  }{\mu^{n}(\boldsymbol\xi) \mu^{n}(\boldsymbol\eta) }},
\end{equation}
$$
\mu_{\alpha_i}^{\boldsymbol{\theta}}(\boldsymbol\xi_i):=\prod_{x=1}^{L} \alpha_i^{\xi^{x}_i}\left(\begin{array}{l}\theta^{x} \\ \xi^{x}_i\end{array}\right)_{q} q^{-\left(2 N_{x-1}^{-}(\boldsymbol{\theta})+\theta^{x}\right) \xi_i^{x}}
$$
is a reversible measure for the single--species ASEP$(q,\boldsymbol\theta)$, and 
$$
\mathcal{C}_{\boldsymbol\alpha}(\boldsymbol\xi,\boldsymbol\eta):=\sqrt{\prod_{i=0}^{n-1}q^{2\binom{N(\boldsymbol\xi_i)}{2}-2\binom{N(\boldsymbol\zeta^{(i)}_i)}{2}}\frac{\left(\alpha_i q^{1-2N(\boldsymbol\eta_{[0,i+1]}-\boldsymbol\xi_{[0,i]})} ; q^2\right)_\infty}{\left(\alpha_iq^{1-2N(\boldsymbol\eta_{[0,i]}-\boldsymbol\xi_{[0,i-1]})} ; q^2\right)_\infty}}.
$$

Note that $\mathcal{C}_{\boldsymbol{\alpha}}(\boldsymbol\xi,\boldsymbol\eta)$ is constant with respect to the dynamics, but is needed for the orthogonality below.

The first main theorem is the following and it states that $\mathcal{D}_{\boldsymbol{\alpha}}^{\boldsymbol\theta} (\boldsymbol\xi,\boldsymbol\eta)$ is an orthogonal self-duality function for the multi-species exclusion process. 
\begin{theorem}\label{FirstThm}

(a) The multi--species ASEP$(q,\boldsymbol\theta)$ is self--dual with respect to the function $\mathcal{D}_{\boldsymbol{\alpha}}^{\boldsymbol\theta} (\boldsymbol\xi,\boldsymbol\eta)$. 

(b)  Let $(a_{\boldsymbol k})$ be a family of non--negative  real numbers which sum to $1$. Then
$$
\mu^n(\boldsymbol\eta):=\sum_{\boldsymbol k} a_{\boldsymbol k} \mu^n_{ \boldsymbol k}(\boldsymbol \eta)
$$
is a reversible measure for the multi--species ASEP($q,\boldsymbol\theta$), and furthermore the duality functions from part (a) are orthogonal  with respect to these reversible measures in the sense that
$$
\sum_{\boldsymbol\xi} \mu^{n}_{}(\boldsymbol\xi) \mathcal{C}_{\boldsymbol{\alpha}}(\boldsymbol\xi,\boldsymbol\eta)\mathcal{C}_{\boldsymbol{\alpha}}(\boldsymbol\xi,\bar{\boldsymbol\eta}) \mathcal{D}_{\boldsymbol{\alpha}}^{\boldsymbol\theta}(\boldsymbol\xi,\boldsymbol\eta)\mathcal{D}_{\boldsymbol{\alpha}}^{\boldsymbol\theta}(\boldsymbol\xi,\bar{\boldsymbol\eta}) = 
\begin{cases}
\frac{1}{\mu^{n}(\boldsymbol\eta)}, \quad \boldsymbol\eta = \bar{\boldsymbol\eta}, \\
0, \quad \boldsymbol\eta \neq \bar{\boldsymbol\eta}.
\end{cases}
$$

\end{theorem}

This theorem generalizes a few previously known results.
\begin{itemize}
\item
In the single--species case, when $n=1$, the results of \cite{CFG21} are recovered, namely an inhomogeneous nested product of $q$-Krawtchouk polynomials.

\item
In the symmetric case, when $q\xrightarrow[]{}1$, the duality function reduces to a special case of the homogeneous product of multivariate Krawtchouk polynomials \cite{ZZ}, using the definition of  multivariate Krawtchouk polynomials from Appendix 2 in \cite{Milch}.

\item
If the two conditions above are considered, i.e. setting $n=1$ and taking the $q\xrightarrow[]{}1$ limit, one gets an homogeneous product of Krawtchouk polynomials as self-duality function, see for example \cite{Gro19} and \cite{FG}, where orthogonal polynomials appeared for the first time as duality function.
\end{itemize}
       
\begin{remark}[From orthogonal dualities to triangular ones]
 One might wonder if there is a link between the orthogonal duality function of Theorem \ref{FirstThm} and the triangular one of Theorem 2.5 (a) in \cite{Kuan-IMRN}, which are both duality function for the multi--species ASEP$(q,\boldsymbol{\theta})$. For one specie this link has been shown in Remark 6.3 of \cite{CFG21}. Here, in the same spirit of rank $1$, we show that the triangular duality can be recovered after taking an appropriate limit. 
    Notice that  
\begin{align*}
    \lim_{\alpha_i\xrightarrow[]{}0}&(-\alpha_i)^{\xi_i^x}K_{\eta_{i}^{x}}\left(q^{-2 \xi_{i}^{x}}, p_{i}^{x}\left(\boldsymbol\xi_{i}, \boldsymbol\eta_{i}\right), \theta^{x}, q^{2}\right)\\
   & =1_{\{\eta_i^x\ge \xi_i^x\}} \frac{\left(q^{-2\eta_i^x};q^2\right)_{\xi_i^x}}{\left(q^{-2\theta^x};q^2\right)_{\xi_i^x}}q^{-2\xi_{i}^{x}\left(N_{x-1}^{-}\left(\boldsymbol\xi_{i}\right)-N_{x}^{+}\left(\boldsymbol\eta_{i}\right)\right)+2 \xi_{i}^{x}N_{x-1}^{-}(\boldsymbol{\theta})-\left(\xi_{i}^{x}\right)^2}.
\end{align*}
Thus 
\begin{align*}
  &  \lim_{\alpha_i\xrightarrow[]{}0,0\le i\le n-1}\mathcal{D}^{\boldsymbol\theta}_{\boldsymbol{\alpha}}(\boldsymbol\xi, \boldsymbol\eta) \prod_{i=0}^{n-1}\frac{(-\alpha_i)^{N(\boldsymbol\xi_i)}}{\alpha_i^{\frac{1}{2}\left(N(\boldsymbol\xi_i+\boldsymbol\zeta^{(i)}_i)\right)}}\\
   & =\frac{\mathcal{G}_{\boldsymbol{\alpha}}(\boldsymbol\xi, \boldsymbol\eta) }{\prod_{i=0}^{n-1}\alpha_i^{\frac{1}{2}\left(N(\boldsymbol\xi_i+\boldsymbol\zeta^{(i)}_i)\right)}} \prod_{i=0}^{n-1}\prod_{x\in \Lambda_L} 1_{\{\eta_{[0,i]}^x\ge\xi_{[0,i]}^x\}}\frac{\left(q^{-2\left(\eta^x_{[0,i]}-\xi^x_{[0,i-1]}\right)};q^2\right)_{\xi_i^x}}{\left(q^{-2\left(\eta^x_{[0,i+1]}-\xi^x_{[0,i-1]}\right)};q^2\right)_{\xi_i^x}}\\
  &\quad\times q^{2\xi_{i}^{x}N_{x}^{+}\left(\boldsymbol\eta_{[0,i]}-\boldsymbol\xi_{[0,i-1]}\right)+2 \xi_{i}^{x}N_{x-1}^{-}(\boldsymbol\eta_{[0,i+1]}-\boldsymbol\xi_{[0,i-1]})-N\left(\boldsymbol\xi_{i}\right)^2}\\
  &=C(N(\boldsymbol\xi_i),N(\boldsymbol\eta_i),\boldsymbol\theta) \prod_{x \in \Lambda_L}\left[\eta_0^x\right]_q^{!} \prod_{i=0}^{n-1} 1_{\eta_{[0, i]}^x \geq \xi_{[0, i]}^x} \frac{\left[\eta_{[0, i+1]}^x-\xi_{[0, i]}^x\right]_q^{!}}{\left[\eta_{[0, i]}^x-\xi_{[0, i]}^x\right]_q^{!}}\\
  &\quad\times q^{\xi_i^x\left(2N_{x-1}^-(\boldsymbol{\theta})+\theta^x+2N_{x+1}^+(\boldsymbol{\eta}_{[0,i]})+\eta^x_{[0,i]}\right)},
\end{align*}
where 
\begin{align*}
C(N(\boldsymbol\xi_i),N(\boldsymbol\eta_i),\boldsymbol\theta)&= -\frac{1}{2}\sum_{x\in\Lambda_L}\left(\theta^x\right)^2\\
&+\sum_{i=0}^{n-1}\left(N(\boldsymbol\xi_i)\left(N(\boldsymbol\eta_{[0,i+1]})-2N(\boldsymbol\xi_{[0,i-1]})-\frac{3}{2}N(\boldsymbol\xi_i)\right)\right.\\
&\left.-N(\boldsymbol\eta_{[0,i+1]}-\boldsymbol\xi_{[0,i-1]})^2+N(\boldsymbol\eta_{i+1})N(\boldsymbol\eta_{[0,i+1]})-\frac{1}{2}N(\boldsymbol\eta_{i+1})\right).
\end{align*}
This function is, up to a constant, equal to the triangular duality given by Theorem 2.5 (a) in \cite{Kuan-IMRN} when $\theta^x=2j$.
\end{remark}

\subsection{Example}
Before proceeding with the next results we look at three different concrete examples where the maximum number of particles for each sites is fixed and equal to $2$.
 \begin{example} Suppose $n=2$, then the self--duality function is 
\begin{multline*}
     \mathcal{D}(\boldsymbol\xi, \boldsymbol\eta)=D_{\alpha_{0}}^{\left(\boldsymbol\eta_{0}+\boldsymbol\eta_{1}\right)}\left(\boldsymbol\xi_{0}, \boldsymbol\eta_{0}\right)  D_{\alpha_{1}}^{\left(\boldsymbol\theta-\boldsymbol\xi_{1}\right)}\left(\boldsymbol\xi_{1}, \boldsymbol\eta_{0}+\boldsymbol\eta_{1}-\boldsymbol\xi_{1}\right)\\ \sqrt{\frac{\mu^{(\boldsymbol\eta_0+\boldsymbol\eta_1)}_{\alpha_0}(\boldsymbol\xi_0)\mu^{(\boldsymbol\eta_0+\boldsymbol\eta_1)}_{\alpha_0}(\boldsymbol\eta_0)\mu^{(\boldsymbol\theta-\boldsymbol\xi_0)}_{\alpha_1}(\boldsymbol\xi_1)\mu^{(\boldsymbol\theta-\boldsymbol\xi_0)}_{\alpha_1}(\boldsymbol\eta_0+\boldsymbol\eta_1-\boldsymbol\xi_0)}{\mu^3(\boldsymbol\xi) \mu^3(\boldsymbol\eta)}}.
\end{multline*}

For $\theta^x\equiv 2,L=2$ consider the 
four particle configurations

$$
\begin{array}{llllllllllll}0 & 2 & & 2 & 2 & &  2 & 2 & & 2 & 0 \\ \underline{1} & \underline{2} && \underline{1} & \underline{0} && \underline{0} & \underline{1} && \underline{2} & \underline{1}\end{array}
$$
where $2$ denotes a hole.
Using these 
four particle configurations as an ordered basis, the duality 
function is expressed as the matrix
\scriptsize
\begin{equation*}
 \mathcal{D}=   \left( \scalemath{0.6}{\begin{array}{cccc}-q^{3}

(q^2 + 1)(\alpha_1 - q)C_0/q^8  &          \alpha_1(q^2 + 1)(\alpha_0 - q)/q^6                 &                     \alpha_0\alpha_1(q + 1/q)/q^5                              &                                                 0\\
 \alpha_1(q^2 + 1)C_0/q^8       &  (\alpha_0 - q)(- q^5 + \alpha_1q^2 + \alpha_1)/q^6  &            \alpha_0(- q^5 + \alpha_1q^2 + \alpha_1)/q^6                    &                                                  0\\
  0                                                                                         &        \alpha_0(- q^3 + \alpha_1q^2 + \alpha_1)/q^4        &  (- q^3 + \alpha_0)(- q^3 + \alpha_1q^2 + \alpha_1)/q^4           &           \alpha_1(q^2 + 1)C_0/q^4\\
   0                                                                                        &                 \alpha_0\alpha_1(q + 1/q)/q^3                  &           \alpha_1(q^2 + 1)(- q^3 + \alpha_0)/q^4                             &  (q^2 + 1)(- q^5 + \alpha_1)C_0/q^4
 \end{array}}\right),
\end{equation*}
where $C_0$ is the constant $(- q^3 +\alpha_0q^2 + \alpha_0)$. With respect to the same basis, the generator is 
$$
\mathcal{L}=\left(\begin{array}{cccc}-q^{3}\left(q+q^{-1}\right)^{2} & q^{3}\left(1+q^{2}\right) & q\left(1+q^{2}\right) & 0 \\ q^{7} & -\left(q+q^{3}+q^{7}\right) & q^{3} & q \\ q^{7} & q^{5} & -\left(q+q^{5}+q^{7}\right) & q \\ 0 & q^{5}\left(1+q^{2}\right) & q^{3}\left(1+q^{2}\right) & -q^{5}\left(q+q^{-1}\right)^{2}\end{array}\right).$$
It follows directly that the duality holds.
\end{example}

\begin{example}

Suppose $L=3,n=2$, $\theta^x\equiv 2$, Let $X_t,Y_t$ be two multi-species ASEP with initial condition $x_1,y_1$ as 
following.
\begin{equation*}
 x_1=    \begin{matrix}
    2 & 2 & 2\\
   \underline{0}  & \underline{2} &\underline{2}
    \end{matrix}
\ \ \ \     y_1=    \begin{matrix}
    2 & 2 & 2\\
   \underline{1}  & \underline{2} &\underline{2}
    \end{matrix}
\end{equation*}

\begin{equation*}
 x_2=    \begin{matrix}
    2 & 2 & 2\\
   \underline{2}  & \underline{0} &\underline{2}
    \end{matrix}
\ \ \ \     y_2=    \begin{matrix}
    2 & 2 & 2\\
   \underline{2}  & \underline{1} &\underline{2}
    \end{matrix}
\ \ \ \        x_3=    \begin{matrix}
    2 & 2 & 2\\
   \underline{2}  & \underline{2} &\underline{0}
    \end{matrix}
\ \ \ \     y_3=    \begin{matrix}
    2 & 2 & 2\\
   \underline{2}  & \underline{2} &\underline{1}
    \end{matrix}
\end{equation*}

Then $X_t,Y_t$ reduce to single particle ASEP with the same jumping rates. 

\begin{equation*}
    \mathbb{E}_{x_1}[D(X_t,y_1)]=\sum_{i=0}^2 \mathbb{P}_{x_1}(X_t=x_i)D(x_i,y_1).
\end{equation*}
Note  that  when $i\neq 1$, $D(x_i,y_1)=D(x_1,y_i)=0$, thus the duality relation holds.
\end{example}
\begin{example}
Suppose $L=3,n=2$, $\theta^x \equiv 2$, Let $X_t,Y_t$ be two multi-species ASEP with initial condition $x_1,y_1$ as following. Then the evolution of species $1$ particle in both process is the same. 

\begin{equation*}
 x_1=    \begin{matrix}
    2 & 2 & 2\\
   \underline{0}  & \underline{2} &\underline{2}
    \end{matrix}  
\ \ \ \     y_1=    \begin{matrix}
    1 & 1 & 1\\
   \underline{0}  & \underline{1} &\underline{1}
    \end{matrix}
\end{equation*}

\begin{equation*}
 x_2=    \begin{matrix}
    2 & 2 & 2\\
   \underline{2}  & \underline{0} &\underline{2}
    \end{matrix}
\ \ \ \     y_2=    \begin{matrix}
    1 & 1 & 1\\
   \underline{1}  & \underline{0} &\underline{1}
    \end{matrix}
\ \ \ \        x_3=    \begin{matrix}
    2 & 2 & 2\\
   \underline{2}  & \underline{2} &\underline{0}
    \end{matrix}
\ \ \ \     y_3=    \begin{matrix}
    1 & 1 & 1\\
   \underline{1}  & \underline{1} &\underline{0}
    \end{matrix}
\end{equation*}
One can check that $D(x_i,y_1)=D(x_1,y_i)$ for any $i$. 
\newline 
By the fact that $\mathbb{P}_{x_1}(X_t=x_i)=\mathbb{P}_{y_1}(Y_t=y_i)$, the duality still holds. 
\end{example}

\subsection{Duality for the multi--species $q$--TAZRP}
The theorem is a duality result for the multi--species totally asymmetric zero range process. The so--called charge--parity transformation $T$ was introduced in \cite{Kuan-IMRN} in the multi--species case. It is an involution which satisfies the property that $TLT^{-1}=L^{\text{rev}}$, where $L$ is the generator of multi--species ASEP$(q, \boldsymbol\theta)$ and $L^{\text{rev}}$ is the space--reversed generator.

The idea is that on site $x \in \Lambda_L$ we replace the $i^{th}$ species with the $(n-i)^{th}$ one  for $0\le i\le n$, namely  $(T(\boldsymbol\xi))^x_i:=\xi^x_{n-i}.$ Once the charge--parity transformation is applied,  we let the number of particles per site be unbounded,  namely we consider the limit $\mathbf\theta^x\rightarrow \infty$, for all lattice sites $x \in \Lambda_L$, then the duality function $\mathcal{D}_{\boldsymbol{\alpha}}^{\boldsymbol\theta} (\boldsymbol\xi,\boldsymbol\eta)$ yields a new duality function for general $n$-species  $q$-TAZRP ($n\ge 2$).

\begin{theorem}\label{qTAZRP}
Let $T$ be the involution $(T(\boldsymbol\xi))^x_i=\xi^x_{n-i}$ for $1\le x\le L,0\le i\le n$. Then the function

\begin{align}
& D^{q-TARZP}(\boldsymbol\xi,\boldsymbol\eta)  \nonumber  \\ & = q^{h(\boldsymbol\xi,\boldsymbol\eta)} \prod_{i=0}^{n-1}\prod_{x \in \Lambda_L}{}_1\phi_0\left(q^{-2\xi_i^x};q^2,q^{-2\left(N_{x-1}^-(\boldsymbol\xi_i)+N_{x+1}^+(T(\boldsymbol\eta)_{i+1})\right)+1}\right)    \nonumber    \\
   &=  q^{h(\boldsymbol\xi,\boldsymbol\eta)} \prod_{i=0}^{n-1}\prod_{x \in \Lambda_L} ( q^{-2\left(N_{x}^-(\boldsymbol\xi_i)+N_{x+1}^+(T(\boldsymbol\eta)_{i+1})\right)+1} ;q^2  )_{\xi_i^x}   \label{dualityq-TARZP} 
\end{align}
is a space reversed duality for $n $-species $q$-TAZRP ($n\ge 1$), where $\boldsymbol\eta$ evolves with total asymmetry to the left and $\boldsymbol\xi$ evolves with total asymmetry to the right, and 
\begin{equation}\label{eq:1}
    h(\boldsymbol\xi,\boldsymbol\eta)=\sum_{x \in \Lambda_L}\sum_{i=0}^{n-1}\left(-\xi_i^xN_{x+1}^+(\boldsymbol\eta_{[0,n-2-i]})+\eta_i^xN_{x}^+(\boldsymbol\xi_{[0,n-2-i]})\right).
\end{equation}
\end{theorem}
Note that when $\boldsymbol\eta$ consists of finitely many sites each occupied with infinitely many particles, and $\boldsymbol\xi$ consists of finitely many particles, the duality function becomes an indicator function.

\subsection{Stochastic higher--spin, higher--rank vertex models}
 In \cite{KuanCMP}, it was proved that after applying charge--parity symmetry of the previous section,  the triangular duality function for the multi--species ASEP$(q,\boldsymbol\theta)$ is also a duality function for the stochastic higher--spin, higher--rank vertex models. The argument used there immediately implies the same for the orthogonal polynomial duality function, as stated below.

\begin{theorem}\label{Vertex}
The stochastic higher--spin, higher--rank vertex model, with total asymmetry to the right, is dual to its space reversal with respect to the function
$$
\mathcal{D}_{\boldsymbol{\alpha}}^{\boldsymbol{\theta}}(\boldsymbol{\xi},T(\boldsymbol{\eta})).
$$

\end{theorem}
Note that taking the $\boldsymbol\theta\rightarrow\infty$ limit recovers Theorem \ref{qTAZRP}.

\subsection{Orthogonal Polynomial Duality for the multi--species $q$--Hahn TAZRP}
In \cite{KuanCMP}, it was shown that the triangular duality for the multi--species $q$--Boson was also a duality for the multi--species $q$--Hahn Boson. The proof proceeded by direct computation. As such, it is natural to check if the orthogonal duality for the multi--species $q$--Boson is also a duality for the multi--species $q$--Hahn Boson. Indeed, this turns out to be the case:
\begin{theorem}\label{qHahn}
The multi--species $q$--Hahn TAZRP (with drift to the right) is dual to its space reversal (with drift to the left), with respect to the same function as in Theorem \ref{qTAZRP}.

\end{theorem}
\subsection{Application of the Duality}
We end this section with an application of the duality function for the two--species $q$-TARZP in equation \eqref{dualityq-TARZP}. It is natural to exploit the duality relation for the computation of suitable moments of the process. In particular, one expects interacting particle systems to exhibit long range correlations arising as a consequence of the dual microscopic interaction. Using a simple initial configuration, the duality function can give information on such quantities, which are explicitly known only for few special cases. However, in the context of multi-species asymmetric particle systems, applications of duality are still under investigation. In \cite{BS3} the authors took advantage of the self-duality of multi-species exclusion process to obtain all invariant measures and to study the microscopic shocks dynamics. Below we show that, choosing a dual configuration $\boldsymbol{\eta}$ consisting of only two particles of different species and following their dynamics, we get a close formula for the expectation with respect to the initial process $\boldsymbol{\xi}$ in terms of the product of two $q-$shifted factorials.

Recall the incomplete gamma function $$\Gamma(a,z) =\int_{z}^{\infty} t^{a-1} \mathrm{e}^{-t} \mathrm{~d} t
$$
and define its normalization 
$$
Q(a, z)=\frac{\Gamma(a, z)}{\Gamma(a)} .
$$

\begin{theorem}\label{Application}
Suppose $\boldsymbol\xi$ has initial conditions with $n_1$ species 0 particles at site $x_1$ and $n_2$ species 1  particles at site $x_2$, or in other words 
$$
\boldsymbol\xi_i^x 
= 
\begin{cases}
n_1 \ \text{for} \  i=0,x=x_1,\\
n_2 \ \text{for} \  i=1,x=x_2,\\
0 \ \  \text{else},
\end{cases}
$$  
and $\boldsymbol\xi$ evolves as a $q$--TAZRP with total asymmetry to the right. Let $y_1$ and $y_2$ be some lattice sites, and additionally assuming that $x_1<x_2$.

The dual process $\boldsymbol\eta$ has 1 species 0 particles at site
$y_1$, 1 species 1 particles at site $y_2$, so in other words
$$
\boldsymbol\eta_i^y 
= 
\begin{cases}
1 \ \text{for} \ i=0,y=y_1,\\
1 \ \text{for} \ i=1,y=y_2,\\
0 \ \  \text{else}.
\end{cases}
$$  
Then there is the explicit formula 
\begin{align*}
\mathbb{E}_{\boldsymbol{\xi}} \Big[ q^{h(\boldsymbol\xi (t),\boldsymbol\eta)} \prod_{i=0}^{1}\prod_{x=1 \in \Lambda_L} & {}_1\phi_0\left(q^{-2\xi_i^x (t)};q^2,q^{-2\left(N_{x-1}^-(\boldsymbol\xi_i (t))+N_{x+1}^+(T(\boldsymbol\eta)_{i+1})\right)+1}\right) \Big]
\\ &  = \sum_{i=1}^6{\delta_ip_i},
\end{align*}
where
\begin{align*}   
   \delta_1:= q^{n_1}\left(q^{-2n_1-1};q^2\right)_{n_1} \left(q^{-2n_2+1};q^2\right)_{n_2}    ,\\
    \delta_2:=     q^{n_1}\left(q^{-2n_1+1};q^2\right)_{n_1} \left(q^{-2n_2+1};q^2\right)_{n_2}  ,\\
     \delta_3:=   q^{-n_1}\left(q^{-2n_1-1};q^2\right)_{n_1} \left(q^{-2n_2+1};q^2\right)_{n_2}    ,\\
      \delta_4:=    q^{-n_1}\left(q^{-2n_1+1};q^2\right)_{n_1} \left(q^{-2n_2+1};q^2\right)_{n_2}  ,\\
   \delta_5:=   q^{-n_1}\left(q^{-2n_1-1};q^2\right)_{n_1} \left(q^{-2n_2-1};q^2\right)_{n_2}  ,\\
    \delta_6:=     q^{-n_1}\left(q^{-2n_1+1};q^2\right)_{n_1} \left(q^{-2n_2-1};q^2\right)_{n_2}  ,
  \end{align*}
  and
  $$
p_1 = q_1-q_4, \quad p_2=q_4, \quad p_3 = q_2+q_3-q_5-q_6, 
$$
$$\quad p_4 = -q_3 + q_5+q_6, \quad p_5=q_5, \quad p_6=q_3-q_5,$$
with
$$
q_1 ={\delta_{y_1\le x_1+}\delta_{y_1> x_1} }Q(y_1-x_1,t) , \quad q_3 = 1 - {\delta_{y_1\ge x_2}Q(y_1-x_2,t)}{-\delta_{y_1<x_2}} , 
$$
and $\quad q_2 = 1- q_1-q_3,$
\begin{align*}
q_4&=\frac{ 1}{(2\pi i)^2}\int_{C} \frac{d w_{1}}{w_{1}}  \int_{C} \frac{d w_{2}}{w_{2}} \frac{w_1-w_2}{w_1-qw_2} \prod_{j=1}^{2}\left[(1-w_j)^{-(y_j-x_1+1)}e^{-w_{j} t}\right] ,\\
q_6 &= 1-q_1 -\frac{ q}{(2\pi i)^2}\int_{\tilde{C}_{1}} \frac{d w_{1}}{w_{1}}  \int_{\tilde{C}_{2}} \frac{d w_{2}}{w_{2}} \frac{w_1-w_2}{w_1-qw_2}\prod_{j=1}^{2}\left[(1-w_j)^{-(y_j-x_1+1)}e^{-w_{j} t}\right],\\
q_5&= \frac{ q}{(2\pi i)^2}\int_{\tilde{C}_{1}}  \frac{dw_1}{w_1}  \int_{\tilde{C}_{2}}  \frac{dw_2}{w_2} \frac{w_1-w_2}{w_1-qw_2}\frac{e^{-(w_1+w_2)t  } }{(1-w_1)^{({y_1-x_2}+1)}(1-w_2)^{({y_2-x_1}+1)}}  . \\
\end{align*}
\end{theorem}

\begin{remark}
An asymptotic analysis of  the two--point correlations can be done, but requires technical details which are better-suited for a future paper. 
\end{remark}

\begin{remark}\label{remark:1}
 Note that in the $x_2\geq x_1$ situation, if furthermore $y_1,y_2>x_2$, then the expectation of the moments trivially reduces to a single--species $q$--TAZRP when there are infinitely many first--class particles. This is because the jump rate for a second--class particle is multiplied by $q^{\infty}=0$, so a jump for a second--class particle never occurs once it reaches $x_1$.
\end{remark}

\section{Proofs}\label{Pfs}
We devote this section to the rigorous proof of the previous results. 

\begin{remark}\label{REURem}

The proofs presented in this section generalize to other quantum groups and representations. In an upcoming paper by the second and third authors, the method is applied to the quantum group $\mathcal{U}_q(\mathfrak{so}_6)$ and the type $D$ ASEP \cite{REU2022}. 
\end{remark}
\noindent

Below we start by recalling the algebraic relations used in our proofs regarding the quantum group $\mathcal{U}_q(\mathfrak{gl}_{n+1})$.

\subsection{Quantum Groups and Representations}
The Drinfeld--Jimbo quantum group $\mathcal{U}_q(\mathfrak{gl}_{n+1})$ is the Hopf algebra with generators 
$$
\{E_{i,i+1},E_{i+1,i}: 0 \leq i < n\}, \{ q^{E_{ii}}: 0 \leq i \leq n\}
$$
and relations
\begin{align*}
& q^{E_{i i}} q^{E_{j j}}=q^{E_{j j}} q^{E_{i i}}=q^{E_{i i}+E_{j j}} \\ & \left[E_{i, i+1}, E_{i+1, i}\right]=\frac{q^{E_{i i}-E_{i+1, i+1}}-q^{E_{i+1, i+1}-E_{i i}}}{q-q^{-1}} \\& 
\left[E_{i, i+1}, E_{j+1, j}\right]=0, \quad i \neq j.
\end{align*}
For $j \neq i,i-1$
\begin{align*}
q^{E_{i i}} E_{i, i+1}=q E_{i, i+1} q^{E_{i i}} \quad q^{E_{i i}} E_{i-1, i}=q^{-1} E_{i-1, i} q^{E_{i i}} \quad\left[q^{E_{i i}}, E_{j, j+1}\right]=0.
\end{align*}
For $j \neq i,i+1$
\begin{align*}
q^{E_{i i}} E_{i, i-1}=q E_{i, i-1} q^{E_{i i}} \quad q^{E_{i i}} E_{i+1, i}=q^{-1} E_{i+1, i} q^{E_{i i}} \quad\left[q^{E_{i i}}, E_{j, j-1}\right]=0.
\end{align*}
For $i=j \pm 1$
\begin{align*}
 & E_{i, i+1}^{2} E_{j, j+1}-\left(q+q^{-1}\right) E_{i, i+1} E_{j, j+1} E_{i, i+1}+E_{j, j+1} E_{i, i+1}^{2}=0,
 \\
 & E_{i, i-1}^{2} E_{j, j-1}-\left(q+q^{-1}\right) E_{i, i-1} E_{j, j-1} E_{i, i-1}+E_{j, j-1} E_{i, i-1}^{2}=0.
 \end{align*}
 For $i \neq j \pm 1$
 \begin{align*}
 {\left[E_{i, i+1}, E_{j, j+1}\right]=0=\left[E_{i, i-1}, E_{j, j-1}\right] }.
\end{align*}
The co--product $\Delta: \mathcal{U}_q(\mathfrak{gl}_{n+1}) \to \mathcal{U}_q(\mathfrak{gl}_{n+1}) \otimes  \mathcal{U}_q(\mathfrak{gl}_{n+1})$ is defined by
\begin{align*}{}\Delta\left(E_{i, i+1}\right)&=E_{i, i+1} \otimes 1+q^{ E_{ii} - E_{i+1,i+1}} \otimes E_{i, i+1}, \\ \Delta\left(E_{i+1, i}\right)&=  1 \otimes E_{i+1, i}  + E_{i+1, i} \otimes q^{-(E_{ii}-E_{i+1,i+1}) },\\ \Delta\left(q^{E_{i, i}}\right)&=q^{E_{i, i}} \otimes q^{E_{i, i}}.\end{align*}
We also define iteratively higher powers $m \geq 2 $ of $\Delta$ as
\begin{equation*}
    \Delta^{m}: \mathcal{U}_q(\mathfrak{gl}_{n+1}) \to \mathcal{U}_q(\mathfrak{gl}_{n+1})^{\otimes (m+1)}
    \quad \text{where}
\end{equation*}
$$
\Delta^1:=\Delta, \qquad \Delta^m := (\Delta\otimes \underbrace{1\otimes\ldots\otimes 1}_{m-1 \text{ times}}) \Delta^{-1} \;.
$$
The antipode will not be explicitly needed. 
When $q$ is a nonzero real number, then $\mathcal{U}_q(\mathfrak{gl}_{n+1})$ has the structure of a Hopf $*$--algebra with the $*$ involution defined by 
$$
E_{i,i+1}^* = E_{i+1,i}q^{E_{ii}-E_{i+1,i+1}}, \quad E_{i+1,i}^*= q^{-E_{ii}+E_{i+1,i+1}}E_{i,i+1}, \quad (q^{E_{ii}})^* = q^{E_{ii}}.
$$

\begin{remark}\label{CentralExplanation}
In general, Drinfeld--Jimbo quantum groups have a $*$--Hopf algebra structure when $q$ is positive. See, for example, sections 9.2.4 and 9.3.5 of \cite{KSBook}. In the context of probability theory, all models will require $q$ to be positive anyway, so this is not a restrictive requirement.
\end{remark}
For any $0 \leq i \neq j \leq n$, define the root vectors  $E_{ij}$ inductively by 
$$
E_{ij} = E_{ik}E_{kj} - q^{-1}E_{kj}E_{ik}, \quad i <k<j \text{ or } i > k > j.
$$
This definition does not depend on the choice of $k$. With this definition, there is a central element \cite{GZB}:
\begin{equation*}
C_1:=\sum_{i=0}^{n} q^{2 i-2 n-1} q^{2 E_{i i}}+\left(q-q^{-1}\right)^{2} \sum_{0 \leq i<j \leq n} q^{2 j-2 n-2} q^{E_{i i}+E_{j j}} E_{i j} E_{j i}.
\end{equation*}
One can check by direct computation that $C_1^*=C_1$. Alternatively, this can be seen without computations. Namely, the $*$--involution must preserve the center, and since there is only one first--order Casimir element (by the Harish--Chandra isomorphism), it must be the case that $C_1^*=aC_1+b$. By checking any $q^{E_{ii}}$ term in $C_1$, we must have that $a=1$ and $b=0$. 

\begin{remark}
The Lie algebra $\mathfrak{gl}_{n+1}$ is not simple, but contains a simple sub-algebra $\mathfrak{sl}_{n+1}$ of co--dimension one. The simple Lie algebra $\mathfrak{sl}_{n+1}$ is  simply the algebra of $(n+1) \times (n+1)$ matrices with trace zero. The reason for instead using $\mathfrak{gl}_{n+1}$ is that the central element $C_1$ had already been explicitly computed. 
\end{remark}

Given any $m \geq 1$, let 
$$\mathcal{B}_m^{(n)} = \{\mu = (\mu_0,\ldots,\mu_n): \mu_0+\ldots+\mu_n=m\} \subseteq \mathbb{Z}^{n+1}.$$
Let $V^{(n)}_m$ be a vector space with basis vectors indexed by $\mathcal{B}_m^{(n)}$. Denote the basis vectors by $v_{\mu}$ for $\mu \in \mathcal{B}_m^{(n)}$. It was proved in Lemma 3.1 of \cite{Kuan-IMRN} that  $V^{(n)}_m$ is a representation of $\mathcal{U}_q(\mathfrak{gl}_{n+1})$ with action given by 
\begin{align*}
E_{i,i+1} v_{\mu} &= [\mu_{i+1}]_q v_{ \mu + \epsilon_i}, \quad 0 \leq i < n, \\
E_{i+1,i}v_{\mu}  &= [\mu_{i}]_q v_{\mu - \epsilon_i},\quad 0 \leq i < n, \\
q^{E_{ii}} v_{\mu} &= q^{\mu_i} v_{\mu}, \quad 0 \leq i \leq n,
\end{align*}
where $\epsilon_i = (0,0,\ldots,0,1,-1,0,\ldots,0)$ with the $1$ in the $i$th position. Furthermore, lemma 3.1 of \cite{Kuan-IMRN} proves that the root vectors $E_{ij}$ act as
\begin{equation}\label{Act}
\begin{aligned} E_{i j} v_{\mu} &=q^{\mu_{i+1}+\mu_{i+2}+\ldots+\mu_{j-1}}\left[\mu_{j}\right]_{q} v_{\mu_{j \rightarrow i}}, \\ E_{j i} v_{\mu} &=q^{\mu_{i+1}+\mu_{i+2}+\ldots+\mu_{j-1}}\left[\mu_{i}\right]_{q} v_{\mu_{i \rightarrow j}}, \end{aligned}
\end{equation}
where for $i<j$
\begin{equation*}
\begin{aligned} \mu_{j \rightarrow i} &=\left(\mu_{0}, \ldots, \mu_{i-1}, \mu_{i}+1, \mu_{i+1}, \ldots, \mu_{j-1}, \mu_{j}-1, \mu_{j+1}, \ldots, \mu_{n}\right), \\ \mu_{i \rightarrow j} &=\left(\mu_{0}, \ldots, \mu_{i-1}, \mu_{i}-1, \mu_{i+1}, \ldots, \mu_{j-1}, \mu_{j}+1, \mu_{j+1}, \ldots, \mu_{n}\right).\end{aligned}
\end{equation*}
So $\mu_{j\rightarrow i}$ represents the particle configuration where a species $j$  particle has been replaced with species $i$ 
particle.

Recall that a $*$--representation $V$ of a $*$--algebra $\mathcal{A}$ is a representation $V$ that carries an inner product $\langle \cdot, \cdot \rangle$ such that
$$
\langle av,w\rangle = \langle v,a^*w\rangle
$$
for all $a\in \mathcal{A}$ and $v,w\in V$. It will be shown later that there exists an inner product such that $V_m^{(n)}$ is a $*$--representation.

\subsection{A direct proof of Theorem \ref{FirstThm}(b)}
The orthogonality result in Theorem \ref{FirstThm}(b) can be proved directly using special functions, without prior algebraic and probability background. As such, it will be proved first. Before doing so, here is a lemma which generalizes Proposition 7.1 of \cite{CFG21} for the inhomogeneous case of the single--species ASEP$(q,\boldsymbol{\theta})$.

Given two single-species particle configurations $\boldsymbol{\xi}$ and $\boldsymbol{\eta}$, define 
\begin{equation*}
    K^{\boldsymbol\theta}(\boldsymbol\xi,\boldsymbol\eta)=\prod_{x}K_{\eta^x}\left(q^{-\xi^x},pq^{N_{x-1}^-(\boldsymbol\theta)}q^{-N_{x-1}^-(\boldsymbol\xi)+N_{x+1}^+(\boldsymbol\eta)},\theta^x,q\right)
\end{equation*}

and also let 
\begin{equation*}
    w_{\boldsymbol\theta}(\boldsymbol\xi)=\frac{(-1)^{N(\boldsymbol\theta)-N(\boldsymbol\xi)}}{\left(pq^{N(\boldsymbol\theta)-N(\boldsymbol\xi)+1};q\right)_\infty }\prod_{x=1}^L\frac{q^{\binom{\xi^x}{2}}q^{\theta^x N_{x-1}^-(\boldsymbol\xi)}}{(q;q)_{\xi^x}(q;q)_{\theta^x-\xi^x}},
\end{equation*}
\begin{equation*}
    h_{\boldsymbol\theta}(\boldsymbol\eta)=\frac{(-1)^{N(\boldsymbol\eta)}q^{\binom{N(\boldsymbol\theta)+1}{2}}p^{N(\boldsymbol\theta)}}{\prod_x[(q;q)_{\theta^x}]^{2} \left(pq^{N(\boldsymbol\eta)+1};q\right)_\infty} \prod_{x=1}^L\frac{(q;q)_{\eta^x}(q;q)_{\theta^x-\eta^x}}{q^{\binom{\eta^x+1}{2}}}q^{\theta^x N_{x}^+(\boldsymbol\eta)}.
\end{equation*}

\begin{lemma}\label{lemma:2}  There is the orthogonality relation
$$
\sum_{\boldsymbol\xi} \frac{w_{\boldsymbol\theta}(\boldsymbol\xi)}{h_{\boldsymbol\theta}(\boldsymbol\eta)} K^{\boldsymbol\theta}(\boldsymbol\xi,\boldsymbol\eta) K^{\boldsymbol\theta}(\boldsymbol\xi,\boldsymbol{\bar{\eta}}) = \delta_{\boldsymbol\eta,\bar{\boldsymbol\eta}}.
$$
\end{lemma}
\begin{proof}

We start with the standard orthogonality relation for $q$-Krawtchouk (equation (112) of \cite{CFG21} or section 1.14.14 of \cite{KLS})
\begin{multline*}
     \sum_{k=0}^{\theta}\frac{(pq;q)_{\theta-k}(-1)^{\theta-k}q^{\binom{k}{2}}}{(q;q)_k(q;q)_{\theta-k}}K_m(q^{-k},p,\theta,q)K_n(q^{-k},p,\theta,q)\\
     =\delta_{m,n}\frac{(-1)^np^{\theta}(q;q)_{\theta-n}(q;q)_n(pq;q)_n}{[(q;q)_{\theta}]^2}q^{\binom{\theta+1}{2}-\binom{n+1}{2}+\theta n}.
\end{multline*}
Divide both sides by $(pq;q)_\infty$ and use (1.8.8) of  \cite{KLS}  to get 
\begin{multline*}
     \sum_{k=0}^{\theta}\frac{(-1)^{\theta-k}q^{\binom{k}{2}}}{(q;q)_k(q;q)_{\theta-k}(pq^{1+\theta-k};q)_\infty}K_m(q^{-k},p,\theta,q)K_n(q^{-k},p,\theta,q)\\
     =\delta_{m,n}\frac{(-1)^np^{\theta}(q;q)_{\theta-n}(q;q)_n}{[(q;q)_{\theta}]^2(pq^{1+n};q)_\infty}q^{\binom{\theta+1}{2}-\binom{n+1}{2}+\theta n}.
\end{multline*}
Define now
\begin{equation*}
    W(k,p,\theta,q)=\frac{(-1)^{\theta-k}q^{\binom{k}{2}}}{(q;q)_k(q)_{\theta-k}(pq^{1+\theta-k};q)_\infty},
\end{equation*}
\begin{equation*}
    H(n,p,\theta,q)=\frac{(-1)^np^{\theta}(q;q)_{\theta-n}(q;q)_n}{[(q;q)_{\theta}]^2(pq^{1+n};q)_\infty}q^{\binom{\theta+1}{2}-\binom{n+1}{2}+\theta n}.
\end{equation*}

Let 
\begin{equation*}
    K^{\boldsymbol\theta}(\boldsymbol\xi,\boldsymbol\eta)=\prod_{x=1}^LK_{\eta^x}\left(q^{-\xi^x},pq^{N_{x-1}^-(\boldsymbol\theta)}q^{-N_{x-1}^-(\boldsymbol\xi)+N_{x+1}^+(\boldsymbol\eta)},\theta^x,q\right)
\end{equation*}
and moreover define
\begin{equation*}
    W^{\boldsymbol\theta}_x(\boldsymbol\xi,\boldsymbol\eta)=W(\xi^x,pq^{N^{-}_{x-1}(\boldsymbol\theta)}q^{-N_{x-1}^-(\boldsymbol\xi)+N_{x+1}^+(\boldsymbol\eta)},\theta^x,q),
\end{equation*}
\begin{equation*}
   H^{\boldsymbol\theta}_x(\boldsymbol\xi,\boldsymbol\eta)=H(\eta^x,pq^{N^{-}_{x-1}(\boldsymbol\theta)}q^{-N_{x-1}^-(\boldsymbol\xi)+N_{x+1}^+(\boldsymbol\eta)},\theta^x,q),
\end{equation*}
\begin{equation*}
    K^{\boldsymbol\theta}_x(\boldsymbol\xi,\boldsymbol\eta)=K_{\eta^x}(\xi^x,pq^{N^{-}_{x-1}(\boldsymbol\theta)}q^{-N_{x-1}^-(\boldsymbol\xi)+N_{x+1}^+(\boldsymbol\eta)},\theta^x,q).
\end{equation*}

Next we show that  $\prod_{x=1}^L\frac{W_x(\boldsymbol\xi,\boldsymbol\eta)}{ H_x(\boldsymbol\xi,\boldsymbol\eta)}$  can be written as $\frac{w_{\boldsymbol\theta}(\boldsymbol\xi)}{h_{\boldsymbol\theta}(\boldsymbol\eta)}$.

\begin{align*} \prod_{x=1}^L\frac{W^{\boldsymbol\theta}_x(\boldsymbol\xi,\boldsymbol\eta)}{ H^{\boldsymbol\theta}_x(\boldsymbol\xi,\boldsymbol\eta)}=& \frac{(-1)^{N(\boldsymbol\theta)-N(\boldsymbol\xi)}\left(pq^{N(\boldsymbol\eta)+1};q\right)_\infty\prod_x[(q;q)_{\theta^x}]^{2}}{(-1)^{N(\boldsymbol\eta)}q^{\binom{N(\boldsymbol\theta)+1}{2}}p^{N(\boldsymbol\theta)}\left(pq^{N(\boldsymbol\theta) -N(\boldsymbol\xi)+1};q\right)_\infty }\\
&\times\prod_{x=1}^L\frac{q^{\binom{\xi^x}{2}}q^{\binom{\eta^x+1}{2}}q^{\theta^x N_{x-1}^-(\boldsymbol\xi)}}{(q;q)_{\xi^x}(q;q)_{\theta^x-\xi^x}(q;q)_{\eta^x}(q;q)_{\theta^x-\eta^x}q^{\theta^x N_{x}^+(\boldsymbol\eta)}}.  
\end{align*}

So now we let 
\begin{equation*}
    w_{\boldsymbol\theta}(\boldsymbol\xi)=\frac{(-1)^{N(\boldsymbol\theta)-N(\boldsymbol\xi)}}{\left(pq^{N(\boldsymbol\theta)-N(\boldsymbol\xi)+1};q\right)_\infty }\prod_{x=1}^L\frac{q^{\binom{\xi^x}{2}}q^{\theta^x N_{x-1}^-(\xi)}}{(q;q)_{\xi^x}(q;q)_{\theta^x-\xi^x}},
\end{equation*}
\begin{equation*}
    h_{\boldsymbol\theta}(\boldsymbol\eta)=\frac{(-1)^{N(\boldsymbol\eta)}q^{\binom{N(\boldsymbol\theta)+1}{2}}p^{N(\boldsymbol\theta)}}{\prod_x[(q;q)_{\theta^x}]^{2} \left(pq^{N(\boldsymbol\eta)+1}; q\right)_\infty} \prod_{x=1}^L\frac{(q;q)_{\eta^x}(q;q)_{\theta^x-\eta^x}}{q^{\binom{\eta^x+1}{2}}}q^{\theta^x N_{x}^+(\boldsymbol\eta)}.
\end{equation*}
Thus
$$
  \sum_{\boldsymbol\xi} \frac{w_{\boldsymbol\theta}(\boldsymbol\xi)}{h_{\boldsymbol\theta}(\boldsymbol\eta)} K^{\boldsymbol\theta}(\boldsymbol\xi,\boldsymbol\eta) K^{\boldsymbol\theta}(\boldsymbol\xi,\boldsymbol{\bar{\eta}}) = \prod_{x=1}^L
 \sum_{\xi^x=0}^{\theta^x} \frac{W^{\boldsymbol\theta}_x(\boldsymbol\xi,\boldsymbol\eta)}{H^{\boldsymbol\theta}_x(\boldsymbol\xi,\boldsymbol\eta)}K^{\boldsymbol\theta}_x(\boldsymbol\xi,\boldsymbol\eta) K^{\boldsymbol\theta}_x(\boldsymbol\xi,\bar{\boldsymbol\eta}).
$$

By the orthogonal relation
\begin{equation*}
    \sum_{\xi^x=0}^{\theta^x} \frac{W^{\boldsymbol\theta}_x(\boldsymbol\xi,\boldsymbol\eta)}{H^{\boldsymbol\theta}_x(\boldsymbol\xi,\boldsymbol\eta)}K^{\boldsymbol\theta}_x(\boldsymbol\xi,\boldsymbol\eta) K^{\boldsymbol\theta}_x(\boldsymbol\xi,\bar{\boldsymbol\eta})=\delta_{\eta^x,\bar{\eta}^x},
\end{equation*}
we get
$$
\sum_{\boldsymbol\xi} \frac{w_{\boldsymbol\theta}(\boldsymbol\xi)}{h_{\boldsymbol\theta}(\boldsymbol\eta)} K^{\boldsymbol\theta}(\boldsymbol\xi,\boldsymbol\eta) K^{\boldsymbol\theta}(\boldsymbol{\boldsymbol\xi,\bar{\eta}}) = \delta_{\boldsymbol\eta,\bar{\boldsymbol\eta}},
$$
or in other words, $K^{\boldsymbol\theta}(\boldsymbol\eta,\boldsymbol\xi)$ and $K^{\boldsymbol\theta}(\boldsymbol{\bar{\eta}},\boldsymbol\xi)$ are orthogonal with respect to $w_{\boldsymbol\theta}(\boldsymbol\xi)$ with squared norm $h_{\boldsymbol\theta}(\boldsymbol\eta)$.  

In addition, to recover the orthogonal relation appeared in \cite{CFG21}, we can rewrite $\frac{w_{\boldsymbol\theta}(\boldsymbol\xi)}{h_{\boldsymbol\theta}(\boldsymbol\eta)}$ as the following
\begin{equation*}
\begin{array}{cl}
   \frac{w_{\boldsymbol\theta}(\boldsymbol\xi)}{h_{\boldsymbol\theta}(\boldsymbol\eta)}  & =\frac{(-1)^{N(\boldsymbol\theta)-N(\boldsymbol\xi)}\left(pq^{N(\boldsymbol\eta)+1};q\right)_\infty\prod_x[(q;q)_{\theta^x}]^{2}}{(-1)^{N(\boldsymbol\eta)}q^{\binom{N(\boldsymbol\theta)+1}{2}}p^{N(\boldsymbol\theta)}\left(pq^{N(\boldsymbol\theta) -N(\boldsymbol\xi)+1}\right)_\infty } \\ & \quad\times
   \displaystyle\prod_{x=1}^L\frac{q^{\binom{\xi^x}{2}}q^{\binom{\eta^x+1}{2}}q^{\theta^x N_{x-1}^-(\boldsymbol\xi)}}{(q;q)_{\xi^x}(q;q)_{\theta^x-\xi^x}(q;q)_{\eta^x}(q;q)_{\theta^x-\eta^x}q^{\theta^x N_{x}^+(\boldsymbol\eta)}}\\ 
     &  =\frac{(p^{-1}q^{-N(\boldsymbol\theta-\boldsymbol\xi)};q)_\infty q^{\frac{N(\boldsymbol\xi)^2}{2}-\frac{N(\boldsymbol\eta)^2}{2}}}{(p^{-1}q^{-N(\boldsymbol\eta)};q)_\infty p^{N(\boldsymbol\xi)+N(\boldsymbol\eta)}}
     \\ & \quad\times\displaystyle{\prod_{x}}\binom{\theta^x}{\xi^x}_{q^{\frac{1}{2}}}\binom{\theta^x}{\eta^x}_{q^{\frac{1}{2}}}q^{-\left(N_{x-1}^-(\boldsymbol\theta)+\frac{1}{2}\theta^x\right)(\xi^x+\eta^x)} 
\end{array}
 \end{equation*}
where in the second equation we used the following identities:
\begin{equation*}
    \left(pq^{N(\boldsymbol\eta)+1};q\right)_\infty=(pq;q)_{\infty}(pq^{N(\boldsymbol\eta)+1};q)_{-N(\boldsymbol{\eta})}=
    \frac{ (pq;q)_{\infty} (-1)^{N(\boldsymbol\eta)}q^{-\binom{N(\boldsymbol\eta)+1}{2}}}{p^{N(\boldsymbol\eta)}(p^{-1}q^{-N(\boldsymbol\eta)};q)_{N(\boldsymbol\eta)}},
\end{equation*}
\begin{equation*}
    \frac{(p^{-1}q^{-N(\boldsymbol\theta-\boldsymbol\xi)};q)_{N(\boldsymbol\theta-\boldsymbol\xi)}}{(p^{-1}q^{-N(\boldsymbol\eta)};q)_{N(\boldsymbol\eta)}}= \frac{(p^{-1}q^{-N(\boldsymbol\theta-\boldsymbol\xi)};q)_{\infty}}{(p^{-1}q^{-N(\boldsymbol\eta)};q)_\infty}.
\end{equation*}

Note that if we substitute  $q$ by $q^2$ and let $p=\alpha^{-1}q^{-1}$ we get exactly \eqref{eq:4}, i.e. the orthogonal duality function for single species case.

\end{proof}
Now we work on the orthogonality of $\mathcal{D}^{\boldsymbol\theta}_{\boldsymbol{\alpha}}(\boldsymbol\xi, \boldsymbol\eta)$. Fix two particle configurations $\boldsymbol\eta$ and $\bar{\boldsymbol\eta}$. Let $\boldsymbol\xi$ be another particle configuration which is viewed as a variable, rather than a fixed configuration. As in section \ref{OPD}, define the intermediate particle configurations
$$
\boldsymbol\zeta^{(1)}, \ldots, \boldsymbol\zeta^{(n-1)}, \quad \quad \bar{\boldsymbol\zeta}^{(1)}, \ldots, \bar{\boldsymbol\zeta}^{(n-1)}.
$$

Since $\boldsymbol\xi_i$ is independent of $\boldsymbol\zeta^{(i)}_i$, thus sum from left to right $\boldsymbol\xi_0$ to $\boldsymbol\xi_{n-1}$, and  apply Lemma \ref{lemma:2},
\begin{equation*}
   \prod_{i=0}^{n-1} \sum_{\boldsymbol\xi_i}\frac{w_{\boldsymbol\theta^{(i)}}(\boldsymbol\xi_i)}{h^{1/2}_{\boldsymbol\theta^{(i)}}(\boldsymbol\zeta^{(i)}_i)h^{1/2}_{\boldsymbol\theta^{(i)}}(\bar{\boldsymbol\zeta}^{(i)}_i)} K^{\boldsymbol\theta^{(i)}}(\boldsymbol\xi_i, \boldsymbol\zeta^{(i)}_i)K^{\boldsymbol\theta^{(i)}}(\boldsymbol\xi_i, \bar{\boldsymbol\zeta}^{(i)}_i)=\delta_{\boldsymbol\eta,\bar{\boldsymbol\eta}},
\end{equation*}
Notice that  $\boldsymbol\theta^{(i)}=\bar{\boldsymbol\theta}^{(i)}$ in the summation of $\boldsymbol\xi_{i}$, thus we have 
\begin{equation*}
   \sum_{\boldsymbol\xi} \prod_{i=0}^{n-1}\sqrt{\frac{w_{\boldsymbol\theta^{(i)}}(\boldsymbol\xi_i)}{h_{\boldsymbol\theta^{(i)}}(\boldsymbol\zeta^{(i)}_i)}}K^{\boldsymbol\theta^{(i)}}(\boldsymbol\xi_i, \boldsymbol\zeta^{(i)}_i)\sqrt{\frac{w_{\bar{\boldsymbol\theta}^{(i)}}(\boldsymbol\xi_i)}{h_{\bar{\boldsymbol\theta}^{(i)}}(\bar{\boldsymbol\zeta}^{(i)}_i)}}K^{\bar{\boldsymbol\theta}^{(i)}}(\boldsymbol\xi_i, \bar{\boldsymbol\zeta}^{(i)}_i)=\delta_{\boldsymbol\eta,\bar{\boldsymbol\eta}},
\end{equation*}
or, in other words, $\prod_{i=0}^{n-1}\sqrt{\frac{w_{\boldsymbol\theta^{(i)}}(\boldsymbol\xi_i)}{h_{\boldsymbol\theta^{(i)}}(\boldsymbol\zeta^{(i)}_i)}}K^{\boldsymbol\theta^{(i)}}(\boldsymbol\xi_i, \boldsymbol\zeta^{(i)}_i) $ is orthonormal w.r.t standard inner product. \newline
Thus far, the parameter $\boldsymbol\alpha$ has not come up. Now, if we substitute  $q$ by $q^2$, let $p_i=\alpha_i^{-1}q^{-1}$ and divide by $\sqrt{\mu^{n}(\boldsymbol\xi) \mu^{n}(\boldsymbol\eta)}$ , up to a constant, we get  $\mathcal{D}^{\boldsymbol\theta}_{\boldsymbol{\alpha}}(\boldsymbol\xi, \boldsymbol\eta)$, which is the orthogonal duality function for multi-species case.  More precisely, now $w$ and $h$ depend on $\boldsymbol\alpha$ and we write
\begin{equation}
   \sum_{\boldsymbol\xi} \prod_{i=0}^{n-1}\sqrt{\frac{w_{\boldsymbol\theta^{(i)}}(\boldsymbol\xi_i)}{h_{\boldsymbol\theta^{(i)}}(\boldsymbol\zeta^{(i)}_i)}}  \frac{K^{\boldsymbol\theta^{(i)}}(\boldsymbol\xi_i, \boldsymbol\zeta^{(i)}_i)}{\sqrt{\mu^n(\boldsymbol\xi)\mu^n(\boldsymbol\eta)}} \sqrt{\frac{w_{\bar{\boldsymbol\theta}^{(i)}}(\boldsymbol\xi_i)}{h_{\bar{\boldsymbol\theta}^{(i)}}(\bar{\boldsymbol\zeta}^{(i)}_i)}} \frac{K^{\bar{\boldsymbol\theta}^{(i)}}(\boldsymbol\xi_i, \bar{\boldsymbol\zeta}^{(i)}_i)}{\sqrt{\mu^n(\boldsymbol\xi)\mu^n(\bar{\boldsymbol\eta})}} \mu^n(\boldsymbol\xi) = \frac{\delta_{\boldsymbol\eta,\bar{\boldsymbol\eta}}}{\mu^n(\boldsymbol\eta)}.
\end{equation}
By a direct, but lengthy calculation, 
$$
\prod_{i=0}^{n-1}\sqrt{\frac{w_{\boldsymbol\theta^{(i)}}(\boldsymbol\xi_i)}{h_{\boldsymbol\theta^{(i)}}(\boldsymbol\zeta^{(i)}_i)}}  = \mathcal{C}_{\boldsymbol\alpha}(\boldsymbol\xi,\boldsymbol\eta)  \sqrt{\prod_{i=0}^{n-1}\mu_{\alpha_i}^{\boldsymbol{\theta}^{(i)}}(\boldsymbol\xi_i)\mu_{\alpha_i}^{\boldsymbol{\theta}^{(i)}}(\boldsymbol\zeta^{(i)}_i)}
$$
and similarly for $\mathcal{C}_{\boldsymbol\alpha}(\boldsymbol\xi,\bar{\boldsymbol\eta}).$ Since
$$
\mathcal{D}^{\boldsymbol\theta}_{\boldsymbol{\alpha}}(\boldsymbol\xi, \boldsymbol\eta) =\mathcal{G}_{\boldsymbol{\alpha}}(\boldsymbol\xi,\boldsymbol\eta)\prod_{i=0}^{n-1}D^{  \boldsymbol{\theta}^{(i)} }_{\alpha_i}(\boldsymbol\xi_i, \boldsymbol\zeta^{(i)}_i)
 ,
$$
where 
$$
\mathcal{G}_{\boldsymbol{\alpha}}(\boldsymbol\xi,\boldsymbol\eta) = \sqrt{ \frac{ \prod_{i=0}^{n-1}\mu_{\alpha_i}^{\boldsymbol{\theta}^{(i)}}(\boldsymbol\xi_i)\mu_{\alpha_i}^{\boldsymbol{\theta}^{(i)}}(\boldsymbol\zeta^{(i)}_i)  }{\mu^{n}(\boldsymbol\xi) \mu^{n}(\boldsymbol\eta) }},$$
this completes the proof.

\subsection{Previously Known Results}

\subsubsection{A general algebraic set up for duality}
The multi--species ASEP$(q,\boldsymbol\theta)$ was constructed in \cite{Kuan-IMRN}, generalizing the single--species ASEP$(q,\theta)$ in \cite{CGRS}. In this process, up to $\theta$ particles can occupy each site, and the particles have a drift with asymmetry parameter $q^{\theta}$. The generator was constructed using the Casimir element from \cite{REU2020}. An inhomogeneous version was constructed in \cite{KuanSF}, where the maximum number of particles per site can vary. 

The set-up involves the following ingredients see Section 3.5 of \cite{Kuan-IMRN} or Proposition 2.1 of \cite{CGRS} for more details:

\begin{itemize}
\item
A Hamiltonian $H$, and a symmetry $S$ as in Definition \ref{symm}.
\item
A change-of-basis diagonal matrix $B$, such that $B^{-1}HB$ is self--adjoint.
\item
A ground state transformation $G$, which is a diagonal matrix such that $L:=G^{-1}HG$ is the generator of a Markov process and $\tilde{S} = G^{-1}SG$ is a symmetry for $L$. 
\end{itemize}

The result is then that $D:= \tilde{S}G^2B^{-2} =  G^{-1}SG^{-1}B^2$ is a self--duality function, and $G^2B^{-2}$ are reversible measures for the multi--species ASEP$(q,\boldsymbol\theta)$. More specifically, 
\begin{equation}\label{GB}
G^2B^{-2}(\boldsymbol\xi,\boldsymbol\xi) = \sum_{\boldsymbol k} \frac{\mu_{\boldsymbol k}^n(\boldsymbol\xi)}{Z_{n,\boldsymbol k}},
\end{equation}
for some normalizing constants $Z_{n,\boldsymbol k}$. The matrices $G$ and $B$ are not unique, but any differences can be absorbed into the normalizing constants. 

\subsection{Preliminary Lemmas}

 In order for the symmetries to be called unitary, they need to be unitary with respect to an inner product on the representation. The next lemma explicitly defines such an inner product.

\begin{lemma}\label{FirstLemma}
(a)
Define an inner product on $V_m^{(n)}$ by 
$$
\langle v_{\mu},v_{\nu}\rangle = 1_{\{\mu=\nu\}} q^{O(\mu)} \frac{[\mu_0]_q^! \cdots [\mu_n]_q^!}{[m]_q^!},
$$
where
$$
  O(\mu) = -( \mu_1 + 2\mu_2 + \ldots + n \mu_n )+ \frac{\mu_0^2}{2} + \frac{\mu_1^2}{2} + \cdots + \frac{\mu_n^2}{2}.
$$
Then $V_m^{(n)}$ is a $*$--representation of $\mathcal{U}_q(\mathfrak{gl}_{n+1})$ with action given by \eqref{Act}. 

(b) An equivalent formula  for the inner product is
$$
\langle v_{\mu},v_{\nu} \rangle = \mathrm{const} \cdot 1_{\{\mu=\nu\}} q^{(\rho,\mu)} \{\mu_0\}_{q^2}^! \cdots \{\mu_n\}_{q^2}^!,
$$
where $\rho$ is half the sum of the positive roots and $(\cdot,\cdot)$ is the Killing form.

(c) For any $X\in \mathcal{U}_q(\mathfrak{gl}_{n+1})$, let $M_X$ be the matrix of the action of $X$ on $V_m^{(n)}$ with respect to the basis $\{v_{\mu}: \mu \in \mathcal{B}_m^{(n)}\}$. Then
$$
B^2 M_X^T B^{-2}   = M_{X^*},
$$
where $B$ is the diagonal matrix with entries $$B(\lambda,\mu) = 1_{\{\lambda=\mu\}} \frac{1}{\sqrt{\{\mu_0\}_{q^2}^! \cdots \{\mu_n\}_{q^2}^!}}.$$
In particular, if $U$ is unitary i.e. $U^*=U^{-1}$, then
$$
B^2 M_U^T B^{-2}   = M_{U}^{-1}.
$$
\end{lemma}
\begin{proof}
(a) This part of the result has possibly appeared before, but the authors were unable to find it in the literature. In any case, it is not difficult to prove anyway, so for completeness the proof is presented here.

Recall that the $*$ involution is given on generators by  
$$
E_{i,i+1}^* = E_{i+1,i}q^{E_{ii}-E_{i+1,i+1}} , \quad E_{i+1,i}^*= q^{-E_{ii}+E_{i+1,i+1}}E_{i,i+1}, \quad (q^{E_{ii}})^* = q^{E_{ii}}.
$$
So one checks that
\begin{align*}
\langle E_{i,i+1}v_{\mu},v_{\nu}\rangle &= [\mu_{i+1}]_q\langle v_{\mu+\epsilon_i},v_{\nu}\rangle \\
&= 1_{\{\mu+\epsilon_i=\nu\}} q^{O(\nu)}  [\mu_{i+1}]_q \frac{[\nu_0]_q^! \cdots [\nu_n]_q^!}{[m]_q^!} \\
&= 1_{\{\mu+\epsilon_i=\nu\}}  [\mu_i+1]_q \cdot  q^{O(\nu)}  \frac{[\mu_0]_q^! \cdots [\mu_n]_q^!}{[m]_q^!}
\end{align*}
while
\begin{align*}
\langle v_{\mu}, E_{i+1,i}q^{E_{ii}-E_{i+1,i+1}} v_{\nu}\rangle &= [\nu_i]_q q^{\nu_i-\nu_{i+1}} \langle v_{\mu},v_{\nu-\epsilon_i}\rangle\\
&=1_{\{\mu=\nu-\epsilon_i\}} [\nu_i]_q q^{\nu_i-\nu_{i+1}} q^{O(\mu)}  \frac{[\mu_0]_q^! \cdots [\mu_n]_q^!}{[m]_q^!}. 
\end{align*}
Now observe that when $\nu=\mu+\epsilon_i$, we have $\nu_i=\mu_i+1$ and 
\begin{align*} \nu_i-\nu_{i+1} +O(\mu) &= \left( \sum_{k\neq i,i+1} \left(-k\mu_k + \frac{\mu_k^2}{2}\right) \right)\\&\quad + \mu_i+1-\mu_{i+1}+1 -  i\mu_i - (i+1)\mu_{i+1} + \frac{\mu_i^2}{2} + \frac{\mu_{i+1}^2}{2} \\
&= \left( \sum_{k\neq i,i+1} \left(-k\nu_k + \frac{\nu_k^2}{2} \right)\right)\\&\quad + 1 - i(\nu_i-1) - (i+1)(\nu_{i+1}+1) + \frac{(\mu_i+1)^2}{2} + \frac{(\mu_{i+1}-1)^2}{2} \\
&=   O(\nu).
\end{align*}
Thus $\langle E_{i,i+1}v_{\mu},v_{\nu}\rangle = \langle v_{\mu}, E_{i,i+1}^*v_{\nu}\rangle$ for all $\mu,\nu \in \mathcal{B}_m^{(n)}$. One can similarly verify that for $E_{i+1,i}$ and $E_{i,i}$.

(b) The $[m]_q^!$ in the denominator can be removed, but by leaving it in the inner product it can be written as the multinomial
$$
\langle v_{\mu},v_{\nu}\rangle = 1_{\{\mu=\nu\}} \frac{q^{O(\mu)}} { \binom{m}{\mu_0,\ldots,\mu_n}_q}.
$$
Letting $\rho$ be half the sum of the positive roots, we have
$$
(\rho -\frac{n}{2}I ,\mu) = -\mu_1 - 2\mu_2 - \ldots -(n-1)\mu_{n-1} -n\mu_n, 
$$
where $(\cdot,\cdot)$ is the Killing form. This shows the equivalence of the two formulas.

(c) By definition, 
$$
\langle X v_{\mu}, v_{\lambda} \rangle = \langle v_{\mu},X^*v_{\lambda} \rangle
$$
so therefore
$$
M_X(\lambda,\mu) \langle v_{\lambda},v_{\lambda} \rangle = M_{X^*}(\mu,\lambda) \langle v_{\lambda}, v_{\lambda} \rangle
$$
implying the result. The statement when $U$ is unitary is immediate. 

\end{proof}

This paper uses an inhomogeneous multi--species ASEP$(q,\boldsymbol\theta)$, where the maximum number of particles at a site may vary. The next proposition establishes that the algebraic construction of \cite{Kuan-IMRN} still carries over in the inhomogeneous case. Compared to \cite{Kuan-IMRN}, the argument below is slightly streamlined by using the $*$--algebra structure. 

\begin{prop}
Consider the representation 
$$
V_{\theta_1}^{(n)} \otimes \cdots \otimes V_{\theta_L}^{(n)}
$$
of the quantum group $\mathcal{U}_q(\mathfrak{gl}_{n+1})$. 
The generator of the multi--species ASEP$(q,\boldsymbol\theta)$ can be written in terms of the $L-1$ co--product of the central elements $C_1$ 

$$
G^{-1}\Delta^{L-1}(C_1)G,
$$
where $G$ is the diagonal map with entries
\begin{equation*}
\prod_{i=0}^{n} \prod_{x=1}^{L} \frac{1}{\left[\eta_{i}^{x}\right]_{q}^{!}} \times \prod_{1 \leq y<x \leq L} \prod_{i=0}^{n-1} q^{-\eta_{i+1}^{y} \eta_{[1, i]}^{x}} .
\end{equation*}
\end{prop}
\begin{proof}
The proofs of Propositions 4.1 and 4.2 of \cite{Kuan-IMRN} generalize to the inhomogeneous case, which would imply the proposition here. By using the $*$--algebra structure, the proof presented below is slightly more streamlined than the proofs in \cite{Kuan-IMRN}. The reader does not need to refer to \cite{Kuan-IMRN} for the proof below.

Fix $0 \leq i < j \leq n$. The claim is that
\begin{multline}
h^{x, x+1}\left(\left(\mu_{i \rightarrow j}, \lambda_{j \rightarrow i}\right),(\mu, \lambda)\right)  
=q^{-2} q^{\mu_{i}+\mu_{i+1}+\cdots+\mu_{j-1}}\left\{\mu_{i}\right\}_{q^{2}}\\ \times q^{\lambda_{i+1}+\cdots+\lambda_{j}}\left\{\lambda_{j}\right\}_{q^{2}} q^{2\left(\lambda_{j+1}+\ldots+\lambda_{n}\right)} q^{2\left(\mu_{1}+\ldots+\mu_{i-1}\right)} \label{LeftJumps}
\end{multline}
and
\begin{multline}
    h^{x, x+1}\left(\left(\mu_{j \rightarrow i}, \lambda_{i \rightarrow j}\right),(\mu, \lambda)\right)=q^{\mu_{i}+\ldots+\mu_{j-1}}\left\{\mu_{j}\right\}_{q^{2}} q^{\lambda_{i+1}+\ldots+\lambda_{j}}\left\{\lambda_{i}\right\}_{q^{2}}\\ \times q^{2\left(\lambda_{j+1}+\ldots+\lambda_{n}\right)} q^{2\left(\mu_{1}+\ldots+\mu_{i-1}\right)}. \label{RightJumps}
\end{multline}

First start with the equation \eqref{LeftJumps}, which is the equation for left jumps. When applying the co--product to the central element $C_1$, the number of off--diagonal terms equals
\begin{equation*}
\sum_{1 \leq i<j \leq n}\left((j-i+1)^{2}-(j-i+1)\right)=\sum_{1 \leq i<j \leq n}(2(i-1)(n-j)+2+2(n-j)+2(i-1)).
\end{equation*}
Each of the four summands on the right--hand--side corresponds to four expressions of the form
\begin{equation}
E_{r i} E_{j r} \otimes E_{i s} E_{s j}, \quad E_{ji} \otimes E_{ij}, \quad  E_{ri} E_{jr} \otimes E_{ij}, \quad E_{ij} \otimes E_{is} E_{sj}.
\end{equation}
Therefore the term that needs to be calculated can be expressed as a sum of four types of terms:
\begin{multline*}
q^{2 j+2} q^{2 E_{i i}} E_{j i} \otimes q^{2 E_{j j}} E_{i j}+\left(q-q^{-1}\right) \sum_{s=j+1}^{n} q^{2 s+1} q^{2 E_{i i}} E_{j i} \otimes q^{E_{s s}} E_{i s} q^{E_{j j}} E_{s j} \\ +\left(q-q^{-1}\right) q^{2 j+2} \sum_{r=1}^{i-1} q^{E_{i i}} E_{r i} q^{E_{r r}} E_{j r} \otimes q^{2 E_{j j}} E_{i j}\\
+\left(q-q^{-1}\right)^{2} \sum_{r=1}^{i-1} \sum_{s=j+1}^{n} q^{2 s+1} q^{E_{i i}} E_{r i} q^{E_{r r}} E_{j r} \otimes q^{E_{s s}} E_{i s} q^{E_{j j}} E_{s j}.
\end{multline*}
At this point, the right jumps \eqref{RightJumps} can be found with a similar calculation, which is how \cite{Kuan-IMRN} proceeded with the proof. Actually, with the $*$--bialgebra structure, the right jumps can be found much more quickly.  Since  (see Remark \ref{CentralExplanation}) $\Delta(C_1) = \Delta(C_1^*)=\Delta(C_1)^*$, this implies that
$$
\langle \Delta(C_1) (v_{\mu_{ i \rightarrow j}} \otimes v_{\lambda_{j \rightarrow i}}), v_{\mu}\otimes v_{\lambda} \rangle = \langle v_{\mu_{i \rightarrow j}} \otimes v_{\lambda_{j \rightarrow i}}, \Delta(C_1)(v_{\mu} \otimes v_{\lambda})\rangle ,
$$
which means that
\begin{align*}
&h^{x, x+1}\left((\mu, \lambda), \left(\mu_{i \rightarrow j}, \lambda_{j \rightarrow i}\right)\right) {   \Vert v_{\mu}\otimes v_{\lambda}\Vert^2}  \\& = h^{x, x+1}\left(\left(\mu_{i \rightarrow j}, \lambda_{j \rightarrow i}\right),(\mu, \lambda)\right) { \Vert v_{\mu_{i\rightarrow j}} \otimes v_{\lambda_{j \rightarrow i}} \Vert^2},  
\end{align*}
so thus
\begin{multline*}
h^{x, x+1}\left((\mu, \lambda), \left(\mu_{j \rightarrow i}, \lambda_{i \rightarrow j}\right)\right) 
= q^{-2} q^{\mu_{i}+\mu_{i+1}+\cdots+\mu_{j-1}}\left\{\mu_{i}\right\}_{q^{2}} q^{\lambda_{i+1}+\cdots+\lambda_{j}}\left\{\lambda_{j}\right\}_{q^{2}} \\ \times q^{2\left(\lambda_{j+1}+\ldots+\lambda_{n}\right)}  q^{2\left(\mu_{1}+\ldots+\mu_{i-1}\right)} 
 \frac{\{\mu_{j}+1\}_{q^2}}{  \{\mu_i\}_{q^2}  }    \frac{ \{\lambda_i +1\}_{q^2}}{ \{\lambda_j\}_{q^2}} .
\end{multline*}
Setting $\hat{\mu}= \mu_{j\rightarrow i}, \hat{\lambda}=\lambda_{i\rightarrow j}$, so that $\mu_i = \hat{\mu}_i+1,\lambda_j = \hat{\lambda}_j+1$ shows that \eqref{RightJumps} holds.

The remaining steps with the ground state transformation $G$ are identical to \cite{Kuan-IMRN}, and will not be repeated here.

\end{proof}

\subsection{Identifying the unitary symmetry from the $*$--bialgebra structure }
Given the one-to-one link between self-duality functions and the corresponding symmetries as shown in Section \ref{markovduality}, it is natural to ask which is the symmetry associated to the orthogonal self-duality function of the multi--species ASEP$(q,\boldsymbol\theta)$. Since the aimed symmetry, introduced in the next lemma, is associated to an orthogonal polynomial then we know it must be a unitary operator as the norm of the cheap self-duality function is preserved. 
For any $X ,Y \in \mathcal{U}_q(\mathfrak{gl}_{n+1})$, 
$$
\left(e_{q^2}(X)\mathcal{E}_{q^2}(Y)\right)^{-1} = e_{q^2}(-Y)\mathcal{E}_{q^2}(-X),
$$
because the $e$ and the  $\mathcal{E}$ are inverses (see \eqref{Inverses}), while
$$
\left(e_{q^2}(X)\mathcal{E}_{q^2}(Y)\right)^{*} = \mathcal{E}_{q^2}(Y^*)e_{q^2}(X^*),
$$
therefore
\begin{align*}
\left(e_{q^2}(X)\mathcal{E}_{q^2}(Y)\right)^{-1} & = \left(e_{q^2}(X)\mathcal{E}_{q^2}(Y)\right)^{*}  \\ & \Longleftrightarrow e_{q^2}(-Y^*) e_{q^2}(-Y) =  e_{q^2}(X^*)e_{q^2}(X).
\end{align*}

\begin{lemma}
Define\footnote{Here, $e_{q^{1/2}}^{1/2}(\ldots)$ refers to an element whose square is $e_{q^{1/2}}(\ldots),$ which is well--defined as long as $(\ldots)$ is in the subalgebra generated by $\{q^{\pm E_{ii}}\}_i$, since this subalgebra acts as diagonal elements on any finite--dimensional module. This element also be defined to be in the quantum group $\mathcal{U}[[h]]$, which is the completion of $\mathcal{U}$ in the $h$--adic topology, but for reasons of space we do not define that object in this paper. } for all $i=0, \ldots, n-1$ the element $U_i(\lambda)$ by 
\begin{multline*}
U_i(\lambda):=
 {e}_{q^2}^{1/2}\left( -\gamma \lambda q^{2E_{ii}}\right) 
{e}_{q^2} \left(\gamma(1-q^2)(q-q^{-1})E_{i+1,i}q^{E_{ii}}\right) \\
\times
\mathcal{E}_{q^2}(-\lambda q^{E_{i+1,i+1}}E_{i,i+1}) 
 \mathcal{E}_{q^2}^{1 / 2}\left( \gamma \lambda(1-q^2){q^{2E_{i+1,i+1}}}  \right)   ,
\end{multline*}
where $\lambda = \gamma(1-q^2)(q-q^{-1})$.
Then $U_i(\lambda)$ is unitary, in other words $U_i(\lambda)^*=U_i(\lambda)^{-1}$.

\end{lemma}
\begin{proof}
This is modified from the proof in \cite{GPVYZ16} to avoid $q$--oscillator algebras and be more applicable to other semi-simple Lie algebras.

Substituting the value of $\lambda$, we have that 
\begin{multline*}
U_i^*=\mathcal{E}_{q^2}^{1 / 2}\left( \gamma \lambda(1-q^2){q^{2E_{i+1,i+1}}}  \right)    \mathcal{E}_{q^2} \left(-\lambda q^{E_{i+1,i+1}}E_{i,i+1}    \right)  ^*  \\\times
e_{q^2}(\lambda E_{i+1,i} q^{E_{ii}}  ) ^*
  e_{q^2}^{1/2}\left( -\gamma \lambda q^{2E_{ii}}\right) 
\end{multline*}
and
\begin{multline*}
U_i^{-1}= e_{q^2}^{1 / 2}\left(- \gamma \lambda(1-q^2){q^{2E_{i+1,i+1}}}  \right)   {e}_{q^2}(\lambda q^{E_{i+1,i+1}}E_{i,i+1})  \\
\times\mathcal{E}_{q^2}
\left(-\gamma(1-q^2)(q-q^{-1})E_{i+1,i}q^{E_{ii}}  \right) \mathcal{E}_{q^2}^{1/2}\left( \gamma \lambda q^{2E_{ii}}\right) .
\end{multline*}
Therefore the equality $U^*=U^{-1}$ is equivalent to 
\begin{align*}
& \mathcal{E}_{q^2} \left(-\gamma(1-q^2)(q-q^{-1}) q^{E_{i+1,i+1}}E_{i,i+1}    \right)  ^*   e_{q^2}(\lambda E_{i+1,i} q^{E_{ii}}  ) ^*
e_{q^2}\left( -\gamma \lambda q^{2E_{ii}}\right)
\\ &
  =e_{q^2}\left(- \gamma \lambda(1-q^2){q^{2E_{i+1,i+1}}} \right)  {e}_{q^2}(\lambda q^{E_{i+1,i+1}}E_{i,i+1}) \\
 &\quad\times\mathcal{E}_{q^2}
\left(-\gamma(1-q^2)(q-q^{-1})E_{i+1,i}q^{E_{ii}}\right).
\end{align*}
By the $*$--structure on $\mathcal{U}_q(\mathfrak{gl}_{n+1})$, this is equivalent to 
\begin{align*}
& \mathcal{E}_{q^2} \left(-\gamma(1-q^2)(q-q^{-1}) E_{i+1,i} q^{E_{ii}}  \right)     e_{q^2}(\lambda q^{E_{i+1,i+1}}E_{i,i+1}) 
e_{q^2}\left( -\gamma \lambda q^{2E_{ii}}\right)
\\
  & = e_{q^2}\left(- \gamma \lambda(1-q^2){q^{2E_{i+1,i+1}}}) \right){e}_{q^2}(\lambda q^{E_{i+1,i+1}}E_{i,i+1}) \\&\quad\times\mathcal{E}_{q^2}
\left(-\gamma(1-q^2)(q-q^{-1})E_{i+1,i}q^{E_{ii}}\right).
\end{align*}
Rewriting the identity solely in terms of the $q$--exponentials $e_{q^2}$, it suffices to prove
\begin{align*}
&e_{q^2}(\lambda q^{E_{i+1,i+1}}E_{i,i+1}) 
e_{q^2}\left( -\gamma \lambda q^{2E_{ii}}\right) 
e_{q^2}
\left(\gamma(1-q^2)(q-q^{-1})E_{i+1,i}q^{E_{ii}}\right)  \\
& =e_{q^2} \left(\gamma(1-q^2)(q-q^{-1}) E_{i+1,i} q^{E_{ii}}  \right)e_{q^2}\left(- \gamma \lambda(1-q^2){q^{2E_{i+1,i+1}}}) \right)\\
&\quad\times{e}_{q^2}(\lambda q^{E_{i+1,i+1}}E_{i,i+1}).
\end{align*}
As pointed out in section 2.3 of \cite{GPVYZ16}, the $q$--BCH equations imply
\begin{align*}
e_{q}\left(\alpha A_{-} B_{+}\right) e_{q}\left(\frac{\alpha \beta}{(1-q)^{2}} q^{B_{0}}\right) e_{q}\left(\beta A_{+} B_{-}\right) =   \\ e_{q}\left(\beta A_{+} B_{-}\right) e_{q}\left(\frac{\alpha \beta}{(1-q)^{2}} q^{A_{0}}\right) e_{q}\left(\alpha A_{-} B_{+}\right)
\end{align*}
for the little $q-$exponential, while for the  big $q-$exponential
\begin{align*}
\mathcal{E}_{q}\left(\gamma A_{+} B_{-}\right) \mathcal{E}_{q}\left(-\frac{\gamma \delta}{(1-q)^{2}} q^{B_{0}}\right) \mathcal{E}_{q}\left(\delta A_{-} B_{+}\right) =  \\  \mathcal{E}_{q}\left(\delta A_{-} B_{+}\right) \mathcal{E}_{q}\left(-\frac{\gamma \delta}{(1-q)^{2}} q^{A_{0}}\right) \mathcal{E}_{q}\left(\gamma A_{+} B_{-}\right)\
\end{align*}
for any $\alpha,\beta,\gamma,\delta$. Above $\left\lbrace A_0, A_-, A_+ \right\rbrace $ and $\left\lbrace B_0, B_-, B_+ \right\rbrace $ are two mutually commuting sets of generators for the $q$-oscillator algebra. Their proof uses the Schwinger model for $\mathcal{U}_q(\mathfrak{sl}_2)$, where the generators are written in terms of two copies of the $q$--oscillator algebra. In order to avoid needlessly introducing additional algebraic structure concerning the $q$--oscillator algebra, and because this paper uses a different $*$--structure, the calculation will be repeated here.
We have that
\begin{align}\label{FirstEqu}
  & e_{q^2}(\lambda q^{E_{i+1,i+1}}E_{i,i+1}) E_{i+1,i} q^{E_{ii}} \mathcal{E}_{q^2}(-\lambda q^{E_{i+1,i+1}}E_{i,i+1})  \nonumber \\
  &  =E_{i+1,i} q^{E_{ii}}  - \lambda \frac{q^{2E_{i+1,i+1}}}{q-q^{-1}}    
+ \sum_{n=0}^{\infty} \lambda^{n+1}q^{n(n-5)/2  } \frac{{   q^{nE_{i+1,i+1}} q^{2E_{ii}} E_{i,i+1}^{n}  }  }   {(1-q^2)(q-q^{-1})}, 
\end{align}
\begin{align}
 & e_{q^2}(\lambda q^{E_{i+1,i+1}}E_{i,i+1}) q^{2E_{ii}}  \mathcal{E}_{q^2}(-\lambda q^{E_{i+1,i+1}}E_{i,i+1})    \nonumber \\ &=\sum_{n=0}^{\infty} \lambda^n q^{n(n-5)/2  } {   q^{nE_{i+1,i+1}} q^{2E_{ii}} E_{i,i+1}^{n}  } \label{SecondEqu}.  
\end{align}

To see equation \eqref{FirstEqu}, first calculate, using \eqref{SecondEq} with $X=E_{i+1,i}$ and $Y=E_{i,i+1}$,
\begin{align*}
[q^{E_{i+1,i+1}}E_{i,i+1}   ,E_{i+1,i} q^{E_{ii}} ]'_0 &= E_{i+1,i} q^{E_{ii}}  ,\\
[q^{E_{i+1,i+1}}E_{i,i+1},E_{i+1,i} q^{E_{ii}} ]'_1 &=  \frac{q^{2E_{ii}}-q^{2E_{i+1,i+1}}}{q-q^{-1}}, 
\end{align*}
\begin{multline*}
  [q^{E_{i+1,i+1}}E_{i,i+1},E_{i+1,i} q^{E_{ii}} ]'_n
  \\=  \frac{   q^{(n-1)E_{i+1,i+1}} q^{2E_{ii}} E_{i,i+1}^{n-1}(q^{-2}-q^2)(q^{-1}-q^5) \cdots (q^{n-4}-q^{3n-4})}{q-q^{-1}}, \\
=  \frac{   q^{-2-1+0+1+\ldots+(n-4)}q^{(n-1)E_{i+1,i+1}} q^{2E_{ii}} E_{i,i+1}^{n-1} (q^2;q^2)_{n} } {(1-q^2)(q-q^{-1})}  , \quad n \geq 2.  
\end{multline*}
Since $-2+(-1)+0+1+\ldots+(n-3)=(n-2)(n-3)/2-3 = n^2/2-5n/2$, therefore \eqref{FirstEqu} holds. To see  \eqref{SecondEqu}, again calculate 
\begin{align*}
[q^{E_{i+1,i+1}}E_{i,i+1},  q^{2E_{ii}}   ]'_0 &= q^{2E_{ii}} ,\\
[q^{E_{i+1,i+1}}E_{i,i+1}, q^{2E_{ii}}]_1' &= (q^{-2}-1)q^{E_{i+1,i+1}}q^{2E_{ii}}E_{i,i+1}  , 
\end{align*}
and when $n\ge 2$
\begin{align*}
[q^{E_{i+1,i+1}}E_{i,i+1},  q^{2E_{ii}}]_n'&=  (q^{-2}-1)(q^{-1} - q^3  )(1-q^6)\cdots (q^{n-3}-q^{3n-3})  \\ & \quad\times q^{nE_{i+1,i+1}}q^{2E_{ii}}E_{i,i+1}^n \\
&=q^{-2}q^{-1}q^0\cdots q^{n-3} (q^2;q^2)_n q^{nE_{i+1,i+1}}q^{2E_{ii}}E_{i,i+1}^n.
\end{align*}
Note that \eqref{SecondEqu}  can also be calculated by first noting that
\begin{align*}
[q^{E_{i+1,i+1}}E_{i,i+1},  q^{-2E_{ii}}   ]'_0 &= q^{-2E_{ii}} ,\\
[q^{E_{i+1,i+1}}E_{i,i+1}, q^{-2E_{ii}}]_1' &= (q^{2}-1)q^{E_{i+1,i+1}}q^{-2E_{ii}}E_{i,i+1}  ,  \\
[q^{E_{i+1,i+1}}E_{i,i+1},  q^{2E_{ii}}]_n'&=  0, \quad  \quad n \geq 2.
\end{align*}
and then recalling that the symmetry is in fact an automorphism, results in the formal series
\begin{align*}
 &\left( e_{q^2}(\lambda q^{E_{i+1,i+1}}E_{i,i+1})q^{-2E_{ii}}  \mathcal{E}_{q^2}(-\lambda q^{E_{i+1,i+1}}E_{i,i+1}) \right)^{-1}  \\
 &= \left(q^{-2E_{ii}}(1 -\lambda q^{E_{i+1,i+1}}E_{i,i+1})\right)^{-1}\\
&= \sum_{n=0}^{\infty}\left( \lambda q^{E_{i+1,i+1}}E_{i,i+1}\right)^nq^{2E_{ii}} \\
&= \sum_{n=0}^{\infty} \lambda^nq^{2E_{ii}} q^{-2n} q^{n(n-1)/2}  q^{nE_{i+1,i+1}}E_{i,i+1}^n.   
\end{align*}

So combining \eqref{FirstEqu} and \eqref{SecondEqu}, and applying them to the element 
\newline 
$(1-q^2)(q-q^{-1})E_{i+1,i}q^{E_{ii}} - \lambda q^{2E_{ii}}$, all $\lambda^n$ terms cancel for $n\geq 2$, leaving only:
\begin{multline*}
e_{q^2}(\lambda q^{E_{i+1,i+1}}E_{i,i+1}) \left((1-q^2)(q-q^{-1})E_{i+1,i}q^{E_{ii}} - \lambda q^{2E_{ii}}\right) \\
\times\mathcal{E}_{q^2}(\lambda q^{E_{i+1,i+1}}E_{i,i+1}) 
=  \left(E_{i+1,i} q^{E_{ii}}  - \lambda\frac{q^{2E_{i+1,i+1}}}{q-q^{-1}}  \right) (1-q^2)(q-q^{-1}),
\end{multline*}
implying that
\begin{multline*}
e_{q^2}(\lambda q^{E_{i+1,i+1}}E_{i,i+1}) 
e_{q^2}
\left(\gamma(1-q^2)(q-q^{-1})E_{i+1,i}q^{E_{ii}} -\gamma \lambda q^{2E_{ii}}\right) 
\\
\times\mathcal{E}_{q^2}(\lambda q^{E_{i+1,i+1}}E_{i,i+1}) \\=  e_{q^2} \left(\gamma(1-q^2)(q-q^{-1}) E_{i+1,i} q^{E_{ii}}  - \gamma\lambda(1-q^2){q^{2E_{i+1,i+1}}}) \right) .
\end{multline*}
By \eqref{EFactor},
\begin{multline*}
e_{q^2}(\lambda q^{E_{i+1,i+1}}E_{i,i+1}) 
 e_{q^2}\left( -\gamma \lambda q^{2E_{ii}}\right) 
\\
\times e_{q^2}\left(\gamma(1-q^2)(q-q^{-1})E_{i+1,i}q^{E_{ii}}\right) 
\mathcal{E}_{q^2}(\lambda q^{E_{i+1,i+1}}E_{i,i+1}) \\
=  e_{q^2} \left(\gamma(1-q^2)(q-q^{-1}) E_{i+1,i} q^{E_{ii}}  \right)e_{q^2}\left(- \gamma \lambda(1-q^2){q^{2E_{i+1,i+1}}}) \right) ,
\end{multline*}
which implies the desired identity.

\end{proof}
As proven in \cite{Kuan-IMRN}, the diagonal matrix $G^2B^{-2}$ encodes reversible measures for the multi--species inhomogeneous ASEP$(q,\boldsymbol\theta)$. Note that the $B$ in \cite{Kuan-IMRN} is the same matrix $B$ in the current paper. There is likely an algebraic reason for this. By defining a diagonal matrix $A$ which is constant on irreducible components of the interacting particle system, one can obtain a family of reversible measures given by $AG^2B^{-2}$. 

The next theorem generalizes (in a sense) Theorem 3.2 and equation (95) of \cite{CFG21}.
\begin{theorem}\label{AlgThm}
Define $U(\boldsymbol\lambda):=U_{n-1}(\lambda_{n-1})\cdots U_0(\lambda_0)$. Then

(i)
The function $D$ given by 
$$
D:= G^{-1} M_{U(\boldsymbol\lambda)} G^{-1}B^2
$$
is a self--duality function, and the entries are orthogonal with respect to the reversible measures, in the sense that
$$
D^T G^2 B^{-2}D = G^{-2}B^2.
$$
(ii) More generally, let $A$ be a quantity which is conserved under the dynamics; in other words, $A$ is a diagonal matrix such that $LA=AL$. Then the function $D_A$ given by 
$$
D_A := AG^{-1}M_{U(\boldsymbol\lambda)}G^{-1}B^2A
$$
is a self--duality function, and its entries are orthogonal with respect to the reversible measures, in the sense that
$$
D_A^T A^{-2}G^2B^{-2} D_A = A^2G^{-2}B^2.
$$
\end{theorem}
\begin{proof}
(i)
The fact that $D$ is a self--duality function follows from the general framework of \cite{Kuan-IMRN}.

To see the second statement, note first that 
\begin{align*}
D^T G^2 B^{-2}D &= B^2 G^{-1} M_{U(\boldsymbol\lambda)}^T G^{-1} G^2 B^{-2}G^{-1}M_{U(\boldsymbol\lambda)}G^{-1}B^2\\
&= B^2 G^{-1} M_{U(\boldsymbol\lambda)}^T B^{-2} M_{U(\boldsymbol\lambda)}G^{-1}B^2.
\end{align*}
Since $U(\boldsymbol\lambda)$ is unitary, then by part (c) of Lemma \ref{FirstLemma},
$$
D^T G^2 B^{-2}D =  B^2 G^{-1} B^{-2} M_{U(\boldsymbol\lambda)}^{-1}M_{U(\boldsymbol\lambda)} G^{-1}B^2,
$$
which simplifies to $G^{-2}B^2$, as needed. 

(ii) Noting that $D_A=ADA$ where $D$ is from part (i), we thus have
$$
LD_A = LADA = ALDA = ADL^TA = ADA L^T = D_AL^T,
$$
so $D_A$ is also a self--duality function.

Since $D^TG^2B^{-2}D = G^{-2}B^2$, we have
$$
D_A^T A^{-2}G^2B^{-2} D_A = AD^TA^T A^{-2}G^2B^{-2} ADA = AG^{-2}B^2A = A^2G^{-2}B^2.
$$
\end{proof}

\begin{remark}\label{InnerProductRemark}
From the definition of the unitary element $U(\boldsymbol{\lambda})$, it can be seen why the inner product method does not work for multi--species ($n>1$) models. For $n=2$, this unitary element has the form
\begin{align*}
&e_{q^2}^{1/2}(-\gamma \lambda_1 q^{2E_{11}}) e_{q^2}(\gamma (1-q^2)(q-q^{-1})E_{21}q^{E_{11}})\\
&\times\mathcal{E}_{q^2}(-\lambda_1 q^{E_{22}}E_{12}) \mathcal{E}_{q^2}^{1/2}(\gamma \lambda (1-q^2)q^{2E_{22}})\\
&\times e_{q^2}^{1/2}(-\gamma \lambda_0 q^{2E_{00}}) e_{q^2}(\gamma (1-q^2)(q-q^{-1})E_{10}q^{E_{00}}) \\
&\times\mathcal{E}_{q^2}(-\lambda_1 q^{E_{11}}E_{01}) \mathcal{E}_{q^2}^{1/2}(\gamma \lambda (1-q^2)q^{2E_{11}}).
\end{align*}
Even ignoring the $q^{2E_{11}}$ and $q^{2E_{00}}$ terms, as well as the constants, the linear terms in the $q$--exponentials yield
$$
E_{21}E_{12}E_{10}E_{01}.
$$
Meanwhile, the inner product method would result in linear terms of the form
$$
E_{21}E_{10}E_{01}E_{12}.
$$
However, there is no quantum group relationship between $E_{12}$ and $E_{01}$, so the inner product method will not produce the appropriate unitary element in the multi--species case.    
\end{remark}

\subsection{Proof of Theorem \ref{FirstThm}: 
Calculation of Orthogonal Duality Functions}
With the algebraic machinery, the explicit form of the orthogonal duality functions can now be calculated. The duality function is given by Theorem \ref{AlgThm}, and is written as:
$$
D_A:= AG^{-1}M_{U(\boldsymbol\lambda)}G^{-1}B^2A
 $$
 and it satisfies
 $$
D_A^T A^{-2}G^2B^{-2} D_A = A^2G^{-2}B^2.
$$
Recall that 
\begin{multline*}
   U_i(\lambda):=
 {e}_{q^2}^{1/2}\left( -\gamma \lambda q^{2E_{ii}}\right) 
{e}_{q^2} \left(\lambda E_{i+1,i}q^{E_{ii}}\right) \\
\times\mathcal{E}_{q^2}(-\lambda q^{E_{i+1,i+1}}E_{i,i+1}) 
 \mathcal{E}_{q^2}^{1 / 2}\left( \gamma \lambda(1-q^2){q^{2E_{i+1,i+1}}}  \right)   , 
\end{multline*}
 $G$ is the diagonal matrix with entries
\begin{equation*}
\prod_{i=0}^{n} \prod_{x=1}^{L} \frac{1}{\left[\eta_{i}^{x}\right]_{q}^{!}} \times \prod_{1 \leq y<x \leq L} \prod_{i=0}^{n-1} q^{-\eta_{i+1}^{y} \eta_{[1, i]}^{x}}
\end{equation*}
and $B$ is the diagonal matrix with entries with entries 
$$B(\lambda,\mu) = 1_{\{\lambda=\mu\}} \frac{1}{\sqrt{\{\mu_0\}_{q^2}^! \cdots \{\mu_n\}_{q^2}^!}}.$$

\subsubsection{Inhomogeneous, single--species case}
First, briefly note that the entries of $e_{q^2}^{1/2}$ and $\mathcal{E}_{q^2}^{1/2}$ terms are actually constants in the duality function  (up to particle conservation) and can be ignored.

The next step is to show that in the case of a fixed $i$, the duality function is an inhomogeneous version of the one from \cite{CFG21}.  Doing so is essentially matching notation from this paper to the notation in \cite{CFG21}, rather than developing new proofs. That previous paper \cite{CFG21} uses generators of $\mathcal{U}_q(\mathfrak{sl}_2)$:
$$
\begin{array}{l}q^{A^{0}} A^{+}=q A^{+} q^{A^{0}} \\ q^{A^{0}} A^{-}=q^{-1} A^{-} q^{A^{0}} \\ {\left[A^{+}, A^{-}\right]=\left[2 A^{0}\right]_{q}}\end{array}
$$
Additionally, the representation
 $$
\begin{aligned} A^{+}\vert n\rangle &=\sqrt{[\theta-n]_{q}[n+1]_{q}}\vert n+1\rangle \\ A^{-}\vert n\rangle &=\sqrt{[n]_{q}[\theta-n+1]_{q}}\vert n-1\rangle \\ A^{0}\vert n\rangle &=(n-\theta / 2)\vert n\rangle \end{aligned}
$$
needs to be compared with
\begin{align*}
E_{i,i+1} v_{\mu} &= [\mu_{i+1}]_q v_{ \mu + \epsilon_i}, \quad 0 \leq i < n \\
E_{i+1,i}v_{\mu}  &= [\mu_{i}]_q v_{\mu - \epsilon_i} ,\quad 0 \leq i < n \\
q^{E_{ii}} v_{\mu} &= q^{\mu_i} v_{\mu},\quad 0 \leq i \leq n.
\end{align*}
Now, note that
$$
A^-A^+ \vert n \rangle = \sqrt{[\theta-n]_{q}[n+1]_{q}} A^- \vert n+1\rangle = {[\theta-n]_{q}[n+1]_{q}} \vert n\rangle,
$$
while in the one--species case
$$
E_{01}E_{10} v_{\mu} = [\mu_0]_q [\mu_1+1]_q v_{\mu}.
$$
Thus it is natural to replace $A^+$ with $E_{10}$ and $A^-$ with $E_{01}$.

In the notation of section 8.3 from \cite{CFG21}, there are the operators
\begin{equation*}
S_{\alpha}^{}=e_{q^{2}}\left(\sqrt{\alpha}\left(1-q^{2}\right) \Delta^{L-1}\left(q^{A^{0}} A^{+}\right)\right)
\end{equation*}
and
\begin{equation*}
\hat{S}_{\alpha}^{}=\mathcal{E}_{q^{2}}\left(\sqrt{\alpha}\left(1-q^{2}\right) q^{-\theta/2-\theta L}\Delta^{L-1}\left(q^{-A^{0}} A^{+}\right)\right).
\end{equation*}
 However, the present paper does not use $q^{A_0}$ because it is not an element of the quantum group. In the notation of this paper, $q^{A_0}$ is the analog of $q^{(E_{00}-E_{11})/2}$. Because the element $q^{E_{00}+E_{11}}$ is central,  inserting it into the $q$--exponentials will introduce an irrelevant constant, in the sense that the constant is preserved under the dynamics. The constant $q^{-\theta/2-\theta L}$ is similarly meaningless. Therefore, the operators can be replaced with 
\begin{equation*}
\mathfrak{S}_{\alpha}^{}=e_{q^{2}}\left(\sqrt{\alpha}\left(1-q^{2}\right) \Delta^{L-1}\left( E_{10}q^{E_{00}}\right)\right)
\end{equation*}
and
\begin{equation*}
\hat{\mathfrak{S}}_{\alpha}^{}=\mathcal{E}_{q^{2}}\left(\sqrt{\alpha}\left(1-q^{2}\right) \Delta^{L-1}\left(q^{E_{11}} E_{10}\right)\right),
\end{equation*}
where we have additionally used the relation $q^{E_{00}}E_{10}= q^{-1}E_{10}q^{E_{00}}$ and that $q^{-1}$ only contributes a constant.

The matrix entries of $\mathfrak{S}_{\alpha}$ and $\hat{\mathfrak{S}}_{\alpha}$ are explicitly calculated in the unnumbered equations before (99) in \cite{CFG21}, and produce the single--species duality function. Therefore, for the value of $\lambda = \sqrt{\alpha}(1-q^2)$, we have proved the theorem in the single--species case. 

\subsubsection{Multi--species case}

In the multi--species case, the parameter $\lambda$ in $U_i(\lambda)$ is inhomogeneous and chosen so that $\lambda_i = \sqrt{\alpha_i}(1-q^2)$. 
Then the duality function is not simply a product of the single--species duality functions. However, it is ``almost'' a product of single--species duality functions, in the sense that 
$$
\mathcal{D}^{\boldsymbol\theta}_{\boldsymbol{\alpha}}(\boldsymbol\xi, \boldsymbol\eta) = \mathcal{G}_{\boldsymbol{\alpha}}(\boldsymbol\xi,\boldsymbol\eta)\prod_{i=0}^{n-1}D^{  \boldsymbol{\theta}^{(i)} }_{\alpha_i}(\boldsymbol\xi_i, \boldsymbol\zeta^{(i)}_i),
$$
where $\mathcal{G}_{\boldsymbol{\alpha}}$ was defined in \eqref{curlyG}. This follows because the unitary symmetry in the multi--species case, as defined at the beginning of Theorem \ref{AlgThm}, is a product of the symmetries in the single--species case, but with diagonal matrices in between terms of this product. Thus, to complete the proof, it only remains to compute the term $\mathcal{C}_{\boldsymbol{\alpha}}(\boldsymbol\xi, \boldsymbol\eta)$.
From the general algebraic set--up in Theorem 4.5(ii), the duality functions $\mathcal{C}_{\boldsymbol{\alpha}}(\boldsymbol\xi, \boldsymbol\eta)\mathcal{D}^{\boldsymbol\theta}_{\boldsymbol{\alpha}}(\boldsymbol\xi, \boldsymbol\eta)$ are orthogonal with respect to the reversible measures defined by $A^2G^{-2}B^2$. Given that part (ii) of that Theorem has already been proven, it therefore suffices to show that
\begin{equation}\label{above}
A^2G^{-2}B^2(\boldsymbol{\xi},\boldsymbol{\xi}) = \mu^n(\boldsymbol{\xi}).
\end{equation}
Recall from \eqref{GB} that 
\begin{equation*}
G^2B^{-2}(\boldsymbol\xi,\boldsymbol\xi) = \sum_{\boldsymbol k} \frac{\mu_{\boldsymbol k}^n(\boldsymbol\xi)}{Z_{n,\boldsymbol k}},
\end{equation*}
where $Z_{n, \boldsymbol{k}}$ is a normalization constant whose explicit value is irrelevant. It turns out that for the choice of the diagonal matrix with entries
$$
A(\boldsymbol\xi, \boldsymbol\xi) = \sum_{\boldsymbol k} a_{\boldsymbol k} \cdot 1_{\left\{N\left(\boldsymbol\xi_{i}\right)=k_{i}, 0 \leq i \leq n\right\}},
$$
equation \eqref{above} holds.

\subsection{Proof of Theorem \ref{qTAZRP}: duality for the multi--species $q$--TAZRP}

First, note that the two expressions in the theorem are equal by equation (1.11.2) of \cite{KLS}, which is a $q-$analogue of Newton’s binomium
$$ 
{ }_{1} \phi_{0}\left(\begin{array}{c}q^{-n} \\ -\end{array} ;q, z\right)=\left(z q^{-n} ; q\right)_{n}, \quad n=0,1,2, \ldots
$$

\subsubsection{Charge--Parity Symmetry}
If the limit $\boldsymbol\theta\rightarrow \infty$ is taken in the duality function, the result is $0$. Indeed, in the totally asymmetric case, this must always be the case, as explained in \cite{Kuan-IMRN}, and also observed in \cite{CGRS} for triangular self-duality functions. But with charge parity, the duality result follows immediately if we can take the limit of the ASEP$(q,\boldsymbol\theta)$ duality function and show that it equals the function in the theorem (up to  constants under the dynamics), the result is proven. The remainder of this subsubsection is to that calculation.

First we consider the single species ASEP$(q,\boldsymbol\theta)$.
\begin{theorem}\label{thm: 4.6}
\begin{equation}
    \prod_{x \in \Lambda_L} {}_1\phi_0\left(q^{-2\xi^x};q^2,q^{-2\left(N_{x-1}^-(\boldsymbol\xi)+N_{x+1}^+(\boldsymbol\eta)\right)+1}\right)
\end{equation}
is a space reversed duality for q-TAZRP, where $\boldsymbol\eta$ evolves with total asymmetry to the left and $\boldsymbol\xi$ evolves with total asymmetry to the right.
\end{theorem}

\begin{proof}
Recall the orthogonality duality for ASEP in equation \eqref{eq:4} when $\alpha_i = \alpha$ for all lattice sites $i$, i.e. 
\begin{equation*}
    D_\alpha^{\boldsymbol\theta}(\boldsymbol\xi,\boldsymbol\eta)=\prod_{x=1}^L K_{\eta^x}\left(q^{-2\xi^x}, p^x(\boldsymbol\xi,\boldsymbol\eta),\theta^x,q^2\right)
\end{equation*}
and
\begin{equation*}
p^x(\boldsymbol\xi,\boldsymbol\eta)=\alpha^{-1}q^{-2\left(N_{x-1}^-(\boldsymbol\xi)-N_{x+1}^+(\boldsymbol\eta)\right)+2N_{x-1}^-(\boldsymbol\theta)-1},
\end{equation*}
where $\alpha$ satisfies $q^{2\theta^x}p^x(\boldsymbol\xi,\boldsymbol\eta)>1$ for all $x$, thus $0<\alpha<q^{2N(\boldsymbol\theta)-1}$.
\newline
Now apply $T$ to $\boldsymbol\eta$, 
\begin{equation}
\begin{array}{cl}
D_\alpha^{\boldsymbol\theta}\left(\boldsymbol\xi,T(\boldsymbol\eta)\right)  &  =\prod_{x=1}^L K_{\theta^x-\eta^x}\left(q^{-2\xi^x}, p^x(\boldsymbol\xi,\boldsymbol\theta-\boldsymbol\eta),\theta^x,q^2\right) \\
     &  =\prod_{x=1}^L {}_2\phi_1(q^{-2\xi^x},q^{-2(\theta^x-\eta^x)};q^{-2\theta^x};q^2,p^x(\boldsymbol\xi,\boldsymbol\theta-\boldsymbol\eta)q^{2(\theta^x-\eta^x+1)}).
\end{array}
\end{equation}

Take $\alpha=q^{2N(\boldsymbol\theta)}$,
Next we take $\boldsymbol\theta$ to $\infty$, i.e $\theta^x\xrightarrow[]{}\infty$ for all $x$,
\begin{equation}
\begin{array}{l}
 \displaystyle{\lim_{\boldsymbol\theta\xrightarrow{}\infty}{}_2\phi_1(q^{-2\xi^x},q^{-2(\theta^x-\eta^x)};q^{-2\theta^x};q^2,p^x(\boldsymbol\xi,\boldsymbol\theta-\boldsymbol\eta)q^{2(\theta^x-\eta^x+1)})} \\
     = \displaystyle{\lim_{\boldsymbol\theta\xrightarrow{}\infty}{}_2\phi_1\left(q^{-2\xi^x},q^{-2(\theta^x-\eta^x)};q^{-2\theta^x};q^2,q^{-2\left(N_{x-1}^-(\boldsymbol\xi)+N_{x+1}^+(\boldsymbol\eta)\right)-2\eta^x+1}\right)}\\   
       =\displaystyle{\sum_{k=0}^{\xi^x}\frac{(q^{-2\xi^x};q^2)_k}{(q^2;q^2)_k}\left(q^{-2\left(N_{x-1}^-(\boldsymbol\xi)+N_{x+1}^+(\boldsymbol\eta)\right)+1}\right)^k}\\  ={}_1\phi_0\left(q^{-2\xi^x};q^2,q^{-2\left(N_{x-1}^-(\boldsymbol\xi)+N_{x+1}^+(\boldsymbol\eta)\right)+1}\right).
\end{array}
\end{equation}
\end{proof}

Here is the generalized duality for general $n$-species  $q$-TAZRP .
\begin{theorem}
\begin{multline}
D^{q-TARZP}(\boldsymbol\xi,\boldsymbol\eta) =     q^{h(\boldsymbol\xi,\boldsymbol\eta)}\\
\times\prod_{i=0}^{n-1}\prod_{x\in\Lambda_L}{}_1\phi_0\left(q^{-2\xi_i^x};q^2,q^{-2\left(N_{x-1}^-(\boldsymbol\xi_i)+N_{x+1}^+(T(\boldsymbol\eta)_{i+1})\right)+1}\right)
\end{multline}
is a space reversed duality for $n $-species $q$-TAZRP , where $\boldsymbol\eta$ evolves with total asymmetry to the left and $\boldsymbol\xi$ evolves with total asymmetry to the right, and 
\begin{equation}\label{eq:2}
    h(\boldsymbol\xi,\boldsymbol\eta)=\sum_{x \in \Lambda_L }\sum_{i=0}^{n-1}\left(-\xi_i^xN_{x+1}^+(\boldsymbol\eta_{[0,n-2-i]})+\eta_i^xN_{x}^+(\boldsymbol\xi_{[0,n-2-i]})\right).
\end{equation}
\end{theorem}
\begin{proof}
 We will follow the proof for Theorem \ref{thm: 4.6} by first applying $T$ to $\boldsymbol\eta$ in $\mathcal{D}_{\boldsymbol{\alpha}}^{\boldsymbol{\theta}}(\boldsymbol\xi,\boldsymbol\eta)$ and then taking the limit as all $\theta^x$ goes to $\infty$. We will show that the limit  equals $D^{q-TARZP}(\boldsymbol\xi,\boldsymbol\eta)$ up to a constant. Namely, there exists function $f(\boldsymbol\theta,N(\boldsymbol\xi),N(\boldsymbol\eta))$ such that 
\begin{equation*}
    \lim_{\boldsymbol\theta\xrightarrow{}\infty} \frac{\mathcal{D}_{\boldsymbol{\alpha}}^{\boldsymbol{\theta}}(\boldsymbol\xi,T(\boldsymbol\eta)) }{q^{\frac{f(\boldsymbol\theta,N(\boldsymbol\xi),N(\boldsymbol\eta))}{2}}} = D^{q-TARZP}(\boldsymbol\xi,\boldsymbol\eta).
\end{equation*}
This will be verified by straightforward calculation.

 Let $\boldsymbol\theta^{(i)}=\boldsymbol\eta_{[0,i+1]}-\boldsymbol\xi_{[0,i-1]}$, $\alpha_i=q^{2N(\boldsymbol\theta^{(i)})}$,
recall that 
\begin{equation*}
\mathcal{D}_{\boldsymbol{\alpha}}^{\boldsymbol{\theta}}(\boldsymbol\xi,\boldsymbol\eta)=\prod_{i=0}^{n-1} D_{\alpha_i}^{(\boldsymbol\theta^{(i)})}(\boldsymbol\xi_i,\boldsymbol\eta_{[0,i]}-\boldsymbol\xi_{[0,i-1]}) 
    \times \sqrt{\frac{\prod_{i=0}^{n-1}\mu_{\alpha_i}^{\boldsymbol\theta^{(i)}}(\boldsymbol\xi_i)\mu_{\alpha_i}^{\boldsymbol\theta^{(i)}}(\boldsymbol\eta_{[0,i]}-\boldsymbol\xi_{[0,i-1]})}{\mu^n(\boldsymbol\xi) \mu^n(\boldsymbol\eta)}}.
\end{equation*}

First, Theorem \ref{thm: 4.6} shows that 
\begin{multline*}
    \lim_{\boldsymbol\theta\xrightarrow[]{}\infty}\prod_{i=0}^{n-1} D_{\alpha_i}^{(T(\boldsymbol\theta^{(i)}))}(\boldsymbol\xi_i,T(\boldsymbol\eta_{[0,i]}-\boldsymbol\xi_{[0,i-1]}))\\
= \prod_{i=0}^{n-1}\prod_{x=1}^L{}_1\phi_0\left(q^{-2\xi_i^x};q^2,q^{-2\left(N_{x-1}^-(\boldsymbol\xi_i)+N_{x+1}^+(T(\boldsymbol\eta)_{i+1})\right)+1}\right). 
\end{multline*}

Next, we calculate
\begin{multline*}
 \lim_{\boldsymbol\theta\xrightarrow[]{}\infty}\  \frac{\prod_{i=0}^{n-1}\mu_{T(\alpha_i)}^{T(\boldsymbol\theta^{(i)})}(\boldsymbol\xi_i)\mu_{T(\alpha_i)}^{T(\boldsymbol\theta^{(i)})}(T(\boldsymbol\eta_{[0,i]})-\boldsymbol\xi_{[0,i-1]})}{\mu^n(\boldsymbol\xi) \mu^n(T(\boldsymbol\eta))}\\
=\lim_{\boldsymbol\theta\xrightarrow[]{}\infty} \frac{\prod_{i=0}^{n-1}\mu_{T(\alpha_i)}^{T(\boldsymbol\theta^{(i)})}(\boldsymbol\xi_i)\mu_{T(\alpha_i)}^{T(\boldsymbol\theta^{(i)})}(T(\boldsymbol\eta_{[0,i]})-\boldsymbol\xi_{[0,i-1]})}{\mu^n(\boldsymbol\xi) \mu^n(T(\boldsymbol\eta))}   ,
\end{multline*}
where 
$T(\boldsymbol\theta^{(i)})=\boldsymbol\theta-\boldsymbol\eta_{[0,n-2-i]}-\boldsymbol\xi_{[0,i-1]}$, $T(\boldsymbol\eta_{[0,i]})=\boldsymbol\theta-\boldsymbol\eta_{[0,n-1-i]}$.
    For clarity, we restate some of the notations: $\xi^x_{[0,i-1]}=\sum_{j=0}^{i-1}\xi_j^x$,  $T(\boldsymbol\theta^{(i)})^x=\theta^x-\eta^x_{[0,n-2-i]}-\xi^x_{[0,i-1]}$ and $T(\boldsymbol\eta_{[0,i]})^x=\theta^x-\eta^x_{[0,n-1-i]}$.
Now plug in the reversible measures $\mu_\alpha^{\boldsymbol\theta}$ and $\mu^n$, we could split

\begin{equation}\label{eq: 25}
\frac{\prod_{i=0}^{n-1}\mu_{T(\alpha_i)}^{T(\boldsymbol\theta^{(i)})}(\boldsymbol\xi_i)\mu_{T(\alpha_i)}^{T(\boldsymbol\theta^{(i)})}(T(\boldsymbol\eta_{[0,i]})-\boldsymbol\xi_{[0,i-1]})}{\mu^n(\boldsymbol\xi) \mu^n(T(\boldsymbol\eta))}   
\end{equation}into two parts,  one contains $q$-binomial terms and the other is $q^{ r(\boldsymbol\theta,\boldsymbol\xi,\boldsymbol\eta)}$, i.e.
$$\eqref{eq: 25}=q^{ r(\boldsymbol\theta,\boldsymbol\xi,\boldsymbol\eta)}\prod_{x \in \Lambda_L }\frac{\prod_{i=0}^{n-1}\binom{T(\boldsymbol\theta^{(i)})^x}{\xi_i^x}_q \binom{T(\boldsymbol\theta^{(i)})^x}{T(\boldsymbol\eta_{[0,i]})^x-\xi^x_{[0,i-1]}}_q}{\binom{\theta^x}{\boldsymbol\xi^x}_q\binom{\theta^x}{\boldsymbol\eta^x}_q},$$
while
\begin{align*}
   r(\boldsymbol\theta,\boldsymbol\xi,\boldsymbol\eta) = & \sum_{i=0}^{n-1}  \sum_{x \in \Lambda_L } (2N_{x+1}^+(T(\boldsymbol\theta^{(i)}))+T(\boldsymbol\theta^{(i)})^x)(T(\boldsymbol\eta_{[0,i]})^x-\xi_{[0,i-1]}^x+\xi_i^x)\\
   &-\sum_{x \in \Lambda_L }\sum_{i=0}^{n}\frac{(\xi_i^x)^2+(\eta_i^x)^2}{2}
        +\sum_{x \in \Lambda_L }\sum_{y<x}\sum_{i=0}^{n-1}2\xi^x_{[0,i]}\xi_{i+1}^y\\
        &+\sum_{x \in \Lambda_L }\sum_{y<x}\sum_{i=0}^{n-1}2T(\boldsymbol\eta)^x_{[0,i]}T(\boldsymbol\eta)_{i+1}^y-2\sum_{i=0}^{n-1}N(T(\boldsymbol\theta^{(i)}))N(\boldsymbol\xi_i)\\ &
  -2\sum_{i=0}^{n-1}N(T(\boldsymbol\theta^{(i)}))N(T(\boldsymbol\eta_{[0,i]})-\boldsymbol\xi_{[0,i-1]})).   
\end{align*}

Now we take the limit for both parts.
\begin{align*}
   &\lim_{\boldsymbol\theta\xrightarrow{}\infty}    \prod_{x \in \Lambda_L }\frac{\prod_{i=0}^{n-1}\binom{T(\boldsymbol\theta^{(i)})^x}{\xi_i^x}_q \binom{T(\boldsymbol\theta^{(i)})^x}{T(\boldsymbol\eta_{[0,i]})^x-\xi^x_{[0,i-1]}}_q}{\binom{\theta^x}{\boldsymbol\xi^x}_q\binom{\theta^x}{\boldsymbol\eta^x}_q}\\
  &=\lim_{\boldsymbol\theta\xrightarrow{}\infty}  \prod_{x \in \Lambda_L }\left(\frac{\prod_{i=0}^{n-1}[\theta^x-\eta^x_{[0,n-2-i]}-\xi^x_{[0,i-1]}]^!_q}{[\theta^x]_q^!\prod_{i=0}^{n-2}[\theta^x-\eta^x_{[0,n-2-i]}-\xi^x_{[0,i]}]^!_q}\right)^2\\
 & =q^{2\sum_x\sum_{i=0}^{n-1}\xi^x_{n-2-i}\eta^x_{[0,i]}}, 
\end{align*}

We claim that $ r(\boldsymbol\theta,\boldsymbol\xi,\boldsymbol\eta)$ could be written as a sum of $f(\boldsymbol\theta,N(\boldsymbol\xi),N(\boldsymbol\eta))$ and a function $g(\boldsymbol\xi,\boldsymbol\eta)$ that does not depend on $\boldsymbol\theta$, we can divide the duality function by $q^{f/2}$, thus get a nontrivial limit.
\begin{equation*}
     r(\boldsymbol\theta,\boldsymbol\xi,\boldsymbol\eta)=f(\boldsymbol\theta,N(\boldsymbol\xi),N(\boldsymbol\eta))+g(\boldsymbol\xi,\boldsymbol\eta),
\end{equation*}
where
\begin{equation*}
\begin{split}
    g(\boldsymbol\xi,\boldsymbol\eta)
    =2\sum_{x \in \Lambda_L }\sum_{i=0}^{n-1}\left(-\xi_i^xN_{x+1}^+(\boldsymbol\eta_{[0,n-2-i]})+\eta_i^xN_{x+1}^+(\boldsymbol\xi_{[0,n-2-i]})\right).
\end{split}    
\end{equation*}
Combining limit of each part yields the desired result with
$$h(\boldsymbol\xi,\boldsymbol\eta)= \frac{g(\boldsymbol\xi,\boldsymbol\eta)}{2}+\sum_x\sum_{i=0}^{n-1}\xi^x_{n-2-i}\eta^x_{[0,i]}.$$

Although we do not provide an explicit formula for $f$, since we have proved the same duality for $q$-Hahn TAZRP, which degenerates to $q-$TAZRP, it is confirmed that $h(\boldsymbol\xi,\boldsymbol\eta)$ is correct.

Note that when $n=1$, $g=0$, so it degenerates to the single-species case as well.

\end{proof}

\subsection{Proof of Theorem \ref{Vertex}: duality for the stochastic multi--species higher--spin vertex model}
Although this proof is difficult from a probabilistic perspective, it is straightforward from an algebraic perspective. Because the argument is exactly the same as in \cite{KuanCMP}, we only outline the idea. The transition matrix for the vertex model is known as the stochastic transfer matrix,  and is a product of the matrices $S(z)$ from the end of section \ref{sec2.4}. These matrices satisfy
$$
S(z) \Delta(u) = \Delta(u) S^{\text{rev}}(z)
$$
for any $u\in \mathcal{U}_q(\mathfrak{gl}_{n+1})$, where
$$
S^{\text{rev}}(z) = P \circ S(z)
$$
with $P$ being the permutation operator
$$
P(v \otimes w) = w\otimes v.
$$
The permutation operator can be interpreted probabilistically as the charge--parity symmetry. The key difference between this paper and \cite{KuanCMP} is a different choice of $u$, which does not change the fact that resulting function is a duality function.

\subsection{Proof of Theorem \ref{qHahn}: duality for the multi--species   $q$-Hahn TAZRP}

First, a few identities will be needed. These will be stated as the below lemmas:

\begin{lemma}\label{lemma:1}
Fix $\vert q\vert<1$ and $0\leq \mu<1, 0 \leq \lambda \leq 1$. Then for all non--negative integers $x$ and $y$, and any real number $c$,
$$
\sum_{j=0}^{x} \Phi_{q}(j \mid x, \lambda, \mu) \frac{\left(q^{-c+j};q\right)_{y}}{\left(q^{-c};q\right)_{y}}=\sum_{s=0}^{y} \Phi_{q}(s \mid y, \lambda, \mu) \frac{\left(q^{-c+s};q\right)_{x}}{\left(q^{-c};q\right)_{x}}.
$$

\end{lemma}

\begin{proof}
Let $S_{x,y}=\sum_{j=0}^x\Phi_q(j\mid x,\lambda,\mu)\frac{(q^{-c+j};q)_y}{(q^{-c};q)_y}$.  First note that 
\begin{equation*}
    S_{x,y}=\frac{(\mu/\lambda)_x}{(\mu)_x}{ }_{3} \phi_{2}\left(\begin{matrix}
\lambda & q^{-x} & q^{-c+y}\\
\frac{\lambda}{\mu}q^{1-x} &  & q^{-c}
\end{matrix};q,q\right).
\end{equation*}
Recall a transformation formula for ${ }_{3} \phi_{2}$ \cite{Gas96}:
\begin{equation*}
    { }_{3} \phi_{2}\left(\begin{matrix}
 q^{-n} & a & b\\
d &  & e
\end{matrix};q,q\right)=\frac{(a;q)_n(e/b;q)_n}{(d;q)_n(e;q)_n}b^n{ }_{3} \phi_{2}\left(\begin{matrix}
 q^{-n} & d/a &  q^{1-n}/e\\
q^{1-n}/a &  & q^{1-n}b/d
\end{matrix};q,q\right).
\end{equation*}
Let $n=x$, $a=q^{-c+y}$, $b=\lambda$, $d=q^{-c}$, $e=\frac{\lambda}{\mu}q^{1-x}$,
\begin{align*}
  S_{x,y}   &  =\frac{\left(\mu/\lambda;q\right)_x\left(q^{-c+y};q\right)_x\left(\frac{1}{\mu}q^{1-x};q\right)_x}{\left(\mu;q\right)_x\left(q^{-c};q\right)_x\left(\frac{\lambda}{\mu}q^{1-x};q\right)_x}\lambda^x{ }_{3} \phi_{2}\left(\begin{matrix}
 q^{-x} & q^{-y} &  \frac{\mu}{\lambda}\\
q^{1-x-y+c} &  & \mu
\end{matrix};q,q\right)\\
     & =\frac{\left(q^{-c+y};q\right)_x}{\left(q^{-c};q\right)_x}{ }_{3} \phi_{2}\left(\begin{matrix}
 q^{-x} & q^{-y} &  \frac{\mu}{\lambda}\\
q^{1-x-y+c} &  & \mu
\end{matrix};q,q\right)\\
&=\frac{\left(q^{-c+x};q\right)_y}{\left(q^{-c};q\right)_y}{ }_{3} \phi_{2}\left(\begin{matrix}
 q^{-x} & q^{-y} &  \frac{\mu}{\lambda}\\
q^{1-x-y+c} &  & \mu
\end{matrix};q,q\right)\\
&=S_{y,x}.
\end{align*}
\end{proof}
\normalsize
\noindent
Take derivative of lemma \ref{lemma:1} w.r.t. $\lambda$ and let $\lambda=1$ yields the following identity.
 
 \begin{corollary}\label{Cor:1}
\begin{equation}  \sum_{0\le j\le x}\Phi'(j\mid x,\mu)\frac{(q^{-c+y};q)_j}{(q^{-c};q)_j}=\sum_{0\le s\le y}\Phi'(s\mid y,\mu)\frac{(q^{-c+x};q)_s}{(q^{-c};q)_s}
     \end{equation}
for $x,y,\in \mathbb{N},c\in \mathbb{R} $.
\end{corollary}
\noindent
In what follows, $\xi_{[n_1,n_2]}=\sum_{i=n_1}^{n_2}\xi_i$, if $n_1>n_2$, then $\xi_{[n_1,n_2]}=0$. \newline $\boldsymbol\xi_{[n_1,n_2]}=(\xi_{n_1},\ldots,\xi_{n_2})$. Here is a multi-species generalization:

\begin{prop}\label{prop:1}
(1)Fix $\mid q\mid<1$ and $0\le \mu<1$, $0\le \lambda\le 1$, then for all $\boldsymbol\eta,\boldsymbol\xi\in\mathbb{Z}^{n}_{\ge 0}$, and $\boldsymbol{c}\in\mathbb{R}^{n}$.
\begin{multline}\label{eq:6}
\sum_{\boldsymbol\gamma\le \boldsymbol\eta }\Phi(\boldsymbol\gamma\mid\boldsymbol\eta,\lambda,\mu) \prod_{i=0}^{n-1}q^{\boldsymbol\gamma_i\xi_{[0,n-2-i]}}\frac{\left(q^{-c_i+\gamma_{n-1-i}};q\right)_{\xi_i}}{\left(q^{-c_i};q\right)_{\xi_i}}
    \\ =\sum_{\boldsymbol\zeta\le \boldsymbol\xi }\Phi(\boldsymbol\zeta\mid\boldsymbol\xi,\lambda,\mu)\prod_{i=0}^{n-1}q^{\zeta_i\eta_{[0,n-2-i]}}\frac{\left(q^{-c_i+\zeta_i};q\right)_{\eta_{n-1-i}}}{\left(q^{-c_i};q\right)_{\eta_{n-1-i}}}  .
\end{multline}
(2)Fix $\mid q\mid<1$ and $0\le \mu<1$, then for all $\boldsymbol\eta,\boldsymbol\xi\in\mathbb{Z}^{n}_{\ge 0}$, and $\boldsymbol{c}\in\mathbb{R}^{n}$ 
\begin{multline*}
    \sum_{\boldsymbol\gamma\le \boldsymbol\eta }\Phi'(\boldsymbol\gamma\mid\boldsymbol\eta,\mu) \prod_{i=0}^{n-1}q^{\gamma_i\xi_{[0,n-2-i]}}\frac{\left(q^{-c_i+\gamma_{n-1-i}};q\right)_{\xi_i}}{\left(q^{-c_i};q\right)_{\xi_i}}
     \\=\sum_{\boldsymbol\gamma\le \boldsymbol\xi }\Phi'(\boldsymbol\zeta\mid\boldsymbol\xi,\mu)\prod_{i=0}^{n-1}q^{\zeta_i\eta_{[0,n-2-i]}}\frac{\left(q^{-c_i+\zeta_i};q\right)_{\eta_{n-1-i}}}{\left(q^{-c_i};q\right)_{\eta_{n-1-i}}}  .
\end{multline*}
\end{prop}
\begin{proof}
Induction on $n$, the base case when $n=1$ is just Lemma \ref{lemma:1}, now assume \eqref{eq:6} holds for $n$-species, next we show the identity holds also for $(n+1)$-species.

Now suppose $\boldsymbol\eta,\boldsymbol\xi\in\mathbb{Z}^{n+1}_{\ge 0}$, and $\boldsymbol{c}\in\mathbb{R}^{n+1}$,
\begin{multline}\label{eq:3}
    \sum_{\boldsymbol\gamma\le \boldsymbol\eta }\Phi(\boldsymbol\gamma\mid\boldsymbol\eta,\lambda,\mu) \prod_{i=0}^{n}q^{\gamma_i\xi_{[0,n-1-i]}}\frac{\left(q^{-c_{n-i}+\gamma_{i}};q\right)_{\xi_{n-i}}}{\left(q^{-c_{n-i}};q\right)_{\xi_{n-i}}}\\
    =\sum_{\boldsymbol\gamma_{[0,n-1]}\le \boldsymbol\eta_{[0,n-1]}}\Phi(\boldsymbol\gamma_{[0,n-1]}\mid\boldsymbol\eta_{[0,n-1]},\lambda,\mu)\prod_{i=0}^{n-1}q^{\gamma_i\xi_{[0,n-1-i]}}\frac{\left(q^{-c_{n-i}+\gamma_{i}};q\right)_{\xi_{n-i}}}{\left(q^{-c_{n-i}};q\right)_{\xi_{n-i}}}\\
    \sum_{\gamma_{n}\le \eta_{n}}\left(\left(\frac{\mu}{\lambda}\right)^{\gamma_n}q^{\gamma_n(\eta_{[0,n-1]}-\gamma_{[0,n-1]})}\frac{(\lambda q^{\gamma_{[0,n-1]}};q)_{\gamma_n}(\frac{\mu}{\lambda} q^{\eta_{[0,n-1]}-\gamma_{[0,n-1]}};q)_{\eta_{n}-\gamma_{n}}}{(\mu q^{\eta_{[0,n-1]}};q)_{\eta_n}}\right.\\
    \left.\times\binom{\eta_n}{\gamma_n}_q\frac{\left(q^{-c_{0}+\gamma_{n}};q\right)_{\xi_{0}}}{\left(q^{-c_{0}};q\right)_{\xi_{0}}}\right).
\end{multline}

Let $\hat{\lambda}=\lambda q^{\gamma_{[0,n-1]}}$, $\hat{\mu}=\mu q^{\eta_{[0,n-1]}}$, by Lemma \ref{lemma:1},
\begin{align*}
        &\sum_{\gamma_{n}\le \eta_{n}}\left(\frac{\mu}{\lambda}\right)^{\gamma_n}q^{\gamma_n(\eta_{[0,n-1]}-\gamma_{[0,n-1]})}\frac{(\lambda q^{\gamma_{[0,n-1]}})_{\gamma_n}(\frac{\mu}{\lambda} q^{\eta_{[0,n-1]}-\gamma_{[0,n-1]}})_{\eta_{n}-\gamma_{n}}}{(\mu q^{\eta_{[0,n-1]}})_{\eta_n}}
        \\ &\times \binom{\eta_n}{\gamma_n}_q\frac{\left(q^{-c_{0}+\gamma_{n}};q\right)_{\xi_{0}}}{\left(q^{-c_{0}};q\right)_{\xi_{0}}}
         =\sum_{\gamma_{n}\le \eta_{n}}\Phi(\gamma_n\mid\eta_n,\hat{\lambda},\hat{\mu})\frac{\left(q^{-c_{0}+\gamma_{n}};q\right)_{\xi_{0}}}{\left(q^{-c_{0}};q\right)_{\xi_{0}}}
          \\ &=\sum_{\zeta_0\le \xi_{0}}\Phi(\zeta_0\mid\xi_{0},\hat{\lambda},\hat{\mu})\frac{\left(q^{-c_{0}+\zeta_0};q\right)_{\eta_n}}{\left(q^{-c_{0}};q\right)_{\eta_n}}.  
\end{align*}
Let $\tilde{\lambda}=\lambda q^{\zeta_0}$,  $\tilde{\mu}=\mu q^{\xi_0}$,
\begin{align*}
&\sum_{\boldsymbol\gamma_{[0,n-1]}\le \boldsymbol\eta_{[0,n-1]}}\Phi(\boldsymbol\gamma_{[0,n-1]}\mid\boldsymbol\eta_{[0,n-1]},\lambda,\mu)\prod_{i=0}^{n-1}q^{\gamma_i\xi_{[0,n-i-1]}}\frac{\left(q^{-c_{n-i}+\gamma_{i}};q\right)_{\xi_{n-i}}}{\left(q^{-c_{n-i}};q\right)_{\xi_{n-i}}}\\
         &\quad \quad \quad \times q^{\zeta_0(\eta_{[0,n-1]}-\gamma_{[0,n-1]})}\frac{(\lambda q^{\gamma_{[0,n-1]}};q)_{\zeta_0}(\frac{\mu}{\lambda} q^{\eta_{[0,n-1]}-\gamma_{[0,n-1]}};q)_{\xi_{0}-\zeta_0}}{(\mu q^{\eta_{[0,n-1]}};q)_{\xi_{0}}}\\
       & = q^{\zeta_0\eta_{[0,n-1]}}\frac{(\lambda;q)_{\zeta_0}(\frac{\mu}{\lambda};q)_{\xi_0-\zeta_0}}{(\mu;q)_{\xi_0}}
        \sum_{\boldsymbol\gamma_{[0,n-1]}\le \boldsymbol\eta_{[0,n-1]}}\Phi(\boldsymbol\gamma_{[0,n-1]}\mid\boldsymbol\eta_{[0,n-1]},\tilde{\lambda},\tilde{\mu})
        \\
        &\quad \quad \quad \times \prod_{i=0}^{n-1}q^{\gamma_i\xi_{[1,n-i-1]}}\frac{\left(q^{-c_{n-i}+\gamma_{i}};q\right)_{\xi_{n-i}}}{\left(q^{-c_{n-i}};q\right)_{\xi_{n-i}}}\\
       & =q^{\zeta_0\eta_{[0,n-1]}}\frac{(\lambda;q)_{\zeta_0}(\frac{\mu}{\lambda};q)_{\xi_0-\zeta_0}}{(\mu;q)_{\xi_0}}\sum_{\boldsymbol\zeta_{[1,n]}\le \boldsymbol\xi_{[1,n]}}\Phi(\boldsymbol\zeta_{[1,n]}\mid\boldsymbol\xi_{[1,n]},\tilde{\lambda},\tilde{\mu})
        \\
           &\quad \quad \quad \times \prod_{i=0}^{n-1}q^{\zeta_{i+1}\eta_{[0,n-i-2]}}\frac{\left(q^{-c_{i}+\zeta_{i}};q\right)_{\eta_{n-i}}}{\left(q^{-c_{i}};q\right)_{\eta_{n-i}}}
        \\
        &=q^{\zeta_0\eta_{[0,n-1]}}\frac{(\lambda)_{\zeta_0}(\frac{\mu}{\lambda})_{\xi_0-\zeta_0}}{(\mu)_{\xi_0}}\sum_{\boldsymbol\zeta_{[1,n]}\le \boldsymbol\xi_{[1,n]}}\Phi(\boldsymbol\zeta_{[1,n]}\mid\boldsymbol\xi_{[1,n]},\tilde{\lambda},\tilde{\mu}) \\
           &\quad \quad \quad \times
        \prod_{i=1}^{n}q^{\zeta_{i}\eta_{[0,n-i-1]}}\frac{\left(q^{-c_{i}+\zeta_{i}};q\right)_{\eta_{n-i}}}{\left(q^{-c_{i}};q\right)_{\eta_{n-i}}},
\end{align*}  
where the second to last equality follows from the induction hypothesis.
Thus plug in the above two equations,
\begin{align*}
     &\text{\eqref{eq:3}}=\sum_{\boldsymbol\gamma_{[0,n-1]}\le \boldsymbol\eta_{[0,n-1]}}\Phi(\boldsymbol\gamma_{[0,n-1]}\mid\boldsymbol\eta_{[0,n-1]},\lambda,\mu)\prod_{i=0}^{n-1}q^{\gamma_i\xi_{[0,n-1-i]}}
     \\
 &\quad \quad \quad\times \frac{\left(q^{-c_{n-i}+\gamma_{i}};q\right)_{\xi_{n-i}}}{\left(q^{-c_{n-i}};q\right)_{\xi_{n-i}}} \sum_{\zeta_0\le \xi_{0}}\Phi(\zeta_0\mid\xi_{0},\hat{\lambda},\hat{\mu})\frac{\left(q^{-c_{0}+\zeta_0};q\right)_{\eta_n}}{\left(q^{-c_{0}};q\right)_{\eta_n}}\\
    & =\sum_{\boldsymbol\gamma_{[0,n-1]}\le \boldsymbol\eta_{[0,n-1]}}\Phi(\boldsymbol\gamma_{[0,n-1]}\mid\boldsymbol\eta_{[0,n-1]},\lambda,\mu)\prod_{i=0}^{n-1}q^{\gamma_i\xi_{[0,n-1-i]}}\frac{\left(q^{-c_{n-i}+\gamma_{i}};q\right)_{\xi_{n-i}}}{\left(q^{-c_{n-i}};q\right)_{\xi_{n-i}}}\\
      &  \quad   \times
        \sum_{\zeta_0\le \xi_{0}}\left(\frac{\mu}{\lambda}\right)^{\zeta_0}q^{\zeta_0(\eta_{[0,n-1]}-\gamma_{[0,n-1]})}\frac{(\lambda q^{\gamma_{[0,n-1]}})_{\zeta_0}(\frac{\mu}{\lambda} q^{\eta_{[0,n-1]}-\gamma_{[0,n-1]}})_{\xi_{0}-\zeta_0}}{(\mu q^{\eta_{[0,n-1]}})_{\xi_{0}}} \\ 
        &     \quad  \times \binom{\xi_{0}}{\zeta_0}_q\frac{\left(q^{-c_{0}+\zeta_0};q\right)_{\eta_n}}{\left(q^{-c_{0}};q\right)_{\eta_n}}
        = \sum_{\zeta_0\le \xi_{0}}\left(\frac{\mu}{\lambda}\right)^{\zeta_0} \binom{\xi_{0}}{\zeta_0}_q\frac{\left(q^{-c_{0}+\zeta_0};q\right)_{\eta_n}}{\left(q^{-c_{0}};q\right)_{\eta_n}} \\    &   \quad  \times \sum_{\boldsymbol\gamma_{[0,n-1]}\le \boldsymbol\eta_{[0,n-1]}}\Phi(\boldsymbol\gamma_{[0,n-1]}\mid\boldsymbol\eta_{[0,n-1]},\lambda,\mu)
        \prod_{i=0}^{n-1}q^{\gamma_i\xi_{[0,n-1-i]}}\frac{\left(q^{-c_{n-i}+\gamma_{i}};q\right)_{\xi_{n-i}}}{\left(q^{-c_{n-i}};q\right)_{\xi_{n-i}}} \\    
        &  \quad  \times q^{\zeta_0(\eta_{[0,n-1]}-\gamma_{[0,n-1]})}\frac{(\lambda q^{\gamma_{[0,n-1]}};q)_{\zeta_0}(\frac{\mu}{\lambda} q^{\eta_{[0,n-1]}-\gamma_{[0,n-1]}};q)_{\xi_{0}-\zeta_0}}{(\mu q^{\eta_{[0,n-1]}};q)_{\xi_{0}}}\\
        &=\sum_{\zeta_0\le \xi_{0}}\left(\frac{\mu}{\lambda}\right)^{\zeta_0}\binom{\xi_{0}}{\zeta_0}_q\frac{\left(q^{-c_{0}+\zeta_0};q\right)_{\eta_n}}{\left(q^{-c_{0}};q\right)_{\eta_n}}q^{\zeta_0\eta_{[0,n-1]}}\frac{(\lambda;q)_{\zeta_0}(\frac{\mu}{\lambda};q)_{\xi_0-\zeta_0}}{(\mu;q)_{\xi_0}}\\
     & \quad \times\sum_{\boldsymbol\zeta_{[1,n]}\le \boldsymbol\xi_{[1,n]}}\Phi(\boldsymbol\zeta_{[1,n]}\mid\boldsymbol\xi_{[1,n]},\lambda q^{\zeta_0},\mu q^{\xi_0})\prod_{i=1}^{n}q^{\zeta_i\eta_{[0,n-1-i]}}\frac{\left(q^{-c_{i}+\zeta_{i}};q\right)_{\eta_{n-i}}}{\left(q^{-c_{i}};q\right)_{\eta_{n-i}}}\\
      &   =\sum_{\boldsymbol\zeta\le \boldsymbol\xi }\Phi(\boldsymbol\zeta\mid\boldsymbol\xi,\lambda,\mu)\prod_{i=0}^{n}q^{\zeta_i\eta_{[0,n-1-i]}}\frac{\left(q^{-c_i+\zeta_i};q\right)_{\eta_{n-i}}}{\left(q^{-c_i};q\right)_{\eta_{n-i}}}\;.
         \end{align*}
  
The second part of the proposition is the derivative of the first part with respect to $\lambda$ at $\lambda=1$.

\end{proof}

With the identities proven, we now turn to the conjecture that the duality function obtained in previous section is a space–reversed self–duality function for the single--species $q^2$-Hahn Boson process. 

\begin{theorem}
\begin{eqnarray}
     D(\boldsymbol\eta,\boldsymbol\xi) = \prod_{x \in \Lambda_L }{}_1\phi_0\left(q^{-\xi^x};q,q^{-\left(N_{x-1}^-(\boldsymbol\xi)+N_{x+1}^+(\boldsymbol\eta)\right)+1/2}\right) \\
     =\prod_{x \in \Lambda_L } \left(q^{-\left(N_{x}^-(\boldsymbol\xi)+N_{x+1}^+(\boldsymbol\eta)\right)+1/2};q\right)_{\xi^x} \ \ \ \ \ \ \ \ \ \ \ \ \ \   \nonumber
\end{eqnarray} 
is a space reversed self-duality function for $q$-Hahn Boson process where $\boldsymbol\eta$ evolves with total asymmetry to the left and $\boldsymbol\xi$ evolves with total asymmetry to the right.
\end{theorem}

\begin{proof}

Let 
\begin{equation}
    D_x(\boldsymbol\eta,\boldsymbol\xi)=\left(q^{-\left(N_{x}^-(\boldsymbol\xi)+N_{x+1}^+(\boldsymbol\eta)\right)+1/2};q\right)_{\xi^x}.  
\end{equation}

Recall that  because the process is a nearest neighbour zero--range process the generator of $\boldsymbol\eta$ can be written as $\mathcal{L}=\sum_{x=1}^{L-1}\mathcal{L}_{x+1}$, generator of $\boldsymbol\xi$ can be written as $\tilde{\mathcal{L}}=\sum_{x=1}^{L-1}\tilde{\mathcal{L}}_x$, where $\mathcal{L}_x$ are the two--site generators. 

Since $D_x(\boldsymbol\eta,\boldsymbol\xi)$ involves counting the number of particles in $\boldsymbol\eta$ at sites to the right of $x$ (exclusive), if $\mathcal{L}(\boldsymbol\eta,\boldsymbol\sigma)\neq 0$ then $D_y(\boldsymbol\sigma,\boldsymbol\xi)=D_y(\boldsymbol\eta,\boldsymbol\xi)$ for $y\neq x$, where $\boldsymbol\sigma:=\boldsymbol\eta+\gamma^{(x)}-\gamma^{(x+1)}$, for  $\gamma\in \mathbb{Z}$ and $\gamma^{(x)}$ denoting $\gamma$ particles at site $x$. In other words,  $\boldsymbol\sigma$ is obtained from $\eta$ by moving $\gamma$ particles from $x$ to $x+1$.
\begin{equation*}
    D(\boldsymbol\sigma,\boldsymbol\xi)=D(\boldsymbol\eta,\boldsymbol\xi)\frac{\left(q^{-\left(N_{x}^-(\boldsymbol\xi)+N_{x+1}^+(\boldsymbol\eta)\right)+\gamma+1/2};q\right)_{\xi^x}}{\left(q^{-\left(N_{x}^-(\boldsymbol\xi)+N_{x+1}^+(\boldsymbol\eta)\right)+1/2};q\right)_{\xi^x}}.
\end{equation*}
Let $\tilde{\boldsymbol\sigma}=\boldsymbol\xi+\gamma^{(x+1)}-\gamma^{(x)}$,
\begin{align*}
D(\boldsymbol\eta,\tilde{\boldsymbol\sigma})& =D(\boldsymbol\eta,\boldsymbol\xi)\frac{\left(q^{-\left(N_{x}^-(\boldsymbol\xi)+N_{x+1}^+(\boldsymbol\eta)\right)+1/2+\gamma};q\right)_{\xi^x-\gamma}}{\left(q^{-\left(N_{x}^-(\boldsymbol\xi)+N_{x+1}^+(\boldsymbol\eta)\right)+1/2};q\right)_{\xi^x}}\\   
&\quad\times\frac{\left(q^{-\left(N_{x+1}^-(\boldsymbol\xi)+N_{x+2}^+(\boldsymbol\eta)\right)+1/2};q\right)_{\xi^{x+1}+\gamma}}{\left(q^{-\left(N_{x+1}^-(\boldsymbol\xi)+N_{x+2}^+(\boldsymbol\eta)\right)+1/2};q\right)_{\xi^{x+1}}}
\\&=D(\boldsymbol\eta,\boldsymbol\xi)\frac{\left(q^{-\left(N_{x}^-(\boldsymbol\xi)+N_{x+2}^+(\boldsymbol\eta)\right)+1/2};q\right)_\gamma}{\left(q^{-\left(N_{x}^-(\boldsymbol\xi)+N_{x+1}^+(\boldsymbol\eta)\right)+1/2};q\right)_\gamma}\\    &=D(\boldsymbol\eta,\boldsymbol\xi)\frac{\left(q^{-\left(N_{x}^-(\boldsymbol\xi)+N_{x+1}^+(\boldsymbol\eta)\right)+1/2+\gamma};q\right)_{\eta^{x+1}}}{\left(q^{-\left(N_{x}^-(\boldsymbol\xi)+N_{x+1}^+(\boldsymbol\eta)\right)+1/2};q\right)_{\eta^{x+1}}},
\end{align*}

\begin{align*}
\mathcal{L}&D(\boldsymbol\eta,\boldsymbol\xi) \\= & D(\boldsymbol\eta,\boldsymbol\xi)\sum_{x=1}^{L-1}\sum_{\gamma}\mathcal{L}_{x+1}(\boldsymbol\eta,\boldsymbol\eta+\gamma^{(x)}-\gamma^{(x+1)})\frac{\left(q^{-\left(N_{x}^-(\boldsymbol\xi)+N_{x+1}^+(\boldsymbol\eta)\right)+\gamma+1/2};q\right)_{\xi^x}}{\left(q^{-\left(N_{x}^-(\boldsymbol\xi)+N_{x+1}^+(\boldsymbol\eta)\right)+1/2};q\right)_{\xi^x}} 
\end{align*}
and
\begin{align*}
    & D\tilde{\mathcal{L}}^*(\boldsymbol\eta,\boldsymbol\xi) \\ & =D(\boldsymbol\eta,\boldsymbol\xi)\sum_{x=1}^{L-1}\sum_{\gamma}\mathcal{\tilde{L}}_{x}(\boldsymbol\xi,\boldsymbol\xi+\gamma^{(x+1)}-\gamma^{(x)})\frac{\left(q^{-\left(N_{x}^-(\boldsymbol\xi)+N_{x+1}^+(\boldsymbol\eta)\right)+1/2+\gamma};q\right)_{\eta^{x+1}}}{\left(q^{-\left(N_{x}^-(\boldsymbol\xi)+N_{x+1}^+(\boldsymbol\eta)\right)+1/2};q\right)_{\eta^{x+1}}}.
\end{align*}
Therefore, it suffices to show 
\begin{multline*}
   \sum_{0\le \gamma\le \eta^{x+1}}\Phi'(\gamma\mid\eta^{x+1})\frac{\left(q^{-\left(N_{x}^-(\xi)+N_{x+1}^+(\eta)\right)+\gamma+1/2};q\right)_{\xi^x}}{\left(q^{-\left(N_{x}^-(\boldsymbol\xi)+N_{x+1}^+(\boldsymbol\eta)\right)+1/2};q\right)_{\xi^x}} =\\
    \sum_{0\le \gamma\le \xi^{x}}\Phi'(\gamma\mid\xi^{x})\frac{\left(q^{-\left(N_{x}^-(\boldsymbol\xi)+N_{x+1}^+(\boldsymbol\eta)\right)+1/2+\gamma};q\right)_{\eta^{x+1}}}{\left(q^{-\left(N_{x}^-(\boldsymbol\xi)+N_{x+1}^+(\boldsymbol\eta)\right)+1/2};q\right)_{\eta^{x+1}}},
\end{multline*}
which is just Corollary \ref{Cor:1}.

Next turn to the discrete-time $q$-Hahn Boson process.
If the evolution is to the left, then the transition probabilities are
\begin{equation*}
    P(\boldsymbol\eta,\boldsymbol\zeta)=\prod_{x=2}^L\Phi(\gamma^x\mid\eta^x),
\end{equation*}
where $\zeta^x=\eta^x-\gamma^x+\gamma^{x+1}$, $1\le x\le L$ and $\gamma^1=\gamma^{L+1}=0$. In words, this denotes  that $\gamma^x$ particles have left lattice site $x$ in particle configuration $\boldsymbol\eta$.
For each $x$,
\begin{equation*}
   D_x (\boldsymbol\zeta,\boldsymbol\xi)=D_x(\boldsymbol\eta,\boldsymbol\xi)\frac{\left(q^{-\left(N_{x}^-(\boldsymbol\xi)+N_{x+1}^+(\boldsymbol\eta)\right)+1/2+\gamma^{x+1}};q\right)_{\xi^x}}{\left(q^{-\left(N_{x}^-(\boldsymbol\xi)+N_{x+1}^+(\boldsymbol\eta)\right)+1/2};q\right)_{\xi^x}}.
\end{equation*}
Let $\mathfrak{X}$ denote the set of particle configurations at one site and  $\mathfrak{X}^L$ the Cartesian product.
Thus, after re-indexing,
\begin{align*}        \sum_{\boldsymbol\zeta\in\mathfrak{X}^L}&P(\boldsymbol\eta,\boldsymbol\zeta) D(\boldsymbol\zeta,\boldsymbol\xi) \\ & =D(\boldsymbol\eta,\boldsymbol\xi)\sum_{\boldsymbol\gamma\in \{0\}\times \mathfrak{X}^{L-1} }\prod_{x=2}^L\Phi(\gamma^x\mid\eta^x)\frac{\left(q^{-\left(N_{x-1}^-(\boldsymbol\xi)+N_{x}^+(\boldsymbol\eta)\right)+1/2+\gamma^{x}};q\right)_{\xi^{x-1}}}{\left(q^{-\left(N_{x-1}^-(\boldsymbol\xi)+N_{x}^+(\boldsymbol\eta)\right)+1/2};q\right)_{\xi^{x-1}}}.
\end{align*}

If the evolution is to the right, 
\begin{equation*}
    P(\boldsymbol\xi,\boldsymbol\zeta)=\prod_{x=1}^{L-1}\Phi(\gamma^x\mid\xi^x),
\end{equation*}
where $\zeta^x=\xi^x-\gamma^x+\gamma^{x-1}$, $1\le x\le L$ and $\gamma^0=\gamma^{L}=0$.
For each $x$,
\begin{equation*}
\begin{array}{cl}
    D (\boldsymbol\eta,\boldsymbol\zeta) &   =D(\boldsymbol\eta,\boldsymbol\xi)\prod_{x=1}^L\frac{\left(q^{-\left(N_{x}^-(\boldsymbol\xi)+N_{x+1}^+(\boldsymbol\eta)\right)+1/2+\gamma^x};q\right)_{\xi^{x}-\gamma^x+\gamma^{x-1}}}{\left(q^{-\left(N_{x}^-(\boldsymbol\xi)+N_{x+1}^+(\boldsymbol\eta)\right)+1/2};q\right)_{\xi^{x}}}\\ 
     &  =D(\boldsymbol\eta,\boldsymbol\xi)\prod_{x=1}^L\frac{\left(q^{-\left(N_{x}^-(\boldsymbol\xi)+N_{x+1}^+(\boldsymbol\eta)\right)+1/2+\xi^x};q\right)_{\gamma^{x-1}}}{\left(q^{-\left(N_{x}^-(\boldsymbol\xi)+N_{x+1}^+(\boldsymbol\eta)\right)+1/2};q\right)_{\gamma^x}}\\
     & =D(\boldsymbol\eta,\boldsymbol\xi)\prod_{x=1}^L\frac{\left(q^{-\left(N_{x}^-(\boldsymbol\xi)+N_{x+2}^+(\boldsymbol\eta)\right)+1/2};q\right)_{\gamma^{x}}}{\left(q^{-\left(N_{x}^-(\boldsymbol\xi)+N_{x+1}^+(\boldsymbol\eta)\right)+1/2};q\right)_{\gamma^x}}\\
     & 
     =D(\boldsymbol\eta,\boldsymbol\xi)\prod_{x=1}^L\frac{\left(q^{-\left(N_{x}^-(\boldsymbol\xi)+N_{x+1}^+(\boldsymbol\eta)\right)+1/2+\gamma^x};q\right)_{\eta^{x+1}}}{\left(q^{-\left(N_{x}^-(\boldsymbol\xi)+N_{x+1}^+(\boldsymbol\eta)\right)+1/2};q\right)_{\eta^{x+1}}}.
\end{array}
\end{equation*}
Thus,
\begin{multline*}
     \sum_{\boldsymbol\zeta\in\mathfrak{X}^L}D(\boldsymbol\eta,\boldsymbol\zeta)P(\boldsymbol\xi,\boldsymbol\zeta)=D(\boldsymbol\eta,\boldsymbol\xi)\\\times\sum_{\gamma\in  \mathfrak{X}^{L-1}\times\{0\} }\prod_{x=2}^{L}\Phi(\gamma^{x-1}\mid\xi^{x-1})\frac{\left(q^{-\left(N_{x-1}^-(\boldsymbol\xi)+N_{x}^+(\boldsymbol\eta)\right)+1/2+\gamma^{x-1}};q\right)_{\eta^{x}}}{\left(q^{-\left(N_{x-1}^-(\boldsymbol\xi)+N_{x}^+(\boldsymbol\eta)\right)+1/2};q\right)_{\eta^{x}}}.
\end{multline*}
Therefore, it suffices to show that
\begin{multline*}
    \sum_{\gamma^x\in  \mathfrak{X} }\Phi(\gamma^x\mid\eta^x) \frac{\left(q^{-\left(N_{x-1}^-(\boldsymbol\xi)+N_{x}^+(\boldsymbol\eta)\right)+1/2+\gamma^{x}};q\right)_{\xi^{x-1}}}{\left(q^{-\left(N_{x-1}^-(\boldsymbol\xi)+N_{x}^+(\boldsymbol\eta)\right)+1/2};q\right)_{\xi^{x-1}}} \\=\sum_{\gamma^{x-1}\in  \mathfrak{X} }\Phi(\gamma^{x-1}\mid\xi^{x-1})\frac{\left(q^{-\left(N_{x-1}^-(\boldsymbol\xi)+N_{x}^+(\boldsymbol\eta)\right)+1/2+\gamma^{x-1}};q\right)_{\eta^{x}}}{\left(q^{-\left(N_{x-1}^-(\boldsymbol\xi)+N_{x}^+(\boldsymbol\eta)\right)+1/2};q\right)_{\eta^{x}}},
\end{multline*}
which is just Lemma \ref{lemma:1}.
\end{proof}

Now we prove similar results for the multi--species case. 
\begin{theorem} 
Recall the function $h$ defined in \eqref{eq:2}:
\begin{equation*}
h(\boldsymbol\xi,\boldsymbol\eta)=\sum_{x \in \Lambda_L}\sum_{i=0}^{n-1}\left(-\xi_i^xN_{x+1}^+(\boldsymbol\eta_{[0,n-2-i]})+\eta_i^xN_{x}^+(\boldsymbol\xi_{[0,n-2-i]})\right).
\end{equation*}
Then
\begin{equation*}
\begin{array}{cl}
    D(\boldsymbol\eta,\boldsymbol\xi)   &   =q^{\frac{h(\boldsymbol\xi,\boldsymbol\eta)}{2}} \prod_{i=0}^{n-1}\prod_{x=1}^L{}_1\phi_0\left(q^{-\xi_i^x};q,q^{-\left(N_{x-1}^-(\boldsymbol\xi_i)+N_{x+1}^+(T(\boldsymbol\eta)_{i+1})\right)+1/2}\right)\\
     &   =q^{\frac{h(\boldsymbol\xi,\boldsymbol\eta)}{2}} \prod_{i=0}^{n-1}\prod_{x=1}^L\left(q^{-\left(N_{x}^-(\boldsymbol\xi_i)+N_{x+1}^+(T(\boldsymbol\eta)_{i+1})\right)+1/2} ;q\right)_{\xi_i^x}
     \\
     & =q^{\frac{h(\boldsymbol\xi,\boldsymbol\eta)}{2}} \prod_{i=0}^{n-1}\prod_{x=1}^L\left(q^{-\left(N_{x}^-(\boldsymbol\xi_i)+N_{x+1}^+(\boldsymbol\eta_{n-1-i})\right)+1/2} ;q\right)_{\xi_i^x}
\end{array}
\end{equation*} is a space reversed self-duality function for $n$-species $q$-Hahn Boson process where $\boldsymbol\eta$ evolves with total asymmetry to the left and $\boldsymbol\xi$ evolves with total asymmetry to the right.
\end{theorem}

\begin{proof}

Recall that the generator of $\boldsymbol\eta$ can be written as $\mathcal{L}=\sum_{x=1}^{L-1}\mathcal{L}_{x+1}$, generator of $\boldsymbol\xi$ can be written as $\tilde{\mathcal{L}}=\sum_{x=1}^{L-1}\tilde{\mathcal{L}}_x$.

 Let $\boldsymbol\sigma=\boldsymbol\eta+\boldsymbol\gamma^{(x)}-\boldsymbol\gamma^{(x+1)}$ for some $\boldsymbol\gamma\in \mathbb{Z}^{n}$,

\begin{multline*} D(\boldsymbol\sigma,\boldsymbol\xi)=D(\boldsymbol\eta,\boldsymbol\xi)q^{\sum_{i=0}^{n-1}\gamma_i\xi_{[0,n-2-i]}^{x}}\\
\times\prod_{i=0}^{n-1}\frac{\left(q^{-\left(N_{x}^-(\boldsymbol\xi_i)+N_{x+1}^+(\boldsymbol\eta_{n-1-i})\right)+\gamma_{n-1-i}+1/2};q\right)_{\xi_i^x}}{\left(q^{-\left(N_{x}^-(\boldsymbol\xi_i)+N_{x+1}^+(\boldsymbol\eta_{n-1-i})\right)+1/2};q\right)_{\xi_i^x}}.
\end{multline*}

Let $\tilde{\boldsymbol\sigma}=\boldsymbol\xi+\boldsymbol\gamma^{(x+1)}-\boldsymbol\gamma^{(x)}$,
\begin{equation*}
  D(\boldsymbol\eta,\tilde{\boldsymbol\sigma})
  =D(\boldsymbol\eta,\boldsymbol\xi)q^{\sum_{i=1}^{n-1}\gamma_i\mid\eta_{[0,n-2-i]}^{x+1}\mid}\prod_{i=0}^{n-1}\frac{\left(q^{-\left(N_{x}^-(\boldsymbol\xi_i)+N_{x+2}^+(\boldsymbol\eta_{n-1-i})\right)+1/2};q\right)_{\gamma_i}}{\left(q^{-\left(N_{x}^-(\boldsymbol\xi_i)+N_{x+1}^+(\boldsymbol\eta_{n-1-i})\right)+1/2};q\right)_{\gamma_i}}.\\
\end{equation*}

\begin{equation*}
    \mathcal{L}D(\boldsymbol\eta,\boldsymbol\xi)=D(\boldsymbol\eta,\boldsymbol\xi)\sum_{x=1}^{L-1}\sum_{\boldsymbol\gamma}\mathcal{L}_{x+1}(\boldsymbol\eta,\boldsymbol\eta+\boldsymbol\gamma^{(x)}-\boldsymbol\gamma^{(x+1)})\frac{D(\boldsymbol\eta+\boldsymbol\gamma^{(x)}-\boldsymbol\gamma^{(x+1)},\boldsymbol\xi)}{D(\boldsymbol\eta,\boldsymbol\xi)}.
\end{equation*}

\begin{equation*}
    D\tilde{\mathcal{L}}^*(\boldsymbol\eta,\boldsymbol\xi)=D(\boldsymbol\eta,\boldsymbol\xi)\sum_{x=1}^{L-1}\sum_{\boldsymbol\gamma}\mathcal{\tilde{L}}_{x}(\boldsymbol\xi,\boldsymbol\xi+\boldsymbol\gamma^{(x+1)}-\boldsymbol\gamma^{(x)})\frac{D(\boldsymbol\eta,\boldsymbol\xi+\boldsymbol\gamma^{(x+1)}-\boldsymbol\gamma^{(x)})}{D(\boldsymbol\eta,\boldsymbol\xi)}.
\end{equation*}
Therefore, it suffices to show 
\begin{align*}
 &  \sum_{ \boldsymbol\gamma\le \boldsymbol\eta^{x+1}}\Phi'(\boldsymbol\gamma\mid\boldsymbol\eta^{x+1})q^{\sum_{i=0}^{n-1}\gamma_i\xi_{[0,n-2-i]}^{x}}\\
 &\quad\times \prod_{i=0}^{n-1}\frac{\left(q^{-\left(N_{x}^-(\boldsymbol\xi_i)+N_{x+1}^+(\boldsymbol\eta_{n-1-i})\right)+\boldsymbol\gamma_{n-1-i}+1/2};q\right)_{\xi_i^x}}{\left(q^{-\left(N_{x}^-(\boldsymbol\xi_i)+N_{x+1}^+(\boldsymbol\eta_{n-1-i})\right)+1/2};q\right)_{\xi_i^x}}\\
 &  = \sum_{\boldsymbol\gamma\le \boldsymbol\xi^{x}}\Phi'(\boldsymbol\gamma\mid\boldsymbol\xi^{x})q^{\sum_{i=0}^{n-1}\boldsymbol\gamma_i\eta_{[0,n-2-i]}^{x+1}}\prod_{i=0}^{n-1}\frac{\left(q^{-\left(N_{x}^-(\boldsymbol\xi_i)+N_{x+2}^+(\boldsymbol\eta_{n-1-i})\right)+1/2};q\right)_{\gamma_i}}{\left(q^{-\left(N_{x}^-(\boldsymbol\xi_i)+N_{x+1}^+(\boldsymbol\eta_{n-1-i})\right)+1/2};q\right)_{\gamma_i}} ,
\end{align*}
which is just Proposition \ref{prop:1}.

Next turn to the discrete-time $q$-Hahn Boson process.
If the evolution is to the left, then the transition probabilities are
\begin{equation*}
    P(\boldsymbol\eta,\boldsymbol\zeta)=\prod_{x=2}^L\Phi(\boldsymbol\gamma^x\mid\boldsymbol\eta^x),
\end{equation*}
where $\boldsymbol\zeta^x=\boldsymbol\eta^x-\boldsymbol\gamma^x+\boldsymbol\gamma^{x+1}$, $1\le x\le L$ and $\boldsymbol\gamma^1=\boldsymbol\gamma^{L+1}=(0,\ldots,0)$.

\begin{equation*}
  \frac{D (\boldsymbol\zeta,\boldsymbol\xi)}{D(\boldsymbol\eta,\boldsymbol\xi)} =\prod_{x=1}^L\prod_{i=0}^{n-1}q^{\gamma_i^{x+1}\xi_{[0,n-2-i]}^{x}}\frac{\left(q^{-\left(N_{x}^-(\boldsymbol\xi_i)+N_{x+1}^+(\boldsymbol\eta_{n-1-i})\right)+1/2+\gamma_{n-1-i}^{x+1}};q\right)_{\xi_i^x}}{\left(q^{-\left(N_{x}^-(\boldsymbol\xi_i)+N_{x+1}^+(\boldsymbol\eta_{n-1-i})\right)+1/2};q\right)_{\xi_i^x}}.
\end{equation*}

Thus, after re-indexing,
\begin{multline*}
      \sum_{\boldsymbol\zeta\in\mathfrak{X}^L}P(\boldsymbol\eta,\boldsymbol\zeta)D(\boldsymbol\zeta,\boldsymbol\xi)= D(\boldsymbol\eta,\boldsymbol\xi)
       \sum_{\boldsymbol\gamma\in \{0\}\times \mathfrak{X}^{L-1} }\prod_{x=2}^L\Phi(\boldsymbol\gamma^x\mid\boldsymbol\eta^x)\\\times\prod_{i=0}^{n-1}q^{\gamma_i^{x}\xi_{[0,n-2-i]}^{x-1}}\frac{\left(q^{-\left(N_{x-1}^-(\boldsymbol\xi_i)+N_{x}^+(\boldsymbol\eta_{n-1-i})\right)+1/2+\gamma_{n-1-i}^{x}};q\right)_{\xi_i^{x-1}}}{\left(q^{-\left(N_{x-1}^-(\boldsymbol\xi_i)+N_{x}^+(\boldsymbol\eta_{n-1-i})\right)+1/2};q\right)_{\xi_i^{x-1}}}.
\end{multline*}

If the evolution is to the right, 
\begin{equation*}
    P(\boldsymbol\xi,\boldsymbol\zeta)=\prod_{x=1}^{L-1}\Phi(\boldsymbol\gamma^x\mid\boldsymbol\xi^x),
\end{equation*}
where $\boldsymbol\zeta^x=\boldsymbol\xi^x-\boldsymbol\gamma^x+\boldsymbol\gamma^{x-1}$, $1\le x\le L$ and $\gamma^0=\gamma^{L}=(0,\ldots,0)$.

\begin{equation*}
\begin{array}{cl}
    &\frac{D (\boldsymbol\eta,\boldsymbol\zeta)}{D(\boldsymbol\eta,\boldsymbol\xi)}   \\ & =\prod_{x=1}^L\prod_{i=0}^{n-1}q^{\gamma_i^{x-1}\eta_{[0,n-2-i]}^{x}}\frac{\left(q^{-\left(N_{x}^-(\boldsymbol\xi_i)+N_{x+1}^+(\boldsymbol\eta_{n-1-i})\right)+1/2+\gamma^x_i};q\right)_{\xi_i^{x}-\gamma_i^x+\gamma_i^{x-1}}}{\left(q^{-\left(N_{x}^-(\boldsymbol\xi_i)+N_{x+1}^+(\boldsymbol\eta_{n-1-i})\right)+1/2};q\right)_{\xi_i^{x}}}\\ 
     & =\prod_{x=1}^L\prod_{i=0}^{n-1}q^{\gamma_i^{x-1}\eta_{[0,n-2-i]}^{x}}\frac{\left(q^{-\left(N_{x}^-(\boldsymbol\xi_i)+N_{x+1}^+(\boldsymbol\eta_{n-1-i})\right)+1/2+\gamma_i^x};q\right)_{\eta_{n-1-i}^{x+1}}}{\left(q^{-\left(N_{x}^-(\boldsymbol\xi_i)+N_{x+1}^+(\boldsymbol\eta_{n-1-i})\right)+1/2};q\right)_{\eta_{n-1-i}^{x+1}}}.
\end{array}
\end{equation*}
Thus,
\begin{multline*}
    \sum_{\boldsymbol\zeta\in\mathfrak{X}^L}D(\boldsymbol\eta,\boldsymbol\zeta)P(\boldsymbol\xi,\boldsymbol\zeta)=D(\boldsymbol\eta,\boldsymbol\xi)\sum_{\boldsymbol\gamma\in  \mathfrak{X}^{L-1}\times\{0\} } \prod_{x=2}^{L}
   \Phi(\boldsymbol\gamma^{x-1}\mid\boldsymbol\xi^{x-1})\\\times\prod_{i=0}^{n-1}q^{\gamma_i^{x-1}\eta_{[0,n-2-i]}^{x}}\frac{\left(q^{-\left(N_{x-1}^-(\boldsymbol\xi_i)+N_{x}^+(\boldsymbol\eta_{n-1-i})\right)+1/2+\gamma_i^{x-1}};q\right)_{\eta_{n-1-i}^{x}}}{\left(q^{-\left(N_{x-1}^-(\boldsymbol\xi_i)+N_{x}^+(\boldsymbol\eta_{n-1-i})\right)+1/2};q\right)_{\eta_{n-1-i}^{x}}}  .
\end{multline*}
Therefore, it suffices to show that
\begin{align*}
   & \sum_{\boldsymbol\gamma^x\in  \mathfrak{X} }\Phi(\boldsymbol\gamma^x\mid\boldsymbol\eta^x) \prod_{i=0}^{n-1}q^{\gamma_i^{x}\xi_{[0,n-2-i]}^{x-1}}\\
    &\quad\times\frac{\left(q^{-\left(N_{x-1}^-(\boldsymbol\xi_i)+N_{x}^+(\boldsymbol\eta_{n-1-i})\right)+1/2+\gamma_{n-1-i}^{x}};q\right)_{\xi_i^{x-1}}}{\left(q^{-\left(N_{x-1}^-(\boldsymbol\xi_i)+N_{x}^+(\boldsymbol\eta_{n-1-i})\right)+1/2};q\right)_{\xi_i^{x-1}}}\\
    & =\sum_{\boldsymbol\gamma^{x-1}\in  \mathfrak{X} }\Phi(\boldsymbol\gamma^{x-1}\mid\boldsymbol\xi^{x-1})\prod_{i=0}^{n-1}q^{\gamma_i^{x-1}\eta_{[0,n-2-i]}^{x}} \\ &\quad\times\frac{\left(q^{-\left(N_{x-1}^-(\boldsymbol\xi_i)+N_{x}^+(\boldsymbol\eta_{n-1-i})\right)+1/2+\gamma_i^{x-1}};q\right)_{\eta_{n-1-i}^{x}}}{\left(q^{-\left(N_{x-1}^-(\boldsymbol\xi_i)+N_{x}^+(\boldsymbol\eta_{n-1-i})\right)+1/2};q\right)_{\eta_{n-1-i}^{x}}} ,
\end{align*}
which is just Proposition \ref{prop:1}.
\end{proof}
Thus concludes the proof of Theorem \ref{qHahn}.

\subsection{Proof of Theorem \ref{Application}}
Suppose $\boldsymbol\xi$ has $n_1$ species 0 particles at site $x_1$,  $n_2$ species 1 particles at site $x_2$. The dual process $\boldsymbol\eta$ has $1$ species 0 particles at site $y_1$,  $1$ species 1 particles at site $y_2$, where $\boldsymbol\eta$ evolves with total asymmetry to the left and $\boldsymbol\xi$ evolves with total asymmetry to the right. 
With these special cases,
\begin{align*}
      D^{q-TARZP} (\boldsymbol\xi,\boldsymbol\eta) = &    q^{-\sum_x\xi_0^xN_{x+1}^+(\eta_0)+\sum_x\eta_0^xN_{x}^+(\xi_0)}\\  &\times\prod_{x \in \Lambda_L }{}_1\phi_0\left(q^{-2\xi_0^x};q^2,q^{-2\left(N_{x-1}^-(\xi_0)+N_{x+1}^+(\eta_1)\right)+1}\right)\\
   & \times\prod_{x \in \Lambda_L }{}_1\phi_0\left(q^{-2\xi_1^x};q^2,q^{-2\left(N_{x-1}^-(\xi_1)+N_{x+1}^+(\eta_0)\right)+1}\right)
\end{align*}

now becomes
\begin{align*}
     & D^{q-TARZP}(\boldsymbol\xi,\boldsymbol\eta)\\
   &=q^{-n_1\delta_{y_1>x_1}+n_1\delta_{y_1\le x_1}}\left(q^{-2n_1-2\delta_{y_2>x_1}+1};q^2\right)_{n_1} \left(q^{-2n_2-2\delta_{y_1>x_2}+1};q^2\right)_{n_2}\\
 &  =\left\{\begin{array}{cc}
    q^{-n_1}\left(q^{-2n_1-1};q^2\right)_{n_1} \left(q^{-2n_2-1};q^2\right)_{n_2}    &  y_1>x_1,y_2>x_1\\
        q^{-n_1}\left(q^{-2n_1+1};q^2\right)_{n_1} \left(q^{-2n_2-1};q^2\right)_{n_2}    &  y_1>x_1,y_2\le x_1  \\
       q^{n_1}\left(q^{-2n_1-1};q^2\right)_{n_1} \left(q^{-2n_2-1};q^2\right)_{n_2}    &  x_2<y_1\le x_1,y_2>x_1\\
         q^{n_1}\left(q^{-2n_1+1};q^2\right)_{n_1} \left(q^{-2n_2-1};q^2\right)_{n_2}    &  x_2<y_1\le x_1,y_2\le x_1\\
      q^{n_1}\left(q^{-2n_1-1};q^2\right)_{n_1} \left(q^{-2n_2+1};q^2\right)_{n_2}    & y_1\le x_2,y_2>x_1\\
        q^{n_1}\left(q^{-2n_1+1};q^2\right)_{n_1} \left(q^{-2n_2+1};q^2\right)_{n_2}    & y_1\le x_2,y_2\le x_1\\
   \end{array}\right.
\end{align*}
when $x_1>x_2$, and
if    $x_1<x_2$
\begin{align*}   
D&^{q-TARZP}(\xi, \eta)  \\  & = \left\{\begin{array}{cc}  
   q^{n_1}\left(q^{-2n_1-1};q^2\right)_{n_1} \left(q^{-2n_2+1};q^2\right)_{n_2} :=\delta_1   &  y_1\le x_1,y_2>x_1\\
        q^{n_1}\left(q^{-2n_1+1};q^2\right)_{n_1} \left(q^{-2n_2+1};q^2\right)_{n_2}  :=\delta_2  &  y_1\le x_1,y_2\le x_1  \\
       q^{-n_1}\left(q^{-2n_1-1};q^2\right)_{n_1} \left(q^{-2n_2+1};q^2\right)_{n_2} :=\delta_3   &  x_1<y_1\le x_2,y_2>x_1\\
         q^{-n_1}\left(q^{-2n_1+1};q^2\right)_{n_1} \left(q^{-2n_2+1};q^2\right)_{n_2}  :=\delta_4  &  x_1<y_1\le x_2,y_2\le x_1\\
      q^{-n_1}\left(q^{-2n_1-1};q^2\right)_{n_1} \left(q^{-2n_2-1};q^2\right)_{n_2}  :=\delta_5  & y_1> x_2,y_2>x_1\\
        q^{-n_1}\left(q^{-2n_1+1};q^2\right)_{n_1} \left(q^{-2n_2-1};q^2\right)_{n_2}  :=\delta_6  & y_1> x_2,y_2\le x_1\\
   \end{array}\right.
  \end{align*}

We will focus on the $x_1<x_2$ case. Corresponding to these six cases, consider the probabilities
\begin{align*}
\mathbb{P}_{(y_1,y_2)}\left(  y_1(t) \leq x_1,y_2(t)>x_1 \right) &= p_1, \\
\mathbb{P}_{(y_1,y_2)}\left( y_1(t) \leq x_1,y_2(t)\le x_1 \right) &=p_2, \\
\mathbb{P}_{(y_1,y_2)}\left( x_1<y_1(t)\le x_2,y_2(t)>x_1  \right) &= p_3,\\
\mathbb{P}_{(y_1,y_2)}\left( x_1<y_1(t)\le x_2,y_2(t)\le x_1 \right) &= p_4,\\
\mathbb{P}_{(y_1,y_2)}\left( y_1(t)> x_2,y_2(t)>x_1 \right) &=p_5, \\
\mathbb{P}_{(y_1,y_2)}\left( y_1(t)> x_2,y_2(t)\le x_1 \right) &= p_6 .
\end{align*}
These six values can be solved by setting the six equations:
\begin{itemize}

\item 
 The value of $q_1:=p_1+p_2 = \mathbb{P}_{y_1}( y_1(t) \leq x_1)$ is just the usual random walk consisting of a single particle.
 
 \item
 The value of $q_2:=p_3+p_4= \mathbb{P}_{y_1}( x_1< y_1(t) \leq x_2)$ is just the usual random walk consisting of a single particle.
 
 \item
 The value of $q_3:=p_5+p_6 = \mathbb{P}_{y_1}( y_1(t) > x_2)$  is just the usual random walk consisting of a single particle.
 
 \item
The value of $q_4:=p_2$ had already been found in Theorem 3.3 of \cite{Kor-Lee} (one just needs to switch the words ``left'' and ``right'' to match the notation); this was then further generalized in Theorem 1.2 of \cite{Lee-Wang-2017}. 

 \item
If $y_1 \leq y_2$, then the value of $q_5:=p_5$ can be found by \cite{Kua21}. It is this value that requires $x_1<x_2$.

\item
The value of 
\begin{multline*}
q_6:=p_4+p_6 = \mathbb{P}_{(y_1,y_2)}(x_1<y_1(t), y_2(t) \leq x_1) \\
= \mathbb{P}_{y_1}(x_1<y_1(t)) - \mathbb{P}_{(y_1,y_2)}(x_1<y_1(t), x_1 < y_2(t) )
\end{multline*}
can be found similarly with the color--blind projection (the second probability is a special case of \cite{Kua21}). It also follows from duality applied to \cite{Lee-Wang-2017}.
 
\end{itemize}

Now note that the  $6\times 6$ matrix can be written as
$$
\left(\begin{array}{cccccc}1&1&0&0&0&0\\0&0&1&1&0&0\\0&0&0&0&1&1\\0&1&0&0&0&0\\0&0&0&0&1&0\\0&0&0&1&0&1\\\end{array}\right) =  \left(\begin{array}{cccccc}1&0&0&-1&0&0\\0&0&0&1&0&0\\0&1&1&0&-1&-1\\0&0&-1&0&1&1\\0&0&0&0&1&0\\0&0&1&0&-1&0\\\end{array}\right)^{-1}
$$
so that
$$
p_1 = q_1-q_4, \quad p_2=q_4, \quad p_3 = q_2+q_3-q_5-q_6,
$$
$$  \quad p_4 = -q_3 + q_5+q_6, \quad p_5=q_5, \quad p_6=q_3-q_5.$$

The values of $q_1,q_2,q_3$ can be obtained from a random walk with a single particle, while $q_4$ and $q_6$ can be obtained from the single--species two--particle $q$--TAZRP. Only $q_5$, which has the most complicated coefficient that does not factor, requires multi--species analysis. 

\begin{lemma}
The values of $q_1,q_2,q_3$ explicitly equal
\begin{align*}
q_1 &= {\delta_{y_1\le x_1+}\delta_{y_1> x_1} }Q(y_1-x_1,t) , \\
q_3 &=  1 - {\delta_{y_1\ge x_2}Q(y_1-x_2,t)}{-\delta_{y_1<x_2}}, \\
q_2 &= 1- q_2-q_3.
\end{align*}
\end{lemma}

\begin{proof}
Recall that
$$
q_1 := \mathbb{P}_{y_1}( y_1(t) \leq x_1  ). 
$$
Then the  exact formula for $q_1$ is given by the Poisson distribution:
$$
q_1 = e^{-t} \sum_{m=y_1-x_1}^{\infty} \frac{t^m}{m!} = Q(y_1-x_1,t).
$$
Similarly, 
The value of $q_3:=p_5+p_6 = \mathbb{P}_{y_1}( y_1(t) > x_2)$ can be found similarly, and $q_1+q_2+q_3=1$ by definition.

\end{proof}

Now turn to the values of $q_4,q_5$ and $q_6$.
\begin{lemma}
The values of $q_4,q_5,q_6$ are explicitly:

\begin{align*}
q_4&=\frac{ 1}{(2\pi i)^2}\int_{C} \frac{d w_{1}}{w_{1}}  \int_{C} \frac{d w_{2}}{w_{2}} \frac{w_1-w_2}{w_1-qw_2} \prod_{j=1}^{2}\left[(1-w_j)^{-({y_j-x_1}+1)}e^{-w_{j} t}\right] ,\\
q_5&= \frac{ q}{(2\pi i)^2}\int_{\tilde{C}_{1}}  \frac{dw_1}{w_1}  \int_{\tilde{C}_{2}}  \frac{dw_2}{w_2} \frac{w_1-w_2}{w_1-qw_2} \frac{(1-w_1)^{-({y_1-x_2}+1)}(1-w_2)^{-({y_2-x_1}+1)}}{e^{(w_1+w_2)t }} ,\\
q_6 &= 1- q_1 - q\frac{ 1}{(2\pi i)^2}\int_{\tilde{C}_{1}}  \frac{dw_1}{w_1}  \int_{\tilde{C}_{2}}  \frac{dw_2}{w_2} \frac{w_1-w_2}{w_1-qw_2}\prod_{j=1}^2\left[ (1-w_j)^{-({y_j-x_1}+1)} e^{-w_jt  } \right],
\end{align*}
Here, the $C$ are ``large contours'' that are centered at $0$ and contain $1$, and the $\tilde{C}$ are ``small contours'' where $C_2$ contains $1$ and not $0$, whereas $C_1$ contains $qC_2,1$ and not $0$.
\end{lemma}
\begin{proof}
For $q_4:=p_2 = \mathbb{P}_{(y_1,y_2)}\left( y_1(t) \leq x_1,y_2(t)\le x_1 \right)$, take $N=n=2$ in Theorem 2.1 of \cite{Lee-Wang-2017}. This is stated separately as Proposition 2.1 of that paper, which states that for a left--most particle with right drift, 
$$\mathbb{P}_{Y}\left(x_{N}(t)>M\right)=\int_{C} \frac{d w_{1}}{w_{1}} \cdots \int_{C} \frac{d w_{N}}{w_{N}} B_{N}\left(w_{1}, \ldots, w_{N}\right) \prod_{j=1}^{N}\left[\prod_{k=y_{j}}^{x_1}\left(\frac{e^{-w_{j} t}}{1-w_{j}}\right) \right].$$
where the $C$ are ``large contours'' that are centered at $0$ and contain $1$. In the current paper, the drift is in the opposite direction, so one obtains the stated value of $q_4$.
 
For
\begin{multline*}
  q_6:=p_4+p_6 = \mathbb{P}_{(y_1,y_2)}(x_1<y_1(t), y_2(t) \leq x_1) \\
= \mathbb{P}_{y_1}(x_1<y_1(t)) - \mathbb{P}_{(y_1,y_2)}(x_1<y_1(t), x_1 < y_2(t) ),  
\end{multline*}
The first term is just $1-q_1$, while the second term can be obtained from Proposition 2.3 of \cite{Lee-Wang-2017}, which reads 
\begin{multline*}
  \mathbb{P}_{Y}\left(x_{1}(t) \leq M\right)=(-1)^{N} q^{N(N-1) / 2} f_{\tilde{C}_{1}} \frac{d w_{1}}{w_{1}} \cdots f_{\tilde{C}_{N}} \frac{d w_{N}}{w_{N}} B_{N}\left(w_{1}, \ldots, w_{N}\right)  \\ \times\prod_{j=1}^{N}\left[\prod_{k=y_{j}}^{x_1}\left(\frac{1}{1-w_{j}}\right) e^{-w_{j} t}\right].  
\end{multline*}

For $q_5$, recall from  Theorem 2.1 of  \cite{Kua21} that the contour integral that needs to be evaluated in the $N$--species case is
$$
\frac{(-1)^N q^{N(N-1)/2}}{(2\pi i)^N}\int \frac{dw_1}{w_1} \cdots \int \frac{dw_N}{w_N} B(w_1,\ldots,w_N) \prod_{j=1}^N (1-w_j)^{-{M}_j} e^{-w_jt  }   
$$
and plugging in $N=2$ yields the result.

\end{proof}

\section*{Acknowledgments}

 C.F. acknowledges the NSF grant DMS-1928930 which supported her participation at the Mathematical Sciences Research Institute in Berkeley, California, during the Fall 2021 semester. J.K. and Z.Z. would like to acknowledge the Simons Foundation Collaboration Grant 713614, for funding travel to Berkeley to visit C.F.

\section*{Declarations}

\subsection*{Funding}
This work was supported by NSF grant DMS-1928930 and Simons Foundation Collaboration Grant 713614. 
\subsection*{Conflict of interest/Competing interests} 
The authors declare they have no conflicts of interests.

\subsection*{ Ethics approval }
Not applicable.
\subsection*{ Consent to participate} Not applicable
\subsection*{Consent for publication}
The authors all consent.
\subsection*{ Availability of data and materials}
Not applicable.
\subsection*{ Code availability }
Not applicable.
\subsection*{ Authors' contributions}
All authors contribute equally to this work.

\end{document}